\documentclass[11pt]{amsart}

\usepackage{amsmath,amsfonts,amssymb,amsthm}
\usepackage{txfonts}
\usepackage{latexsym,bm,graphicx}
\usepackage{mathrsfs}
\usepackage{color}

\title[Yau's Gradient Estimates]{Yau's Gradient  Estimates  on Alexandrov Spaces}
\author{Hui-Chun Zhang}
\address{Department of Mathematics\\  Sun Yat-sen University\\ Guangzhou 510275\\ E-mail address: zhanghc3@mail.sysu.edu.cn}
\author{Xi-Ping Zhu}
\address{Department of Mathematics\\  Sun Yat-sen University\\ Guangzhou 510275\\ E-mail address: stszxp@mail.sysu.edu.cn}
\newtheorem{thm}{Theorem}[section]
\newtheorem{prop}[thm]{Proposition}
\newtheorem{lem}[thm]{Lemma}

\newtheorem{cor}[thm]{Corollary}

\theoremstyle{definition}
\theoremstyle{remark}

\newtheorem{defn}[thm]{Definition}
\newtheorem{rem}[thm]{Remark}

\numberwithin{equation}{section}

\newcommand{\ls}{\leqslant}
\newcommand{\gs}{\geqslant}
\newcommand{\wa}{\widetilde\angle}

\newcommand{\ip}[2]{\left<{#1},{#2}\right>}
\newcommand{\rv}{{\rm vol}}

\newcommand{\be}[2]{\begin{#1}{#2}\end{#1}}

\newcommand{\R}{\mathbb{R}}
\newcommand{\M}{\mathbb{M}}

\begin{document}

\begin{abstract}In this paper, we establish a  Bochner type formula on Alexandrov spaces with Ricci curvature bounded below.
 Yau's gradient estimate for harmonic functions is also obtained on Alexandrov spaces.
\end{abstract}

\maketitle

\section{Introduction}
 The study of harmonic functions on Riemannian manifolds has been one of the basic topic
in geometric analysis.  Yau in \cite{y75} and Cheng--Yau in \cite{cy75} proved the following well known  gradient estimate for harmonic functions
on smooth manifolds (see also \cite{sy94}).
\begin{thm} \label{thm1.1}{\rm (Yau \cite{y75}, Cheng--Yau \cite{cy75})}\indent
Let $M^n$ be an $n$-dimensional complete noncompact Riemannian manifold with Ricci curvature bounded from below by $-K$, $(K\gs0)$.
 Then there exists a constant $C_n$, depending only on $n$, such that every positive harmonic function $u$ on $M^n$ satisfies
  $$|\nabla \log u|\ls C_n(\sqrt K+\frac{1}{R})$$
   in any ball $B_p(R).$
\end{thm}
 A direct consequence of the gradient estimate is  the Yau's Liouville theorem which states  a
positive harmonic function on a complete Riemannian manifold of nonnegative Ricci
curvature must be constant.

The main purpose of this paper is to extend the Yau's estimate to Alexandrov spaces. Roughly speaking, an Alexandrov space
with curvature bounded below is a length space  $X$  with the   property that  any geodesic triangle in $X$ is ``fatter"
 than the corresponding one in the associated  model space. The seminal paper \cite{bgp92} and the 10th chapter in the
 book \cite{bbi01} provide
introductions to Alexandrov geometry.

Alexandrov spaces (with curvature bounded below) generalize successfully the notion of lower bounds of sectional curvature
 from Riemannian manifolds to metric spaces. In the last few years, several notions for the ¡°Ricci curvature bounded below¡± on general
metric spaces appeared. Sturm \cite{s06-2} and Lott--Villani \cite{lv09,lv07}, independently, introduced a
so called curvature-dimension condition on  metric measure spaces, denoted by $CD(K,n)$.  The curvature-dimension condition
implies a generalized Brunn--Minkowski inequality (hence also Bishop--Gromov  comparison and Bonnet--Myer's theorem) and a
 Poincar\'e inequality (see \cite{s06-2,lv09,lv07}).
 Meanwhile, Sturm  \cite{s06-2} and Ohta  \cite{o07}
introduced a measure contraction property, denoted by $ MCP(K,n)$, which is a slight modification of a property
introduced earlier by Sturm in \cite{s98} and in a similar form by Kuwae and Shioya in \cite{ks01,ks03}. The condition
 $MCP(K,n)$ also implies Bishop--Gromov comparison, Bonnet--Myer's theorem and a Poincar\'e inequality (see \cite{s06-2,o07}).

In the framework of Alexandrov spaces, Kuwae--Shioya in \cite{ks07} introduced an infinitesimal version of the Bishop--Gromov
 comparison condition, denoted by $BG(K,n)$. On an $n-$dimensional Alexandrov space with its Hausdorff measure,
  the condition $BG(K,n)$ is equivalent to $MCP(K,n)$ (see \cite{ks07}). Under the condition $BG(0,n)$, Kuwae--Shioya
 in \cite{ks07} proved a topological splitting theorem of Cheeger--Gromoll type. In \cite{zz10}, the authors introduced
 a notion of ``Ricci curvature has a lower bound  $K$", denoted by $Ric\gs K$,  by averaging the second variation of
 arc-length (see \cite{petu98}). On an $n$-dimensional Alexandrov space $M$, the condition $Ric\gs K$ implies that $M$
 (equipped  its Hausdorff measure) satisfies $CD(K,n)$ and $BG(K,n)$ (see \cite{petu09} and Appendix in \cite{zz10}).
 Therefore, Bishop--Gromov comparison and a  Poincar\'e inequality hold on Alexandrov spaces with Ricci curvature bounded below.
  Furthermore,  under this Ricci curvature condition, the authors  in \cite{zz10} proved an \emph{isometric} splitting theorem
   of Cheeger--Gromoll type and the maximal diameter theorem of Cheng type. Remark that all of these generalized notions of
   Ricci curvature bounded below are equivalent to the classical one on smooth Riemannian manifolds.

 Let $M$ be an Alexandrov space. In \cite{os94}, Ostu--Shioya established a $C^1$-structure and a corresponding $C^0$-Riemannian
  structure on the set of regular points of $M$. Perelman  in \cite{p-dc} extended it to a $DC^1$-structure and a corresponding $BV_{loc}^0$-Riemannian structure. By applying this $DC^1$-structure, Kuwae--Machigashira--Shioya in \cite{kms01} introduced
   a canonical Dirichlet form on $M$. Under a $DC^1$ coordinate system and written the $BV_{loc}^0$-Riemannian metric by
   $(g_{ij}),$ a harmonic functions $u$  is a solution of the equation \begin{equation}\label{eq1.1}
 \sum_{i,j=1}^n\partial_i\big(\sqrt{g}g^{ij}\partial_j u\big)=0
 \end{equation}
 in the sense of distribution,  where $g=\det{(g_{ij})}$ and $(g^{ij})$ is the inverse matrix of $(g_{ij}).$ By adapting
  the standard Nash--Moser iteration argument, one knows that a harmonic function must be locally  H\"older continuous.
   More generally, in a metric space with a doubling measure and a Poincar\'e inequality for upper gradient, the same
    regularity assertion still holds for Cheeger-harmonic functions, (see \cite{c99,kshan01} for the details).

The classical  Bernstein trick   in PDE's implies that any harmonic function on smooth Riemannian manifolds is actually
 locally Lipschitz continuous. In the language of differential geometry, one can use Bochner formula to bound the
  gradient of a harmonic function on smooth manifolds. The well known Bochner formula states that for any $C^3$
   function $u$ on a smooth $n$-dimensional Riemannian manifold, there holds
\begin{equation}\label{eq1.2}
\Delta|\nabla u|^2=2|\nabla^2 u|^2+2\ip{\nabla u}{\nabla \Delta u}+2 Ric(\nabla u,\nabla u ).
\end{equation}
 But for singular spaces (including Alexandrov spaces), one meets serious  difficulty to study the Lipschitz continuity
  of harmonic function. Firstly, due to the lacking of the notion of second order derivatives, the Bernstein trick does
  not work directly on singular spaces. Next one notes the singular set might be dense in an Alexandrov space. When
   one considers the partial differential equation \eqref{eq1.1} on an Alexandrov space, the coefficients $\sqrt{g}g^{ij}$
   might be not well defined and not continuous on a dense subset. It seems that all PDE's approaches fail to give
    the Lipschitz continuity for the (weak) solutions of \eqref{eq1.1}.

 The first result for the Lipschitz continuity of harmonic functions on Alexandrov spaces was announced by Petrunin
 in \cite{petu03}. In \cite{petu96}, Petrunin developed an argument based on the second variation formula of arc-length
  and Hamitlon--Jacobi shift, and sketched a proof to the Lipschitz continuity of harmonic functions on Alexandrov
   spaces with nonnegative curvature, which is announced in \cite{petu03}. In the present paper, a detailed exposition
    of Petrunin's proof is contained in Proposition \ref{prop5.3} below. Furthermore, we will prove the Lipschitz
    continuity of solutions of general Poisson equation, see Corollary \ref{cor5.5} below.
 In \cite{krs03}, Koskela--Rajala--Shanmugalingam  proved that the same regularity of Cheeger-harmonic functions on
 metric measure spaces, which supports an Ahlfors regular measure, a Poincar\'e inequality and a certain heat kernel
  condition. In the same paper, they gave an example to show that, on a general metric metric supported a doubling
  measure  and a Poincar\'e inequality, a harmonic function
might fail to be Lipschitz continuous. In \cite{zz10-2}, based on the Lipschitz continuity of harmonic functions and
  a representation of heat kernel in \cite{kms01}, we proved that every solution of  heat equation on an Alexandrov
  space must be Lipschitz continuous. Independently, in \cite{gko10}, by applying the contraction property of gradient
   flow of the relative entropy in $L^2$--Wasserstein space, Gigli--Kuwada--Ohta also obtained the Lipschitz continuity
    of solutions of heat equation on Alexandrov spaces.

Yau's gradient estimate in the above Theorem \ref{thm1.1} is an improvement of the classical Bernstein gradient estimate.
 To extend Yau's estimates to Alexandrov spaces, let us recall what is  its proof  in smooth case. Consider a positive
  harmonic function $u$ on an $n$-dimensional Riemannian manifold. By  applying  Bochner formula \eqref{eq1.2} to $\log u$,
  one has
 $$\Delta Q\gs \frac{2}{n}Q^2-2\ip{\nabla \log u}{\nabla Q}-2KQ,$$
  where $Q=|\nabla \log u|^2.$    Let $\phi$ be a cut-off
  function. By applying maximum  principle to the smooth function $\phi Q$, one can get the desired gradient estimate
   in Theorem \ref{thm1.1}. In this proof, it is crucial  to exist the positive quadratic term $\frac{2}{n}Q^2$ on
    the RHS of the above inequality.

Now let us consider an $n$-dimensional Alexandrov space $M$ with $Ric\gs -K$. In \cite{gko10}, Gigli--Kuwada--Ohta
 proved a weak form of the $\Gamma_2$-condition
 $$\Delta |\nabla u|^2\gs 2\ip{\nabla u}{\nabla\Delta u}-2K|\nabla u|^2, \quad {\rm for\ all}\quad u\in D(\Delta)\cap W^{1,2}(M).$$
  This is a weak version of Bochner  formula. If we use the formula to $\log u$ for a positive harmonic function $u$, then
   $$\Delta Q\gs -2\ip{\nabla \log u}{\nabla Q}-2KQ,$$
    where $Q=|\nabla \log u|^2$. Unfortunately, this does not suffice  to derive the Yau's estimate because the positive
     term $\frac{2}{n}Q^2$ vanishes. The first result in this paper is the following Bochner type formula which keeps the
      desired positive quadratic term.
\begin{thm}\label{thm1.2}Let $M$ be an $n$-dimensional Alexandrov space with Ricci curvature bounded from below by $-K$,
and $\Omega$ be a bounded domain in $M$.
 Let  $f(x,s):\Omega\times[0,+\infty)\to \R$ be a Lipschitz function  and satisfy the following:\\
\indent (a)\indent  there exists a zero measure set $\mathcal N\subset \Omega$ such that for all $s\gs0$, the functions $f(\cdot,s)$ are differentiable at any $x\in \Omega\backslash \mathcal N;$\\
\indent (b)\indent the function $f(x,\cdot)$ is of class $C^1$ for  all $x\in \Omega$  and the function $\frac{\partial f}{\partial s}(x,s)$ is continuous, non-positive on $\Omega\times [0,+\infty)$.

Suppose that $u$ is  Lipschitz on $\Omega$ and
  $$-\int_{\Omega}\ip{\nabla u}{\nabla \phi}d\rv=\int_{\Omega}\phi\cdot f\big(x,|\nabla u|^2\big)\rv $$
 for all Lipschitz function $\phi$ with compact support in $\Omega$.

Then we have $|\nabla u|^2\in W^{1,2}_{loc}(\Omega) $ and
 \begin{equation*}\begin{split}
 -\int_{\Omega}&\ip{\nabla \varphi}{|\nabla u |^2}d\rv\\
 & \gs 2\int_{\Omega}\varphi\cdot\Big(\frac{f^2(x,|\nabla u|^2)}{n}+\ip{\nabla u}{\nabla f(x,|\nabla u|^2)}-K|\nabla u|^2\Big)d\rv
 \end{split}
\end{equation*}
for all  Lipschitz function $\varphi\gs0$ with compact support in $\Omega$,  provided  $|\nabla u|$ is lower semi-continuous at almost all $x\in \Omega$  $($That is, there exists a representative of $|\nabla u|$, which is lower semi-continuous at almost all $x\in \Omega.)$.
\end{thm}

Instead of the maximum principle argument in the above proof of Theorem \ref{thm1.1}, we will adapt a Nash--Moser
 iteration method to establish the following Yau's gradient estimate, the second result of this paper.
\begin{thm}\label{thm1.3}
Let $M$ be an $n$-dimensional Alexandrov space with Ricci curvature bounded from below by $-K$ $(K\gs0),$ and let
 $\Omega$ be a bounded domain in $M$.  Then there exists a constant $C=C\big(n,\sqrt K{\rm diam}(\Omega)\big)$ such
  that every positive harmonic function $u$ on $\Omega$ satisfies
$$\max_{x\in B_p(\frac{R}{2})}|\nabla \log u|\ls C(\sqrt K+\frac{1}{R})
 $$
for any ball $B_p(R)\subset \Omega.$ If $K=0$, the constant $C$ depends only on $n$.
\end{thm}
 We also obtain  a global version of the above gradient estimate.
\begin{thm}\label{thm1.4}
 Let $M$ be as above and $u$ be a positive harmonic function on $M$. Then we have $$|\nabla\log u|\ls C_{n,K}$$
 for some constant $C_{n,K}$ depending only on $n,K$.
\end{thm}

The paper is organized as follows. In Section 2, we will provide some necessary
materials for calculus, Sobolev spaces and Ricci curvature on Alexandrov spaces. In Section 3, we will investigate
 a further property of Perelman's concave functions. Poisson equations and mean value inequality on Alexandrov
  spaces will be discussed in Section 4. Bochner type formula will be established in Section 5.
   In the last section, we will prove Yau's gradient estimates on Alexandrov spaces.

\noindent\textbf{Acknowledgements.} We are grateful to Prof. Petrunin for his patient explanation on his manuscript
 \cite{petu96}. We also would like to thank Dr. Bobo Hua for his careful reading on the first version of this paper.
  He told us a gap in the previous proof of Proposition \ref{prop5.3}. The second author is partially supported by NSFC 10831008.

\section{Preliminaries}
\subsection{Alexandrov spaces}

Let $(X,|\cdot\cdot|)$ be a metric space. A rectifiable curve $\gamma$ connecting two points $p,q$ is called a
 geodesic if its length is equal to $|pq|$ and it has unit speed. A metric space $X$ is called a geodesic space
  if every pair points $p,q\in X$ can be connected by \emph{some} geodesic.

Let $k\in\R$ and $l\in\mathbb N$. Denote by $\mathbb M^l_k$ the simply connected, $l$-dimensional space form of
 constant sectional curvature $k$. Given three points $p,q,r$ in a geodesic space $X$, we can take a comparison
  triangle $\triangle \bar p\bar q\bar r$ in the model spaces $\M^2_k$ such that $|\bar p\bar q|=|pq|$,
  $|\bar q\bar r|=|qr|$ and $|\bar r\bar p|=|rp|$. If $k>0$, we add assumption  $|pq|+|qr|+|rp|<2\pi/\sqrt{k}$.
   Angles $\wa_k pqr:=\angle \bar p\bar q\bar r$ are called comparison angles.

A geodesic space $X$ is called an Alexandrov space (of \emph{locally} curvature bounded below) if it satisfies the
 following properties:\\
\indent(i) it is locally compact;\\
\indent(ii)  for any point $x\in X$ there exists a neighborhood $U_x$ of $x$ and a real number $\kappa$ such that,
 for any four different points $p, a, b,c$ in $U_x$, we have $$\wa_\kappa apb+\wa_\kappa bpc+\wa_\kappa cpa\ls 2\pi.$$

The Hausdorff dimension of an Alexandrov space is always an integer. Let $M$ be an $n$-dimensional Alexandrov space,
 we denote by $\rv$ the $n$-dimensional Hausdorff measure of $M$.
Let $p\in M$, given two geodesics $\gamma(t)$ and $\sigma(s)$ with $\gamma(0)=\sigma(0)=p$, the angle
 $$\angle \gamma'(0)\sigma'(0):=\lim_{s,t\to0}\wa_\kappa \gamma(t)p\sigma(s)$$
 is well defined. We denote by $\Sigma'_p$ the set of equivalence classes of geodesic $\gamma(t)$ with $\gamma(0)=p$,
  where $\gamma(t)$ is equivalent to $\sigma(s)$ if $\angle\gamma'(0)\sigma'(0)=0$. The completion of metric space
  $(\Sigma_p',\angle)$ is called the space of directions at $p$, denoted by $\Sigma_p$. The tangent cone at $p$, $T_p$,
   is the Euclidean cone over $\Sigma_p$.
  For two tangent vectors $u,v\in T_p$, their ``scalar product" is defined by (see Section 1 in \cite{petu07})
$$\ip{u}{v}:=\frac{1}{2} (|u|^2 + |v|^2- |uv|^2).$$

  For each point $x\not=p$, the symbol $\uparrow_p^x$ denotes the direction at $p$ corresponding to \emph{some} geodesic $px$.
We refer to the seminar paper \cite{bgp92} or the text book \cite{bbi01} for the details.

Let $p\in M$. Given a direction $\xi\in \Sigma_p$, there does possibly not exists geodesic $\gamma(t)$ starting at $p$
 with $\gamma'(0)=\xi$. To overcome the difficulty, it is shown in \cite{pp95} that for any $p\in M$ and any
  direction $\xi\in \Sigma_p$, there exists a \emph{quasi-geodesic} $\gamma: [0,+\infty)\to M$ with $\gamma=p$
   and $\gamma'(0)=\xi.$ (see also Section 5 of \cite{petu07}).

 Let $M$ be an $n$-dimensional Alexandrov space and $p\in M$. Denote by (\cite{os94})
 $$W_p:=\big\{x\in M\backslash\{p\}\ \big| \ {\rm there\ exists}\ y\in M\ {\rm such\ that}\ y\not
  =x \ {\rm and } \ |py|=|px|+|xy|\big\}.$$
 According to \cite{os94}, the set $W_p$ has full measure in $X$. For each $x\in W_p$, the direction
  $\uparrow_p^x$ is uniquely determined, since any geodesic in $M$ does not branch (\cite{bgp92}).
  Recall that the  map $\log_p: W_p\to T_p$ is defined by $\log_p(x):=|px|\cdot\uparrow_p^x$ (see \cite{petu07}).
 We denote by
 $$\mathscr W_p:=\log_p(W_p)\subset T_p.$$
  The map $\log_p: W_p\to\mathscr W_p$ is one-to-one.
   After Petrunin in \cite{petu98}, the \emph{exponential map} $\exp_p : T_p \to M$ is
defined as follows. $\exp_p(o)=p$ and for any $v \in T_p\backslash\{o\}$, $\exp_p(v)$ is a point on some quasi-geodesic
of length $|v|$ starting point $p$ along direction $v/|v| \in \Sigma_p$.  If the quasi-geodesic is not
unique, we fix some one of them as the definition of $\exp_p(v)$.
    Then $\exp_p|_{\mathscr W_p}$ is the inverse map of $\log_p$, and hence $\exp_p|_{\mathscr W_p}:\ \mathscr W_p\to W_p$
     is one-to-one.  If $M$ has curvature $\gs k$ on $B_p(R)$, then exponential map
     $$\exp_p: B_o(R)\cap \mathscr W_p\subset T^k_p\to M$$ is an non-expending map (\cite{bgp92}), where $T^k_p$
      is the $k$-cone over $\Sigma_p$ and $o$ is the vertex of $T_p$.

 A point $p$ in an $n$--dimesional Alexandrov space $M$ is called to be \emph{regular} if its tangent cone $T_p$ is isometric
  to Euclidean space $\R^n$ with standard metric. A point $p\in M$ is called a singular point if it is not regular.
   Denote by $S_M$ the set of singular points of $M$. It is shown (in Section 10 of \cite{bgp92}) that the Hausdorff
    dimension of $S_M$ is $\ls n-1$ (see \cite{bgp92,os94}). Remark that the singular set $S_M$ is possibly dense in
     $M$ (see \cite{os94}). It is known that $M\backslash S_M$ is convex \cite{petu98}. Let $p$ be a regular point in
      $M$, for any $\epsilon>0$ there is a neighborhood $B_p(r)$ which is bi-Lipschitz  onto  an open domain in
       $\mathbb R^n$ with bi-Lipschitz constant $1+\epsilon$ (see Theorem 9.4 of \cite{bgp92}). Namely, there
        exists a map $F$ from $B_p(r)$ onto an open domain in $\mathbb R^n$ such that
 $$ (1+\epsilon)^{-1}\ls \frac{\|F(x)-F(y)\|}{|xy|}\ls 1+\epsilon\qquad \forall\ x,y\in B_p(r),\ x\not=y.$$

 A (generalized) $C^1$-structure on $M\backslash S_M$ is established in \cite{os94} as the following sense:
  there is an open covering $\{U_\alpha\}$ of  an open set containing $M\backslash S_M$, and a family of homeomorphism
   $\phi_\alpha:\ U_\alpha\to O_\alpha\subset\mathbb R^n$   such that if $U_\alpha\cap U_\beta\not =\varnothing$, then
 $$\phi_\alpha\circ\phi_\beta^{-1}:\ \phi_\beta(U_\alpha\cap U_\beta)\to \phi_\alpha(U_\alpha\cap U_\beta)$$
 is $C^1$ on $\phi_\beta\big((U_\alpha\cap U_\beta)\backslash S_M\big)$. A corresponding $C^0$-Riemannian
  metric $g$ on $M\backslash S_M$ is introduced in \cite{os94}. In \cite{p-dc}, this $C^1$-structure and the
   corresponding $C^0$-Riemannian metric has been extended to be a $DC^1$-structure and the corresponding
    $BV^0_{loc}$-Riemannian metric. Moreover, we have the following:\\
 \indent (1)\indent The distance function on $M\backslash S_M$ induced from $g$ coincides with the original
  one of $M$ (\cite{os94});\\
 \indent (2)\indent The Riemannian measure on $M\backslash S_M$ coincides with the Haudorff measure of $M$,
  that is, under a coordinate system $(U,\phi)$, the metric $g=(g_{ij}), $ we have
 \begin{equation}\label{eq2.1}d\rv(x)=\sqrt{\det(g(\phi(x)))}dx^1 \wedge\cdots\wedge dx^n\end{equation}
 for  all $x\in U\backslash S_M $ (Section 7 in \cite{os94}).

  A point $p$ is called a \emph{smooth} point  if it is regular and there exists a coordinate system $(U,\phi)$
   around $p$ such that \begin{equation}\label{eq2.2}
  |g_{ij}(\phi(x))-\delta_{ij}|=o(|px|),
  \end{equation}
   where $(g_{ij})$ is the corresponding Riemannian metric (see \cite{os94}) near $p$ and $(\delta_{ij})$
   is the identity $n\times n$ matrix. the set of smooth points has full measure \cite{p-dc}.

  \begin{lem}\label{lem2.1}
  Let $p\in M$ be a smooth point.  We have
   \begin{equation}\label{eq2.3}
   \Big|\frac{d\rv (x)\emph{}}{d{H^n}(v)}-1\Big|=o(r),\qquad \forall\ v\in B_o(r)\cap\mathscr W_p ,
   \end{equation}
  where $x=\exp_p (v), $
  and\begin{equation}\label{eq2.4}
  {H^n}\big(B_o(r)\cap\mathscr W_p\big)\gs H^n\big(B_o(r)\big)\cdot(1-o(r))
  \end{equation}
 where $H^n$ is $n$-dimensional Hausdorff measure on $T_p$.
 \end{lem}
\begin{proof}
Let $(U,\phi)$ be a coordinate system such that  $\phi(p)=0$ and  $B_p(r)\subset U$.
 For each $v\in B_o(r)\cap\mathscr W_p\subset T_p$,
$$d\rv(x)=\sqrt{\det[g_{ij}(\phi(x))]}dx^1\wedge\cdots\wedge dx^n,$$
where $x=\exp_p(v).$ Since $p$ is regular,  $T_p$ is isometric to $\mathbb R^n$. We obtain that
$$dH^n(v)=dH^n(o)= dx^1\wedge\cdots\wedge dx^n$$
 for all $v\in T_p$.
 We get
$$\frac{d\rv (x)\emph{}}{d{H^n}(v)}-1=\sqrt{\det[g_{ij}(\phi(x))]}-1.$$
Now the estimate (\ref{eq2.3}) follows from this and the equation (\ref{eq2.2}).

Now we want to show \eqref{eq2.4}.

  Equation \eqref{eq2.2} implies that (see \cite{p-dc}) for any $x,y\in B_p(r)\subset U$,
  $$\big||xy|-\|\phi(x)-\phi(y)\|\big|=o(r^2).$$
  In particular, the map $\phi: U\to \mathbb R^n$ satisfies
  $$ \phi\big(B_p(r)\big)\supset B_o\big(r-o(r^2)\big).$$
On one hand, from  \eqref{eq2.2}, we have
\begin{equation}\label{eq2.5}\begin{split}
{\rm vol}(B_p(r))&=\int_{\phi(B_p(r))}\sqrt{\det(g_{ij})}dx^1\wedge\cdots\wedge dx^n\\
&\gs H^n\big(\phi(B_p(r))\big)\cdot\big(1- o(r)\big)\gs   H^n\big( B_o(r-o(r^2))\big)\cdot(1- o(r))\\
&= H^n(B_o(r))\cdot(1- o(r)).\end{split}
\end{equation}

On the other hand, because
     $\exp_p: B_o(R)\cap \mathscr W_p\subset T^k_p\to M$ is an non-expending map (\cite{bgp92}), where $T^k_p$
      is the $k$-cone over $\Sigma_p$ and $o$ is the vertex of $T_p$,  we have
  $$ \exp_p: B_o(R)\cap \mathscr W_p\subset T_p\to M$$
  is a Lipschitz map with Lipschitz constant $1+O(r^2).$ Hence we get
   $$H^n(B_o(r)\cap\mathscr W_p)\cdot(1+O(r^2))\gs{\rm vol}(B_p(r)).$$
Therefore, by combining with equation \eqref{eq2.5},
we have
$$H^n(B_o(r)\cap\mathscr W_p)\gs H^n\big(B_o(r)\big)\cdot(1-O(r^2))\cdot(1-o(r))= H^n(B_o(r) \cdot(1-o(r)).$$
This is the desired estimate (\ref{eq2.4}).
\end{proof}
\be{rem}{If $M$ is a $C^2$-Riemannian manifold, then for sufficiently small $r>0$, we have
 $$\Big|\frac{d{\rm vol}(x)}{dH^n(v)}-1\Big|=O(r^2),\quad \forall \ v\in B_o(r)\subset T_p\quad {\rm and }\quad x=\exp_p(v).$$
}

Let $M$ be an Alexandrov space without boundary and $\Omega\subset M $ be an open set. A locally Lipschitz
function $f: \Omega\to\R$ is called to be  $\lambda$-concave (\cite{petu07}) if for all geodesics $\gamma(t)$ in $ \Omega$, the function
$$f\circ\gamma(t)-\lambda\cdot t^2/2$$
is concave.  A function $f: \Omega\to\R$ is called to be semi-concave if for any $x\in \Omega$, there exists
 a neighborhood of $U_x\ni  x$ and a number $\lambda_x\in \R$ such that $f|_{U_x}$ is  $\lambda_x$-concave.
  In fact, it was shown that the term ``geodesic" in the definition can be replaced by ``quasigeodesic" (\cite{pp95,petu07}).
Given a semi-concave function $f: M\to\R$, its differential $d_pf$ and gradient $\nabla_pf$ are well-defined
 for each point $p\in M$ (see Section 1 in \cite{petu07} for the basic properties  of semi-concave functions).
 \begin{center}\emph{From now on, we always consider Alexandrov spaces without boundary.}
 \end{center}

 Given a semi-concave function $f: M\to\mathbb R$, a point $p$ is called a \emph{$f$-regular} point if $p$ is
  smooth, $d_pf$ is a linear map on $T_p \ (=\mathbb R^n)$ and there exists a quadratic form $H_pf$ on $T_p$ such that
 \be{equation}
{\label{eq2.6}
f(x)=f(p)+d_pf(\uparrow^x_p)\cdot|xp|+\frac{1}{2}H_pf(\uparrow^x_p,\uparrow^x_p)\cdot|px|^2+o(|px|^2)
}
for any direction $\uparrow^x_p$. We denote by $Reg_f$ the set of all $f$-regular points in $M$.
 According to \cite{p-dc}, $Reg_f$ has full measure in $M$.
\be{lem}
{\label{lem2.3}Let $f$ be a semi-concave function on $M$ and $p\in M$. Then we have
\begin{equation}\label{eq2.7}
\fint_{B_p(r)}\Big(f(x)-f(p)\Big)d{\rm vol}(x)=\frac{nr}{n+1}\cdot\fint_{\Sigma_p}d_pf(\xi)d\xi+o(r),
\end{equation}
where $\fint_Bfd{\rm vol}=\frac{1}{\rv (B)}\int_Bfd{\rm vol}.$ Furthermore, if we add to assume that $p\in Reg_f$, then  \begin{equation}\label{eq2.8}
\fint_{B_p(r)}\Big(f(x)-f(p)\Big)d{\rm vol}(x)=\frac{nr^{2}}{2(n+2)}\cdot\fint_{\Sigma_p}H_pf(\xi,\xi)d\xi+o(r^{2}).
\end{equation}
}
\begin{proof}
According to Theorem 10.8 in \cite{bgp92}, we have
\begin{equation}\label{eq2.9}
\frac{d{\rm vol}(\exp_p(v))}{d{H^n}(v)}=1+o(1), \qquad \frac{\rv(B_p(r))}{H^n(B_o(r))}=1+o(1).
\end{equation}
Similar as in the proof of equation \eqref{eq2.4}, we have
$$ \rv(B_o(r)\cap \mathscr W_p)\gs H^n(B_o(r)) \cdot(1-o(1)).$$
Since   $f(x)-f(p)=d_pf(\uparrow^x_{p})\cdot|px|+o(|px|)$, we get
\begin{equation}\label{eq2.10}\begin{split}
\int_{B_p(r)}\Big(&f(x)-f(p)\Big)d\rv(x)\\
&=\int_{B_o(r)\cap \mathscr W_p}\Big( d_pf(v)+o(|v|)\Big)(1+o(1))dH^n(v).\end{split}
\end{equation}
On the other hand, from (\ref{eq2.9}), we have
$$\Big|\int_{B_o(r)\backslash \mathscr W_p} d_pf(v)dH^n(v)\Big|\ls O(r)\cdot H^n\big(B_o(r)\backslash \mathscr W_p\big)\ls o(r^{n+1}).$$
By combining this and (\ref{eq2.10}), we obtain
\begin{equation*}\begin{split}
\fint_{B_p(r)}\Big(f(x)-f(p)\Big)d\rv(x)&=\frac{H^n(B_o(r))}{\rv(B_p(r))}\fint_{B_o(r)}d_pf(v)dH^n(v)+o(r)\\
&=\fint_{B_o(r)}d_pf(v)dH^n(v)(1+o(1))+o(r)\\
&=\fint_{B_o(r)}d_pf(v)dH^n(v) +o(r)\\&
=\frac{nr}{n+1}\fint_{\Sigma_p} d_pf(\xi)d\xi+o(r).
\end{split}
\end{equation*}
This is equation (\ref{eq2.7}).

Now we want to prove (\ref{eq2.8}). Assume that $p$ is a $f$-regular point. From  \eqref{eq2.6} and Lemma \ref{lem2.1}, we have
 \begin{equation}\label{eq2.11}\begin{split}
\int_{B_p(r)}&\Big(f(x)-f(p)\Big)d\rv(x)\\
&=\int_{B_o(r)\cap\mathscr W_p}\Big(d_pf(v)+\frac{1}{2}H_pf(v,v)+o(|v|^2)\Big)\cdot(1+o(r))dH^n(v).\end{split}
\end{equation}
Using Lemma \ref{lem2.1} again, we have
$$\Big|\int_{B_o(r)\backslash\mathscr W_p}d_pf(v)dH^n\Big|\ls O(r)\cdot H^n(B_o(r)\backslash\mathscr W_p)=O(r)\cdot o(r)\cdot H^n(B_o(r))=o(r^{n+2}).$$
Noticing that $\int_{B_o(r)}d_pf(v)dH^n=0$, we get
\begin{equation}\label{eq2.12}\int_{B_o(r)\cap\mathscr W_p}d_pf(v)dH^n=o(r^{n+2}).
\end{equation}
Similarly, we have
\begin{equation}\label{eq2.13}\begin{split}
\int_{B_o(r)\cap\mathscr W_p}H_pf(v,v)dH^n&=\int_{B_o(r)}H_pf(v,v)dH^n+o(r^{n+3}). \end{split}\end{equation}
From \eqref{eq2.11}--(\ref{eq2.13}) and Lemma \ref{lem2.1}, we have
\begin{equation*}
\begin{split}
\fint_{B_p(r)}\Big(f(x)-f(p)\Big)d\rv(x)&=\frac{H^n(B_o(r))}{\rv(B_p(r))}\fint_{B_o(r)}H_pf(v,v)dH^n+o(r^{2})\\
&= \fint_{B_o(r)}H_pf(v,v)dH^n(1+o(r))+o(r^{2})\\
&=\frac{nr^{2}}{2(n+2)}\fint_{\Sigma_p}H_pf(\xi,\xi)d\xi+o(r^{2}).
\end{split}\end{equation*}
This is the desired (\ref{eq2.8}).
\end{proof}
Given a continuous function $g$ defined on $B_p(\delta_0)$, where $\delta_0$ is a sufficiently small
 positive number, we have $$\int_{\partial B_p(r)}gd{\rv}=\frac{d}{dr}\int_{B_p(r)}gd{\rv}$$
for almost all $r\in(0,\delta_0).$\\
\noindent{\bf Lemma $2.3'$}
\emph{Let $f$ be a semi-concave function on $M$ and $p\in M$. Assume $\delta_0$ is a sufficiently small
 positive number. Then we have, for almost all $r\in (0,\delta_0),$
\begin{equation}\label{eq2.14}
\fint_{\partial B_p(r)}\Big(f(x)-f(p)\Big)d{\rm vol}(x)=nr\cdot\fint_{\Sigma_p}d_pf(\xi)d\xi+o(r).
\end{equation}
 Furthermore, if we add to assume that $p\in Reg_f$, then we have, for almost all $r\in (0,\delta_0),$
  \begin{equation}\label{eq2.15}
\fint_{\partial B_p(r)}\Big(f(x)-f(p)\Big)d{\rm vol}(x)=\frac{r^{2}}{2}\cdot\fint_{\Sigma_p}H_pf(\xi,\xi)d\xi+o(r^{2}).
\end{equation}
}

\subsection{Sobolev spaces}
Several different notions of Sobolev spaces have been established,
 see\cite{c99,kms01,n00,ks93,ks03}\footnote{In \cite{c99,ks93,n00,ks03}, Sobolev spaces
  are defined on  metric measure spaces supporting a doubling property and a Poincar\'e inequality. Since $\Omega$
   is bounded, it satisfies a   doubling property and supports a  weakly Poincar\'e inequality \cite{kms01}.}.
    They coincide each other on Alexandrov spaces.

Let $M$ be an $n$-dimensional Alexandrov space and let $\Omega$ be a bounded open domain in $M$.  Given $u\in C(\Omega)$.
 At a point $p\in \Omega$, the \emph{pointwise Lipschitz constant} (\cite{c99}) and \emph{subgradient norm}
  (\cite {lv07-hj}) of $u$ at $x$ are defined by:
$${\rm Lip}u(x):=\limsup_{y\to x}\frac{|f(x)-f(y)|}{|xy|}\quad {\rm and }\quad |\nabla ^-u|(x):=\limsup_{y\to x}\frac{\big(f(x)-f(y)\big)_+}{|xy|},$$
where $a_+=\max\{a,0\}.$ Clearly, $|\nabla ^-u|(x)\ls {\rm Lip}u(x).$ It was shown in \cite{lv07-hj} for
 a locally Lipschitz function $u$ on $\Omega$, $$ |\nabla^-u|(x)={\rm Lip}u(x)$$
for almost all $x\in\Omega$\footnote{ See Remark 2.27 in \cite{lv07-hj} and its proof.}.

 Let $x\in \Omega$ be a regular point, We say that a function $u$ is \emph{differentiable} at $x$,
  if there exist a vector in $T_x$\ ($=\mathbb R^n$), denoted by $\nabla u(x)$, such that for all
   geodesic $\gamma(t):[0,\epsilon)\to\Omega$ with $\gamma(0)=x$ we have
\begin{equation}\label{eq2.16}
u(\gamma(t))=u(x)+t\cdot\ip{\nabla u(x)}{\gamma'(0)}+o(t).
\end{equation}
 Thanks to Rademacher theorem, which was proved by Cheeger \cite{c99} in the framework of general
 metric measure spaces with
a doubling measure and a Poincar\'e inequality for upper gradients and  was proved by Bertrand \cite{b08}
 in Alexandrov space via a
simply argument, a locally Lipschitz function $u$ is differentiable almost everywhere in $M$. (see also \cite{o95}.)
 Hence  the  vector $\nabla u(x)$ is well defined almost everywhere in $M$.

Remark that any semi-concave function $f$ is locally Lipschitz. The differential of $u$ at any point $x$, $d_xu$, is
well-defined.(see Section 1 in \cite{petu07}.) The gradient $\nabla_xu$ is defined as the maximal value point of
 $d_xu: B_o(1)\subset T_x\to \R.$

\begin{prop}\label{diff}
Let $u$ be a semi-concave function on an open domain $\Omega\subset M$. Then for any $x\in \Omega\backslash S_M$,
 we have
$$|\nabla_xu|\ls |\nabla^-u|(x).$$
Moreover, if $u$ is differentiable at $x$, we have
$$|\nabla_xu|= |\nabla^-u|(x)={\rm Lip}u(x)=|\nabla u(x)|.$$
\end{prop}
\begin{proof} Without loss of generality,  we can assume that $|\nabla_xu|>0$. (Otherwise, we are done.)
 Since $x$ is regular, there exists direction $-\nabla_xu$. Take a sequence of point $\{y_j\}_{j=1}^\infty$ such that $$\lim_{j\to\infty}y_j=x\quad{\rm and }\quad\lim_{j\to\infty}\uparrow^{y_j}_x=-\frac{\nabla_xu}{|\nabla_xu|}.$$
 By semi-concavity of $u$, we have
 $$u(y_j)-u(x)\ls |xy_j|\cdot\ip{\nabla_x u}{\uparrow_x^{y_j}}+\lambda|xy_j|^2/2,\quad j=1,2,\cdots$$
for some $\lambda\in\R$.
Hence
$$-\ip{\nabla_x u}{\uparrow_x^{y_j}}\ls \frac{\big(u(x)-u(y_j)\big)_+}{|xy_j|}+\lambda|xy_j|/2,\quad j=1,2,\cdots$$
Letting $j\to\infty$, we conclude  $|\nabla_xu|\ls |\nabla^-u|(x).$

Let us prove the second assertion.  We need only to show ${\rm Lip}u(x)\ls|\nabla u(x)|$ and $|\nabla u(x)|\ls |\nabla_xu|.$
 Since $u$ is differentiable at $x$, we have
$$u(y)-u(x)=|xy|\cdot\ip{\nabla u(x)}{\uparrow_x^y}+o(|xy|)$$
for all $y$ near $x$.
Consequently,
$$|u(y)-u(x)|=|xy|\cdot|\ip{\nabla u(x)}{\uparrow_x^y}|+o(|xy|)\ls|xy|\cdot |\nabla u(x)|+o(|xy|).$$
This implies that ${\rm Lip}u(x)\ls|\nabla u(x)|.$

Finally, let us show $|\nabla u(x)|\ls |\nabla_xu|.$ Indeed, combining the differentiability and semi-concavity of $u$, we have
$$ |xy|\cdot\ip{\nabla u(x)}{\uparrow_x^y}+o(|xy|)=u(y)-u(x)\ls |xy|\cdot\ip{\nabla_x u}{\uparrow_x^y}+\lambda |xy|^2/2$$
for all $y$ near $x$.
Without loss of generality, we can assume that $|\nabla u(x)|>0$. Take $y$ such that direction $\uparrow_x^y$
 arbitrarily close to $\nabla u(x)/|\nabla u(x)|$. We get
$$|\nabla u(x)|^2\ls  \ip{\nabla_xu}{\nabla u(x)}\ls |\nabla _xu|\cdot|\nabla u(x)|.$$
This is
 $|\nabla _xu|\ls|\nabla u(x)|.$
 \end{proof}

According to this Proposition \ref{diff}, we will not distinguish  between two notations $\nabla_x u $ and $\nabla u(x)$
 for any semi-concave function $u$.\\

We denote by $Lip_{loc}(\Omega)$ the set of locally Lipschitz continuous functions on $\Omega$, and by $Lip_0(\Omega)$
 the set of Lipschitz continuous functions on $\Omega$ with compact support in $\Omega.$ For any $1\ls p\ls +\infty$ and
   $u\in Lip_{loc}(\Omega)$, its $W^{1,p}(\Omega)$-norm is defined by
$$\|u\|_{W^{1,p}(\Omega)}:=\|u\|_{L^{p}(\Omega)}+\|{\rm Lip}u\|_{L^{p}(\Omega)}.$$
 Sobolev spaces $W^{1,p}(\Omega)$ is defined by the closure of the set
$$\{u\in Lip_{loc}(\Omega)|\ \|u\|_{W^{1,2}(\Omega)}<+\infty\},$$
under  $W^{1,p}(\Omega)$-norm.
Spaces $W_0^{1,p}(\Omega)$ is defined by the closure of $Lip_0(\Omega)$ under  $W^{1,p}(\Omega)$-norm.
(This coincides with the definition in \cite{c99}, see Theorem 4.24 in \cite{c99}.)
We say a function $u\in W^{1,p}_{loc}(\Omega)$ if $u\in W^{1,p}(\Omega')$ for every open subset $\Omega'\Subset\Omega.$
 According to   Kuwae--Machigashira--Shioya \cite{kms01}  (see also Theorem 4.47 in \cite{c99}),  the ``derivative" $\nabla u$
  is well-defined for all  $u\in W^{1,p}(\Omega)$ with $1<p<\infty$. Cheeger in Theorem 4.48 of \cite{c99}
   proved that $W^{1,p}(\Omega)$ is reflexive for any $1<p<\infty.$

\subsection{Ricci curvature}
For an Alexandrov space, several different definitions of ``Ricci curvature having lower bounds by $K$" have been
 given (see Introduction).

Here, let us recall the  definition of lower bounds of Ricci curvature on Alexandrov space in \cite{zz10}.

 Let $ M$ be an $n$-dimensional Alexandrov space. According to Section 7 in \cite{bgp92}, if $p$ is an interior point of a geodesic $\gamma$, then the
  tangent cone $T_p$ can be isometrically split into $$T_p=L_p\times \R\cdot\gamma',\qquad v=(v^\bot,t).$$
   We set
   $$\Lambda_p=\{\xi\in L_p: \ |\xi|=1\}.$$
\begin{defn}\label{defn2.4}
 Let $\sigma(t):(-\ell,\ell)\to M$ be a geodesic  and $\{g_{\sigma(t)}(\xi)\}_{-\ell<t<\ell}$ be a  family of
  functions on $\Lambda_{\sigma(t)}$ such that $g_{\sigma(t)}$ is continuous on $\Lambda_{\sigma(t)}$ for
   each $t\in(-\ell,\ell)$.
We say that the family $\{g_{\sigma(t)}(\xi)\}_{-\ell<t<\ell}$ satisfies $Condition\ (RC)$ on $\sigma$ if
for any two points $q_1,q_2\in \sigma$ and  any sequence $\{\theta_j\}_{j=1}^\infty$ with $\theta_j\to 0$
 as $j\to\infty$, there exists an isometry $T: \Sigma_{q_1}\to \Sigma_{q_2}$ and a subsequence $\{\delta_j\}$ of $\{\theta_j\}$
 such that
\begin{equation}\label{eq2.17}
\begin{split}|\exp_{q_1}&(\delta_j l_1\xi),\ \exp_{q_2}(\delta_j l_2T\xi)|\\ \ls  &|q_1q_2|+(l_2-l_1)\ip{\xi}{\gamma'}\cdot \delta_j\\&+\Big(\frac{(l_1-l_2)^2}{2|q_1q_2|}
-\frac{g_{q_1}(\xi^\bot)\cdot|q_1q_2|}{6}\cdot(l_1^2+l_1\cdot l_2+l^2_2)\Big)\cdot\Big(1-\ip{\xi}{\gamma'}^2\Big)
\cdot\delta^2_j\\
&+o(\delta_j^2)
\end{split}\end{equation}
 for any $l_1,l_2\gs0$ and any $\xi\in \Sigma_{q_1}$.
\end{defn}

If $M$ has curvature bounded below by $k_0$ (for some $k_0\in \R$), then by Theorem 1.1 of \cite{petu98}
(or see Theorem 20.2.1 of \cite{akp11}),  the family of functions $\{g_{\sigma(t)}(\xi)=k_0\}_{-\ell<t<\ell}$
 satisfies $Condition \ (RC)$ on $\sigma$. In particular, if a family $\{g_{\sigma(t)}(\xi)\}_{-\ell<t<\ell}$
 satisfies $Condition\ (RC)$, then the family  $\{g_{\sigma(t)}(\xi)\vee k_0\}_{-\ell<t<\ell}$ satisfies $Condition\ (RC)$ too.
\begin{defn}\label{defn2.5}
 Let  $\gamma:[0,a)\to M$ be a geodesic. We say that $M$ has \emph{Ricci curvature bounded below by $K$ along $\gamma$},
if for any $\epsilon>0$ and any $0<t_0<a$, there exists $\ell=\ell(t_0,\epsilon)>0$ and a family of continuous functions $\{g_{\gamma(t)}(\xi)\}_{t_0-\ell<t<t_0+\ell}$ on $\Lambda_{\gamma(t)}$ such that the family satisfies
 $Condition\ (RC)$ on $\gamma|_{(t_0-\ell,\ t_0+\ell)}$ and
\begin{equation}\label{eq2.18}
(n-1)\cdot\fint_{\Lambda_{\gamma(t)}}g_{\gamma(t)}(\xi)d\xi\gs K-\epsilon, \qquad \forall t\in(t_0-\ell,t_0+\ell),
\end{equation}
where $\fint_{\Lambda_{x}}g_x(\xi)=\frac{1}{vol(\Lambda_x)}\int_{\Lambda_x}g_x(\xi)d\xi$.

 We say that $M$ has \emph{Ricci curvature  bounded below by $K$}, denoted by $Ric(M)\gs K$, if  each point $x\in M$
  has  a neighborhood $U_x$   such that  $M$ has Ricci curvature bounded below by $K$ along every geodesic $\gamma$ in $ U_x$.
 \end{defn}

\begin{rem}\label{cuvtoRic}
Let $M$ be an $n$-dimensional Alexandrov space  with curvature $\gs k$. Let  $\gamma:[0,a)\to M$ be any  geodesic.
 By \cite{petu98},  the family of functions $\{g_{\gamma(t)}(\xi):=k\}_{0<t<a}$ satisfies Condition $(RC)$ on $\gamma$.
   According to the Definition \ref{defn2.5}, we know that $M$ has     Ricci curvature bounded from below by $(n-1)k$
   along $\gamma$. Because of  the arbitrariness of geodesic $\gamma$, $M$ has  Ricci curvature bounded from below by $(n-1)k$.
\end{rem}

Let $M$ be an $n$-dimensional Alexandrov space $M$ having Ricci curvature $\gs K$. In \cite{petu09} and Appendix of
 \cite{zz10}, it is shown that  metric measure space $(M,d,\rv)$ satisfies Sturm--Lott--Villani curvature-dimension
 condition $CD(K,n)$, and hence measure contraction property $MCP(K,n)$ (see \cite{s06-2,o07}, since Alexandrov spaces
  are non-branching) and infinitesimal Bishop-Gromov condition $BG(K,n)$ (\cite{ks07}, this is equivalent to $MCP(K,n)$
  on Alexandrov spaces). Consequently, $M$ satisfies a corresponding Bishop--Gromov volume comparison theorem \cite{s06-2,ks07}
   and a corresponding Laplacian comparison in sense of distribution \cite{ks07}.

\section{Perelman's concave functions}
Let $M$ be an Alexandrov space and $x\in M$. In \cite{p94}, Perelman constructed a strictly concave function on a neighborhood
 of $x$. This implies that there exists a convex neighborhood for each point in $M$. In this section, we will investigate a
  further property of Perelman's concave functions.

In this section, we always assume that $M$ has curvature bounded from below by $k$ (for some $k\in \R$).

Let $f:\Omega\subset M\to \R$ be a semi-concave function and $x\in\Omega$. Recall that
a vector $v_s\in T_x$ is said to be a supporting vector of $f$ at $x$ (see \cite{petu07}) if
$$d_xf(\xi)\ls -\ip{ v_s}{\xi} \quad {\rm for\ all}\ \xi\in \Sigma_x.$$
The set of supporting vectors of $f$ at $x$ is a non-empty convex set (see Lemma 1.3.7 of \cite{petu07}).
For a distance function $f=dist_p$, by the first variant formula (see, for example, \cite{bbi01}), any direction
 $\uparrow_x^p$ is a supporting vector of $f$ at $x\not=p$.
\be{prop}
{\label{prop3.1}Let $f:\Omega\subset M\to \R$ be a semi-concave function and $x\in\Omega$. Then we have
$$\int_{\Sigma_x}d_xf(\xi)d\xi\ls 0.$$

Furthermore, if $f$ is a distance function $f=dist_p$ and  $x\not=p$,  the $``="$ holds implies that $\uparrow_x^p$
 is uniquely determined and $\max_{\xi\in\Sigma_x}|\xi,\uparrow_x^p|=\pi.$
}
\begin{proof}
Let $v_s$ be a support vector of $f$ at $x$, then
$$d_xf(\xi)\ls -\ip{v_s}{\xi},\quad \forall \ \xi\in \Sigma_x.$$
Without loss of generality, we can assume $v_s\not=0$. (If $v_s=0$, then $d_xf(\xi)\ls0$. We are done.)
Setting $\eta_0=\frac{v_s}{|v_s|}\in \Sigma_x$, we have
$$d_xf(\xi)\ls -\ip{v_s}{\xi}=-|v_s|\cdot\cos(|\eta_0,\xi|) \qquad \forall \ \xi\in \Sigma_x.$$
Denote $D=\max_{\xi\in\Sigma_x}|\xi,\eta_0|$. By using co-area formula, we have $$I:=\int_{\Sigma_x}d_xf(\xi)d\xi\ls-|v_s|\cdot\int_{\Sigma_x}\cos(|\eta_0,\xi|)d\xi=
  -|v_s|\cdot\int_0^D\cos t\cdot A(t)dt,$$
  where $A(t)={\rm vol}_{n-2}(\{\xi\in\Sigma_x: \ |\xi,\eta_0|=t\}).$

If $D\ls\pi/2$, then $I<0$.

We consider the case $D>\pi/2.$ Since $\Sigma_x$ has curvature $\gs1$, by Bishop--Gromov comparison,
we have
$$A(\pi-t)\ls A(t)\cdot\frac{{\rm vol}_{n-2}(\partial B_o(\pi-t)\subset \mathbb S^{n-1})}{{\rm vol}_{n-2}(\partial B_o(t)
\subset \mathbb S^{n-1})}=A(t)$$
 for any $t\ls\pi/2.$
  Hence
  \begin{align*}\frac{I}{|v_s|}
&\ls-\int^{\pi/2}_0\cos t\cdot A(t)dt-\int^{D}_{\pi/2}\cos t\cdot A(t)dt\\
&\ls -\int^{\pi/2}_0\cos t\cdot A(\pi-t)dt-\int^{D}_{\pi/2}\cos t\cdot A(t)dt\\&=\int^{\pi}_{D}\cos t\cdot A(t)dt\ls0.
\end{align*}
 Moreover, if $I=0$, then $D=\pi.$

If $f=dist_p$, then $v_s$ can be chosen as any direction $\uparrow^p_x$. When $I=0$, we have
\begin{equation}\label{eq3.0+1}
d_xf(\xi)=-\ip{\uparrow_x^p}{\xi},\quad \forall\ \xi\in \Sigma_x,
 \end{equation}
and
$$\max_{\xi\in\Sigma_x}|\xi,\uparrow_x^p|=\pi. $$
The left-hand side of \eqref{eq3.0+1} does not depend on the choice of direction $\uparrow_x^p$. This implies that $\uparrow^p_x$  is determined uniquely.
\end{proof}
\be{lem}
{\label{lem3.2}Given any $n\in\mathbb N$ and any constant $C>0$, we can find $\delta_0=\delta_0(C,n)$ satisfying the following
 property: for any $n$-dimensional Alexandrov spaces $\Sigma^n$ with curvature $\gs1$, if there exist $0<\delta<\delta_0$
  and  points $\{p_j\}_{j=1}^N\subset \Sigma^n$ such that
\begin{equation}\label{eq3.1}
|p_ip_j|>\delta\quad (i\not=j), \qquad N:=\#\{p_j\}\gs C\cdot \delta^{-n}
\end{equation}
 and \begin{equation}\label{eq3.2}
{\rm rad}(p_j):=\max_{q\in \Sigma^n}|p_jq|=\pi\quad  {\rm for \ each}\quad 1\ls j\ls N,
\end{equation}
 then $\Sigma^n$ is isometric to $\mathbb S^n$.
}
\begin{proof}
We use an induction argument with respect to the dimension $n$.
When $n=1$, we take $\delta_0(C,1)=C/3$. Then for each 1--dimensional Alexandrov space $\Sigma^1$ satisfying the
 assumption of the Lemma must contain  at least three different points $p_1, p_2$ and $p_3$ with
 ${\rm rad}(p_i)=\pi$, $i=1,2,3$. Hence $\Sigma^1$ is isometric to $\mathbb S^1$.

Now we assume that the Lemma holds for dimension $n-1$. That is, for any $\widetilde{C}$, there exists
$\delta_0(\widetilde{C},n-1)$ such that any $(n-1)$--dimensional Alexandrov space satisfying the condition
of the Lemma must  be isometric to $\mathbb S^{n-1}$.

We want to prove the Lemma for dimension $n$. Fix any constant $C>0$ and let
\begin{equation}\label{eq3.3}
\delta_0(C,n)=\min\Big\{\frac{10}{8}\cdot \delta_0\Big(\frac{C}{11\pi} \cdot (10/8)^{1-n},n-1\Big),1\Big\}.
\end{equation}

Let $\Sigma^n$ be an $n$--dimensional Alexandrov space with curvature $\gs1.$ Suppose that there exists
 $0<\delta<\delta_0(C,n)$ and a set of points $\{p_\alpha\}_{\alpha=1}^N\subset \Sigma^n$ such that they
  satisfy (\ref{eq3.1}) and (\ref{eq3.2}).

 Let $q_1\in \Sigma^n$ be the point that $|p_1q_1|=\pi$. Then $\Sigma^n$ is a suspension over some $(n-1)$-dimensional Alexandrov space $\Lambda$ of curvature $\gs1$ and with vertex $p_1$
  and $q_1$, denoted by $\Sigma^n=S(\Lambda)$.  We divide $\Sigma^n$ into  pieces
  $A_1, \ A_2,  \cdots, A_l,\ \cdots, \ A_{\bar l}$ as
  $$A_l=\big\{x\in \Sigma^n :\
  (\delta/10)\cdot l<|xp_1|\ls(\delta/10)\cdot (l+1)\big\},\quad 0\ls l\ls \bar l:=[\frac{\pi}{\delta/10}],$$
  where $[a]$ is the integer such that $[a]\ls a<[a]+1.$
Then there exists some piece, say $A_{l_0}$, such that
 \begin{equation}\label{eq3.4}
 N_1:=\# \big(A_{l_0}\cap \{p_j\}_{j=1}^N\big)\gs \frac{N}{\bar l+1}\gs  \frac{N}{10\pi/\delta
 +1}\overset{(\delta<1)}{\gs} \frac{C}{11\pi}\cdot \delta^{1-n}.
 \end{equation}
Notice that $$A_1\cup A_2\subset B_{p_1}(\delta/2)\quad {\rm and}\quad   A_{\bar l}\cup A_{\bar l-1}\subset B_{q_1}(\delta/2),$$
we have $l_0\not\in\{1,2,\bar l-1,\bar l\}.$

We denote the points $A_{l_0}\cap \{p_\alpha\}_{\alpha=1}^N$ as $(x_i,t_i)_{i=1}^{N_1}\subset S(\Lambda)\
 (=\Sigma^n),$ where $x_i\in \Lambda$ and $0<t_i<\pi$ for $1\ls i\ls N_1.$ Let $\gamma_i$
  be the geodesic $p_1(x_i,t_i)q_1$ and $\widetilde{p}_i=\gamma_i\cap \partial B_{p_1}\big((l_0+1)\cdot \delta/10\big).$
   By triangle inequality, we have
\begin{equation}\label{eq3.5}
|\widetilde{p}_i\widetilde{p}_j|\gs\frac{8}{10}\cdot \delta.
\end{equation}
Applying  cosine law, we have
$$\cos(|\widetilde{p}_i\widetilde{p}_j|)
=\cos(|p_1\widetilde{p}_i|)\cdot\cos(|p_1\widetilde{p}_j|)
+\sin(|p_1\widetilde{p}_i|)\cdot\sin(|p_1\widetilde{p}_j|)\cdot\cos(|x_ix_j|)$$
for each $i\not=j$.
Since $|p_1\widetilde{p}_i|=|p_1\widetilde{p}_j|$, we get
\begin{equation}\label{eq3.6}
|x_ix_j|\gs |\widetilde{p}_i\widetilde{p}_j|.
\end{equation}
By the assumption \eqref{eq3.2}, there exist points   $(\bar{x_i},\bar{t_i})\in \Sigma^n\ \big(=S(\Lambda)\big)$   such that $$|(x_i,t_i),(\bar{x_i},\bar{t_i})|=\pi$$
 for each $1\ls i\ls N_1$.
 By using the cosine law again, we have
 \begin{align*}
 -1=\cos(|(x_i,t_i)(\bar{x_i},\bar{t_i})|)
 &=\cos t_i\cdot\cos \bar{t_i}+\sin t_i\cdot\sin \bar{t_i}\cdot\cos (|x_i\bar{x_i}|)\\
 &=\cos(t_i+\bar{t_i})+\sin t_i\cdot\sin \bar{t_i}\cdot\big(\cos (|x_i\bar{x_i}|)+1\big)\\
 &\gs \cos(t_i+\bar{t_i}).
 \end{align*}
By  combining  with $0<t_i,\bar{t_i}<\pi$, we deduce
\begin{equation}\label{eq3.7}
|x_i\bar{x_i}|=\pi\quad {\rm and}\quad t_i+\bar{t_i}=\pi.
\end{equation}
By the  induction assumption and (\ref{eq3.3})--(\ref{eq3.7}), we know $\Lambda$ is isometric to
 $\mathbb S^{n-1}. $ Hence $\Sigma^n$ is isometric to $\mathbb S^n$.
\end{proof}
\be{lem}
{\label{lem3.3}{\rm(Perelman's concave function.)}\indent  Let $p\in M$. There exists a constant $r_1>0$
 and  a function $h: B_p(r_1)\to\R$ satisfying:\\
\indent (i)\indent $h$ is $(-1)$--concave;\\
\indent (ii)\indent$h$ is $2$-Lipschitz, that is, $h$ is Lipschitz continuous with a Lipschitz constant 2;\\
\indent (iii)\indent for each $x\in B_p(r_1)$,  we have
\begin{equation}\label{eq3.8}
\int_{\Sigma_x}d_xh(\xi)d\xi\ls0.
\end{equation}
Moreover, if $``="$ holds, then $x$ is regular.
}
\begin{proof}
Let us recall Perelman's construction in \cite{p94}. Fix a small $r_0>0$ and choose a maximal set of points
 $\{q_\alpha\}_{\alpha=1}^N\subset \partial B_p(r_0)$ with $\widetilde\angle q_\alpha pq_\beta>\delta$
  for $\alpha\not=\beta$, where $\delta$ is an arbitrarily (but fixed) small positive number $\delta\ll r_0$.
   By Bishop--Gromov volume comparison, there exists a constant $C_1$, which is independent of $\delta$, such that
\begin{equation}\label{eq3.9}
N\gs C_1\cdot \delta^{1-n}.
\end{equation}
 Consider the function
 $$h(y)=\frac{1}{N}\cdot\sum_{\alpha=1}^N\phi(|q_\alpha y|) $$
  on $B_p(r_1)$ with $0<r_1\ls\frac{1}{2}r_0$,
  where $\phi(t)$ is a real function with $\phi'(t)=1$ for $t\ls r_0-\delta$, $\phi'(t)=1/2$
  for $t\gs r_0+\delta$ and $\phi''(t)=-1/(4\delta)$ for $t\in(r_0-\delta,r_0+\delta).$

 The assertions  (i) and (ii) have been proved for some positive constant $r_1\ll r_0$ in \cite{p94}, (see also \cite{k02} for more details).
  The assertion (iii) is implicitly claimed in Petrunin's manuscript \cite{petu96}. Here we provide a proof as follows.

Let $x$ be a point near $p$.  It is clear that (\ref{eq3.8}) follows from Proposition \ref{prop3.1}
 and the above construction of $h$.  Thus we only need to consider the case of
 \begin{equation}\label{eq3.10}
 \int_{\Sigma_x}d_xh(\xi)d\xi=0.
 \end{equation}
  We want to show that $x$ is a regular point.

From $\angle q_\alpha pq_\beta \gs \widetilde\angle q_\alpha pq_\beta>\delta$ for $\alpha\not=\beta$ and
 the lower semi-continuity of angles (see Proposition 2.8.1 in \cite{bgp92}), we can assume $\angle q_\alpha xq_\beta\gs \delta/2$
  for $\alpha\not =\beta.$
  Proposition \ref{prop3.1} and (\ref{eq3.10}) imply that
  $$\int_{\Sigma_x}d_x{\rm dist}_{q_\alpha}(\xi)d\xi=0\qquad {\rm for\ each}\quad 1\ls \alpha\ls N.$$
Using Proposition \ref{prop3.1} again, we have
 \begin{equation}\label{eq3.11}
 \max_{\xi\in \Sigma_x}|\uparrow^{q_\alpha}_x\xi|=\pi\qquad {\rm for\ each}\quad 1\ls \alpha\ls N.
 \end{equation}
From Lemma \ref{lem3.2} and the arbitrarily small property of $\delta$, the combination of  (\ref{eq3.9})
 and (\ref{eq3.11}) implies  that $\Sigma_x$ is isometric to $\mathbb S^{n-1}$. Hence $x$ is regular.
\end{proof}

\section{Poisson equations and mean value inequality}
\subsection{Poisson equations}
Let $M$  be an $n$-dimensional  Alexandrov space and $\Omega$
 be a bounded domain in $M$. In \cite{kms01}, the canonical Dirichlet form $\mathscr E: W_0^{1,2}(\Omega)\times W_0^{1,2}(\Omega)\to \R$
  is defined by
 $$\mathscr E(u,v)=\int_\Omega\ip{\nabla u}{\nabla v}d{\rm vol}\qquad {\rm for}\ u,v\in W_0^{1,2}(\Omega).$$
 Given a function $u\in W_{loc}^{1,2}(\Omega)$,  we define  a functional $\mathscr L_u$ on $Lip_0(\Omega)$ by
$$\mathscr L_u(\phi)=-\int_\Omega\ip{\nabla u}{\nabla \phi}d{\rm vol},\qquad \forall \phi\in Lip_0(\Omega).$$

When a function $u$ is   $\lambda$-concave, Petrunin in \cite{petu09} proved that   $\mathscr L_u$ is signed Radon measure.
 Furthermore, if we write its Lebesgue's decomposition as
\begin{equation}
\label{eq4.1}\mathscr L_u=\Delta u\cdot{\rm vol}+\Delta^s u,
\end{equation}
 then  $\Delta^su\ls 0$ and
 \begin{equation}\label{eq4.2}
 \Delta u(p)=n\fint_{\Sigma_p}H_pu(\xi,\xi)d\xi\ls n\cdot\lambda
 \end{equation}
 for almost all points $p\in M$, where $H_pu$ is the Perelman's Hessian (see \eqref{eq2.6} or \cite{p-dc}).

 Nevertheless,  to study harmonic functions on Alexandrov spaces, we can not restrict our attention only on semi-concave functions.
  We have to consider the functional $\mathscr L_u$ for general functions  in $W^{1,2}_{loc}(\Omega).$

Let $f\in L^2(\Omega) $ and $u\in W_{loc}^{1,2}(\Omega)$. If the functional $\mathscr L_u$  satisfies
$$\mathscr L_u(\phi)\gs\int_\Omega f\phi d{\rm vol}\qquad \Big({\rm or}\quad \mathscr L_u(\phi)\ls\int_\Omega f\phi d{\rm vol}\Big)$$
 for all nonnegative $\phi\in Lip_0(\Omega)$, then, according to \cite{h89}, the functional $\mathscr L_u$ is a signed Radon  measure.
  In this case, $u$ is said to be a subsolution (supersolution, resp.) of Poisson equation $$\mathscr L_u=f\cdot{\rm vol}.$$

Equivalently,  $u\in W_{loc}^{1,2}(\Omega)$ is subsolution of $\mathscr L_u=f\cdot{\rm vol}$ if and only if
 it is a local minimizer of the energy
 $$\mathcal E(v)=\int_{\Omega'}\big(|\nabla v|^2+2fv\big)d{\rm vol}$$
  in the set of  functions $v$ such that $u\gs v$ and $u-v$ are  in $W^{1,2}_0(\Omega')$ for every fixed $\Omega'\Subset\Omega.$
   It is known (see for example \cite{k08}) that every continuous subsolution of $\mathscr L_u=0$ on $\Omega$ satisfies Maximum Principle,
    which states that $$\max_{x\in B}u\ls \max_{x\in\partial B}u$$ for any ball $B\Subset\Omega.$

 A function $u$ is a (weak) solution of Poisson equation $\mathscr L_u=f\cdot{\rm vol}$ on $\Omega$ if it is
 both a subsolution and a supersolution of the equation. In particular, a (weak) solution of $\mathscr L_u=0$ is called a harmonic function.

Now remark that $u$ is a (weak) solution of Poisson equation $\mathscr L_u=f\cdot{\rm vol}$ if and only if $\mathscr L_u$ is a signed Radon measure
 and its Lebesgue's decomposition $\mathscr L_u=\Delta u\cdot \rv+\Delta^su$ satisfies
$$\Delta u=f\qquad{\rm and}\qquad\Delta^su=0.$$

 Given a function $f\in L^2(\Omega)$ and $g\in W^{1,2}(\Omega)$, we can solve  Dirichlet problem of the equation
  \begin{equation*}
 \begin{cases}\mathscr L_u&=f\cdot{\rm vol}\\ u&=g|_{\partial \Omega}.\end{cases}
 \end{equation*}
  Indeed, by Sobolev compact embedding theorem (see \cite{hk00,kms01}) and a standard argument
   (see, for example, \cite{gt01}),  it is known that the solution of Dirichlet problem exists and  is unique  in
    $W^{1,2}(\Omega).$ (see, for example, Theorem 7.12 and Theorem 7.14 in \cite{c99}.) Furthermore, if we add the
     assumption $f\in L^s$ with $s>n/2$, then the solution is locally H\"older continuous  in $\Omega$ (see \cite{kshan01,kms01}).

\begin{defn}
A function $u\in C(\Omega)\cap W^{1,2}_{loc}(\Omega)$ is called a $\lambda$-superharmonic (or $\lambda$-subharmonic, resp.)
 on $\Omega$, if it satisfies the following comparison property: for every open subset $\Omega'\Subset \Omega$, we have
  $$\widetilde{u}\ls u,\qquad ({\rm or}\quad \widetilde{u}\gs u,  {\rm resp.}),$$
   where $\widetilde{u}$ is the (unique) solution of the equation  $\mathscr L_{\widetilde{u}}=\lambda\cdot\rv$
    in $\Omega'$ with boundary value $\widetilde{u}= u$ on $\partial \Omega'.$

In particular, a $0$-superharmonic (or, $0$-subharmonic, resp.) function is simply said a superharmonic (or, subharmonic, resp.) function.
\end{defn}

In partial differential equation theory, this definition is related to the  notion of viscosity solution (see \cite{c89}).

According to the maximum principle, we know that a continuous supersolution of $\mathscr L_u =0$ must be a superharmonic function.
 Notice that  the converse is not true in general metric measure space (see \cite{km02}). Nevertheless, we will prove a
  semi-concave superharmonic function on $M$ must be a supersolution   of $\mathscr L_u =0$ (see Corollary \ref{cor4.6} below).

\subsection{Mean value inequality for solutions of Poisson equations}

Let  $u\in W^{1,2}(\Omega)$  such  that $\mathscr L_u$ is a signed Radon measure on $\Omega$ and $A\Subset \Omega $ be an open set.
   We define a functional
$I_{u,A} $  on $W^{1,2}(A)$ by
\begin{equation}\label{eq4.3}
I_{u,A}(\phi)=\int_A\ip{\nabla u}{\nabla \phi}d\rv+\int_A\phi d\mathscr L_u.
\end{equation}
\be{rem}
{\label{rem4.2}(i)\  If $\phi_1,\phi_2\in W^{1,2}(A)$ and $\phi_1-\phi_2\in W_0^{1,2}(A)$, then, by the definition of
 $\mathscr L_u$, we have $I_{u,A}(\phi_1)=I_{u,A}(\phi_2).$\\
(ii)\  If $M$ is a smooth manifold and $\partial A$ is smooth, then $I_{u,A}(\phi)=\int_A{\rm div}(\phi\nabla u)d\rv.$
}
\be{lem}
{\label{lem4.3}
Let $0<r_0<R_0$ and  $w(x)=\varphi(|px|)$ satisfy $\mathscr L_w\gs0$ on some neighborhood of $B_p(R_0)\backslash B_p(r_0)$,
 where $\varphi\in C^2(\R)$. Consider a function
 $v\in W^{1,2}(B_p(R_0)\backslash \overline{B_p(r_0)})\cap L^\infty(\overline{B(p,R_0)}\backslash B(p,r_0)).$
  Then for almost all $r,R\in(r_0,R_0)$, we have
   $$I_{w,A}(v)=\varphi'(R)\int_{\partial B_p(R)}vd\rv-\varphi'(r)\int_{\partial B_p(r)}vd\rv,$$
    where $A=B_p(R)\backslash\overline{B_p(r)}.$
}
\begin{proof}Since $\mathscr L_w$ is a signed Radon measure, we have $\mathscr L_w\big(B_p(R_0)\backslash B_p(r_0)\big)<+\infty.$
 Hence, for almost all $r,R\in (r_0,R_0)$,  $\mathscr L_w(A_j\backslash A)\to0$ as $j\to\infty,$
  where $A_{j}= B_p(R+\frac{1}{j})\backslash B_p(r-\frac{1}{j})$.     Now let us fix such  $r$ and $R$.

Let $ v_j=v\cdot \eta_j(|px|)\in W^{1,2}_0(D),$ where $D=B_p(R_0)\backslash\overline{B_p(r_0)}\ $  and  \begin{align*}
\eta_j(t):=\begin{cases}1&{\rm if}\quad t\in[r,R]\\ j\cdot(t-r)+1&{\rm if}\quad t\in[r-\frac{1}{j},r]\\
 -j\cdot(t-R)+1&{\rm if}\quad t\in[R,R+\frac{1}{j}]\\
 0&{\rm if}\quad t\in (-\infty,r-\frac{1}{j})\cup(R+\frac{1}{j}, \infty).\end{cases}
\end{align*}
By the definitions of $I_{w,A}(v)$ and $\mathscr L_w$, we have
\begin{align}\label{eq4.4}
I_{w,A}(v)&=\int_D\ip{\nabla w}{\nabla v_j}d\rv-\int_{D\backslash A}\ip{\nabla w}{\nabla v_j}d\rv+\int_Dv_jd\mathscr L_w-\int_{D\backslash A}v_jd\mathscr L_w\\ \nonumber&=-\int_{D\backslash A}v\ip{\nabla w}{\nabla \eta_j}d\rv-\int_{D\backslash A}\eta_j\ip{\nabla w}{\nabla v}d\rv-\int_{D\backslash A}v_jd\mathscr L_w\\ \nonumber&:=-J_1-J_2-J_3.
\end{align}
Notice that
\begin{equation*}  |J_2|\ls \int_{A_j\backslash A}|\nabla w|\cdot|\nabla v|d\rv  \quad {\rm and}\quad|J_3|\ls \mathscr L_w(A_j\backslash A)\cdot|v|_{L^\infty(D)},\end{equation*} Hence we have $J_2\to 0$ and $J_3\to0$ as $j\to\infty.$
\begin{equation}
\label{eq4.5}
J_1=j\cdot\int_{B_p(r)\backslash B_p(r-1/j)}v\varphi'd\rv-j\cdot\int_{B_p(R+1/j)\backslash B_p(R)}v\varphi'd\rv.
\end{equation}

The assumption  $v\in L^\infty(D)$ implies the function $h(t)=\int_{B_p(t)}vd\rv$ is Lipschitz continuous in $(r_0,R_0)$.
 Indeed, for each $r_0<s<t<R_0$,
$$|h(t)-h(s)|\ls \int_{B_p(t)\backslash B_p(s)}|v|d\rv\ls |v|_{L^\infty}\cdot \rv\big(B_p(t)\backslash B_p(s)\big)\ls c\cdot(t^n-s^n),$$
 where constant $c$ depends only on $R_0,\ n$ and the lower bounds of curvature on $B_p(R_0)$.
 Then $h(t)$ is differentiable almost all $t\in(r_0,R_0)$. By co-area formula, we have
$$h'(t)=\int_{\partial B_p(t)}vd\rv$$for almost all $t\in(r_0,R_0).$

Without loss of generality, we may assume that $r$ and $R$ are differentiable points of function $h$. Now
\begin{equation*}
\begin{split}\Big|j\int_{B_p(r)\backslash B_p(r-1/j) }\varphi'vd\rv&-\varphi'(r)\cdot j\Big(h(r)-h(r-1/j)\Big)\Big|\\
&\ls \int_{B_p(r)\backslash B_p(r-1/j)}\max |\varphi''|\cdot|v|d\rv\to 0
\end{split}\end{equation*}
as $j\to\infty.$ The similar estimate also holds for  $ j\int_{B_p(R+1/j)\backslash B_p(R) }\varphi'vd\rv $. Therefore,
\begin{equation*}\begin{split}\lim_{j\to\infty} J_1&=\lim_{j\to\infty}\varphi'(r)\cdot j\Big(h(r)-h(r-1/j)\Big)
-\lim_{j\to\infty}\varphi'(R)\cdot j\Big(h(R+1/j)-h(R)\Big)\\
&=\varphi'(r)h'(r)-\varphi'(R)h'(R).\end{split}
\end{equation*}
 By combining this and (\ref{eq4.4}),  we get the desired assertion.
\end{proof}

 If $M$ has $Ric\gs (n-1)k$, then for a distance function $dist_p(x):=|px|$,
Laplacian comparison (see Theorem 1.1 and Corollary 5.9 in \cite{ks07}) asserts  that  $\mathscr L_{dist_p}$
 is a signed Radon measure and
 $$\mathscr L_{dist_p}\ls (n-1)\cdot\cot_k\circ dist_p\cdot{\rm vol}\qquad {\rm on}
\quad M\backslash\{p\}.$$
 Moreover, $G(x):=\phi_k(|px|)$ is defined on $M\backslash\{p\}$ and
 $$\mathscr L_G\gs0 \qquad {\rm on}\quad M\backslash\{p\},$$
 where  $\phi_k(r)$ is the real value function such that $\phi\circ dist_o$ is the Green function on
 $\mathbb M^n_k$ with singular
point $o$. That is, if $n\gs3$,
 \begin{equation}
\label{eq4.6}\phi_k(r)=\frac{1}{(n-2)\cdot\omega_{n-1}}\int_r^{\infty}s^{1-n}_k(t)dt,
\end{equation} where $\omega_{n-1}={\rv}
(\mathbb S^{n-1})$ and
 \begin{equation*}s_k(t)=\begin{cases}\sin(\sqrt kt)/\sqrt k& \quad k>0\\t&\quad k=0\\ \sinh(\sqrt{- k}t)/\sqrt{- k}& \quad k<0.\end{cases}
\end{equation*}
If $n=2$, the function $\phi_k$ can be given similarly.

By applying the Lemma \ref{lem4.3} to  function $G$, we have the following mean value inequality for nonnegative supersolution
 of Poisson equation.
\be{prop}{\label{prop4.4}
Let $M$ be an $n$--dimensional Alexandrov space with $Ric\gs (n-1)k$ and $\Omega$ be a bounded domain in $M$.
 Assume that $f\in L^\infty(\Omega)$ and $u$ is a continuous, nonnegative supersolution of Poisson equation
  $\mathscr L_u=f\cdot{\rm vol}$ on $\Omega$. Then for any ball $B_p(R)\Subset \Omega,$ we have \begin{align}
\label{eq4.7}\frac{{\rm vol}(\Sigma_p)}{\omega_{n-1}}&\bigg(\frac{1}{H^{n-1}(\partial B_o(R)\subset T^k_p)}\int_{\partial B_p(R)}ud\rv- u(p)\bigg)\\
\nonumber&\ls(n-2)\cdot\int_{B^*_p(R)}Gfd{\rm vol}-(n-2)\cdot\phi_k(R)\int_{B_p(R)}fd{\rm vol},
\end{align}
where $B^*_p(R)=B_p(R)\backslash\{p\}$ and $T_p^k$ is the $k$-cone over $\Sigma_p$ (see \cite{bbi01} p. 354).}
\begin{proof}
For simplicity, we only give a proof for the case $n\gs3$.  A slight modification of the argument will prove the case $n=2.$\\

\emph{Case 1:} Assume that $u$ is a solution of $\mathscr L_u=f\cdot\rv$.

Let $G(x)=\phi_k(|px|)$, where the real function $\phi_k$ is chosen such that $\phi_k(|ox|)$ is the Green function on $\M^n_k$
 with singular point at $o$. Then, by Laplacian comparison theorem (see \cite{ks07} or \cite{zz10}),
  the signed Radon measure $\mathscr L_G$ is nonnegative on $M\backslash\{p\}$.

Since $u$ is continuous on $B_p(R)$, the function $h(s)=\int_{B_p(s)}ud{\rm vol}$ is Lipschitz. From Lemma \ref{lem4.3}, we have $$I_{G,A}(u)=\phi_k'(t)\cdot h'(t)-\phi_k'(s)\cdot h'(s)$$ for almost all $s,t\in (0,R)$ with $s<t$, where
$A=B_p(t)\backslash\overline{B_p(s)}.$
By the definition of supersolution of Poisson equation, we have
$$I_{G,A}(u)-I_{u,A}(G)=\int_Aud\mathscr L_G-\int_AGd\mathscr L_u\gs-\int_AGfd{\rm vol}.$$
On the other hand, letting
 \begin{align*}\bar G(x)=\begin{cases}G(x)&{\rm if}\quad s\ls|px|\ls t\\
\phi_k(t)&{\rm if}\quad |px|> t\\ \phi_k(s)&{\rm if}\quad |px|< s,\end{cases}\end{align*}
we have \begin{align*}\int_A\ip{\nabla G}{\nabla u}&=\int_{B_p(t)}\ip{\nabla (\bar G-\phi_k(t))}{\nabla u}-\int_{B_p(s)}\ip{\nabla (\bar G-\phi_k(s))}{\nabla u}\\&=-\int_AGd{\mathscr L_u}+\phi_k(t)\int_{B_p(t)}d{\mathscr L_u}-\phi_k(s)\int_{B_p(s)}d{\mathscr L_u}.
\end{align*}
  Hence, by $\mathscr L_u=f\cdot\rv$,
   $$I_{u,A}(G)=\phi_k(t)\int_{B_p(t)}fd{\rm vol}-\phi_k(s)\int_{B_p(s)}fd{\rm vol}.$$
   If we set
   $$\psi(\tau)=\phi_k'(\tau)\cdot h'(\tau)-\phi_k(\tau)\int_{B_p(\tau)}fd{\rm vol},$$
   then  the function
   $$\psi(\tau)+\int_{B_p^*(\tau)}Gfd{\rm vol}$$
   is nondecreasing  with respect to $\tau$ (for almost all $\tau\in (0,R)$). Indeed,   for almost all $s<t$,
   $$\psi(t)+\int_{B_p^*(t)}Gfd{\rm vol}-\psi(s)-\int_{B_p^*(s)}Gfd{\rm vol}=I_{G,A}(u)-I_{u,A}(G)+\int_AGfd\rv\gs0.$$
 Thus by
 $$\phi'_k(t)= -s^{1-n}_k(t)\cdot\frac{1}{(n-2)\omega_{n-1}}
 =-\frac{1}{n-2}\cdot\frac{\rv(\Sigma_p)}{\omega_{n-1}}\cdot\frac{1}{H^{n-1}(\partial B_o(t)\subset T_p^k)},$$ we have \begin{align*}\phi_k'(t)h'(t)-\phi_k(t)\int_{B_p(t)}fd\rv+ \int_{B_p^*(t)}Gf\rv&\gs \lim_{s\to0}\Big(\psi(s)+\int_{B_p^*(s)}Gfd\rv\Big)\\&=-\frac{1}{n-2}
 \cdot\frac{{\rm Vol}(\Sigma_p)}{\omega_{n-1}}u(p)\end{align*}
By combining this
  and $h'(s)= \int_{\partial B_p(r)}ud\rv$ a.e. in $(0,R)$, we obtain that  (\ref{eq4.7}) holds for almost all $r\in(0,R)$.

By combining the Bishop--Gromov inequality on spheres (see \cite{bbi01} or Lemma 3.2 of \cite{ks07}), the assumption
  $u\gs0$ and the continuity of $u$,  we have  \begin{equation}
\label{eq4.8}\liminf_{r\to R^-}\int_{\partial B_p(r)}ud\rv\gs \int_{\partial B_p(R)}ud\rv.
\end{equation} Therefore, we get the desired result for this case.\\

\emph{Case 2:}   $u$ is a supersolution of $\mathscr L_u=f\cdot\rv$.

For each $R>0$, let $ \widetilde{u}$ be a solution of $\mathscr L_{  \widetilde{u}}=f\cdot\rv$ on $B_p(R)$ with boundary value condition
$  \widetilde{u}=u$ on $\partial B_p(r)$. Since $\mathscr L_{\widetilde{u}-u}\gs0$, by maximal principle, we get $\widetilde u(p)\ls u(p)$. Therefore, by applying  $Case\ 1 $ to $\widetilde{u}$, we get the desired estimate.
\end{proof}

\be{cor}
{\label{cor4.5}
Let $M$, $\Omega$, $u$ and $f$ be as above.  If $p$ is a Lebesgue point of  $f$, i.e.,
\begin{equation}
\label{eq4.9}\fint_{B_p(R)}fd\rv=f(p)+o(1),
\end{equation}
then
$$\frac{1}{H^{n-1}(\partial B_o(R)\subset T^k_p)}\int_{\partial B_p(R)}u(x)d\rv\ls u(p)+\frac{ f(p)}{2n}\cdot R^{2}+o(R^{2}).$$
}
\begin{proof}
By using \eqref{eq4.7}, we have  \begin{equation}
\label{eq4.10} \frac{1}{H^{n-1}(\partial B_o(R)\subset T^k_p)}\int_{\partial B_p(R)}ud\rv-u(p)\ls (n-2)\cdot\frac{\omega_{n-1}}{\rv(\Sigma_p)}\cdot\varrho(R),
\end{equation}
where
\begin{equation*}\begin{split}
\varrho(R)&=\int_{B^*_p(R)}Gfd{\rm vol}-\phi_k(R)\int_{B_p(R)}fd{\rm vol}\\
&=\int_0^R\int_{\partial B_p(s)}\phi_k(s)f-\phi_k(R)\int_0^R\int_{\partial B_p(s)}f.
\end{split}\end{equation*}
 Hence, by (\ref{eq4.9}), we have
 \begin{align*}\varrho'(R)&=-\phi_k'(R)\int_{B_p(R)}fd\rv\\ \nonumber&=
\frac{\rv(\Sigma_p)}{(n-2)\omega_{n-1}}\cdot\frac{\int_0^Rs_k^{n-1}(r)dr}{s_k^{n-1}(R)}
\cdot\frac{\rv(B_p(R))}{H^n(B_o(R)\subset T^n_k)}\fint_{B_p(R)}fd\rv\\
\nonumber&=\frac{\rv(\Sigma_p)}{(n-2)\omega_{n-1}}\cdot\big(\frac{R}{n}+o(R)\big)\cdot(1+o(1))\cdot(f(p)+o(1))\\
\nonumber&=\frac{\rv(\Sigma_p)}{n(n-2)\omega_{n-1}}f(p)\cdot R+o(R).\end{align*}
Hence, noting that $\rho(0)=0$, we get
\begin{equation}\label{eq4.11}
\rho(R)=\frac{\rv(\Sigma_p)}{2n(n-2)\omega_{n-1}}f(p)\cdot R^2+o(R^2).
\end{equation}
Therefore, the desired result follows from (\ref{eq4.10}) and (\ref{eq4.11}).
\end{proof}
\be{cor}
{\label{cor4.6}
Let $M$ be an $n$-dimensional Alexandrov space with $Ric\gs (n-1)k$ and $\Omega$ be a bounded domain in $M$.
 Let $u$ be a semi-concave function on $\Omega$ and $f\in L^\infty(\Omega)$.
Then $u$ is a supersolution of $\mathscr L_u=f\cdot\rv$ provided it satisfies the property:
for each  point $p\in Reg_u$ and every sufficiently small ball $B_p(R)\Subset \Omega$, we have
\be{equation}
{\label{eq4.12}\widetilde{u}_R-u\ls0,}
where the function $\widetilde{u}_R$ is the (unique) solution of Dirichlet problem :
\begin{equation*}\begin{cases}
\mathscr L_{\widetilde{u}_R}=f\cdot\rv&\quad {\rm in}\ B_p(R)\\ \widetilde{u}_R=u&\quad {\rm on}\ \ \partial B_p(R).\end{cases}
\end{equation*}

In particular, a semi-concave superharmonic function  must be a supersolution of  the equation $\mathscr L_u=0.$}
 \begin{proof}Since the singular part of $\mathscr L_u$ is non-positive,
we need only to consider its absolutely continuous part $\Delta u\cdot\rv$.

Fix a point $p\in Reg_u$ such that (\ref{eq4.2}) holds and $p$ is a Lebesgue point of  $f$. Since the set of such points has
full measure in $\Omega$, we need only to show that $\Delta u(p)\ls f(p).$

We set $$g_R(x)=u(x)-\min_{x\in\overline{B_p(R)}}\widetilde{u}_R(x)\quad {\rm and }\quad \widetilde{g}_R(x)=\widetilde{u}_R(x)-\min_{x\in\overline{B_p(R)}}\widetilde{u}_R(x).$$
Then $\widetilde g_R\ls g_R$ and  $\widetilde g_R|_{\partial B_p(R)}\ls g_R|_{\partial B_p(R)}$.
 Noting that the functions $\widetilde{g}_R$ is nonnegative and $\mathscr L_{\widetilde {g}_R}=f\cdot\rv$.
 By  Corollary \ref{cor4.5} and  $p$ is regular, we have
\begin{align}\label{eq4.13}
\int_{\partial B_p(R)}g_R& =\int_{\partial B_p(R)}\widetilde{g}_R \ls H^{n-1}(\partial B_o(R)\subset T_p^k)\cdot\Big( \widetilde{g}_R(p)+\frac{f(p)}{2n}R^{2} +o(R^{2})\Big)\\&\ls\nonumber g_R(p)\cdot H^{n-1}(\partial B_o(R)\subset T_p^k)+\frac{f(p)}{2n}R^{n+1}\cdot\omega_{n-1}+o(R^{n+1}).
\end{align}
On the other hand, since $p\in Reg_{g_R}$, from (\ref{eq2.15}) and (\ref{eq4.2}), we have
\begin{equation}\label{eq4.14}
\int_{\partial B_p(R)}g_R= g_R(p)\cdot\rv(\partial B_p(R))+\frac{\Delta g_R(p)}{2n}R^{2}\cdot\rv(\partial B_p(R))+o(R^{n+1})
\end{equation}
for almost all $R\in(0,\delta_0)$, where $\delta_0$ is a small positive number.
Because  $p$ is a smooth point, Lemma \ref{lem2.1} implies
\begin{equation}\label{eq4.15}
H^{n-1}(\partial B_o(R)\subset T_p^k)-\rv(\partial B_p(R))=
 o(R^{n})\end{equation}
for almost all $R\in(0,\delta_0)$.

Now we want to show $g_R(p)=O(R)$. Noticing  that  (\ref{eq4.12}) and the fact that $u$ is locally Lipshitz
 (since $u$ is semi-concave), we have
 \begin{equation}
 \begin{split}\label{eq4.16}
0\ls g_R(p)&=u(p)-\min_{x\in\partial \overline{B_p(R)}}\widetilde{u}_R(x)+\min_{x\in\partial \overline{B_p(R)}}\widetilde{u}_R(x)
-\min_{x\in \overline{B_p(R)}}\widetilde{u}_R(x)\\
    & \ls C_1R+\min_{x\in\partial \overline{B_p(R)}}\widetilde{u}_R(x)-\min_{x\in \overline{B_p(R)}}\widetilde{u}_R(x).
    \end{split}
    \end{equation}
  Since $R$ is sufficiently small, there exists the Perelman concave function $h$ on $B_p(2R)$ given in Lemma \ref{lem3.3}.
   We have
  $$\mathscr L_{\widetilde{u}_R+\|f\|_{L^\infty}\cdot h}\ls\mathscr L_{\widetilde{u}_R} -\|f\|_{L^\infty}\ls  0.$$
Hence, by applying maximal principle, we have for any  point $x\in \overline{B_p(R)}$,
\begin{equation*}\begin{split}
  \widetilde{u}_R(x)+\|f\|_{L^\infty}h(x)&\gs \min_{x\in\partial \overline{B_p(R)} }\Big(\widetilde{u}_R(x)+\|f\|_{L^\infty}h(x)\Big)\\ &\gs \min_{x\in\partial \overline{B_p(R)} }\widetilde{u}_R(x)+\|f\|_{L^\infty}\min_{x\in\partial \overline{B_p(R)} }h(x).
  \end{split}
  \end{equation*}
Since $h$ is Lipschitz continuous, this implies that
$$\min_{x\in\partial \overline{B_p(R)}}\widetilde{u}_R(x)
-\widetilde{u}_R(x)\ls \|f\|_{L^\infty}\big(h(x)-\min_{x\in\partial \overline{B_p(R)} }h(x)\big)\ls C_2R$$
for any point $x\in \overline{B_p(R)}$.
The combination of this and \eqref{eq4.16} implies \begin{equation}\label{eq4.17} g_R(p)=O(R).\end{equation}
By combining  (\ref{eq4.13})--(\ref{eq4.15}) and \eqref{eq4.17}, we have
$$\frac{\Delta g_R(p)}{2n}R^2\cdot\rv(\partial B_p(R))- \frac{ f(p)}{2n}\omega_{n-1}R^{n+1}\ls O(R)\cdot o(R^n)+o(R^{n+1})=o(R^{n+1})$$
for almost all $R\in(0,\delta_0)$. Hence, $\Delta g_R(p)\ls f(p)$.
Therefore, $\Delta u(p)\ls f(p),$ and the proof of the corollary is completed.
\end{proof}

\subsection{Harmonic measure}

In this subsection, we basically follow Petrunin in \cite{petu96} to consider harmonic measure.

 \be{lem}
{{\rm (Petrunin \cite{petu96})}\label{lem4.7} \indent Let $M$ be an $n$--dimensional Alexandrov space with $Ric\gs (n-1)k$
 and $\Omega$ be a bounded domain in $M$. If $u$ is a nonnegative harmonic function on $\Omega$, then for any ball $B_p(R)\Subset\Omega$, we have
\begin{equation}\label{eq4.18}u(p)\gs \frac{1}{\rv(\Sigma_p)\cdot s_k^{n-1}(R)}\int_{\partial B_p(R)}ud\rv.\end{equation}
}
\begin{proof}By the definition, $u$ is harmonic if and only if it is a solution of equation $\mathscr L_u=0.$ Now the result follows from (\ref{eq4.7}) with $f=0.$
\end{proof}

Consider an $n$-dimensional Alexandrov space $M$ and a ball $B_p(R)\subset M$.
In order to define a new measure $\nu_{p,R}$ on $B_p(R),$  according to \cite{h89}, we need only to define a positive functional on $Lip_0(B_p(R))$.

Now fix a  nonnegative function $\varphi\in Lip_0(B_p(R))$. First we define a function $\mu:(0,R)\to\R$ as follows:
 for each $r\in(0,R)$, define $$\mu(r):=u_r(p),$$
 where $u_r$ is the (unique) solution of Dirichlet problem $\mathscr L_{u}=0$ in $B_p(r)$ with boundary value $u=\varphi$
  on $\partial B_p(r).$
\begin{lem}\label{lem4.8}There exists $R>0$ such that $\mu(r)$ is continuous in $(0,R).$\end{lem}
\begin{proof}
From Lemma 11.2 in \cite{bgp92}, we know that there exists $R>0$ such that, for all $x\in B_p(R)\backslash\{p\}$,
 we can find a point $x_1$ satisfying
  $$\wa pxx_1>\frac{99}{100}\pi \quad {\rm and}\quad|px_1|\gs 2|px|.$$
In particular, this implies, for each $r\in(0,R)$, that $B_p(r)$ satisfies an exterior ball condition in the following sense:
 there exists $C>0$ and $\delta_0>0$ such that for all $x\in \partial B_p(r)$ and $0<\delta<\delta_0$,
  the set $B_{x}(\delta)\backslash B_p(r)$ contains a ball with radius $C\delta.$ Indeed, we can choose $x_2$ in geodesic $xx_1$ with $|xx_2|=\delta/3$ (with $\delta\ls r/10$). The monotonicity  of comparison angles says that $\wa pxx_2\gs \wa pxx_1\gs \frac{99}{100}\pi$.
   This concludes $|px_2|\gs |px|+\delta/6$. Therefore, $B_{x_2}(\delta/6)\subset B_x(\delta)\backslash B_p(r).$

Since $\varphi$ is Lipschitz continuous on $B_p(r)$, Bj\"orn in \cite{bj02} (see Remark 2.15 in \cite{bj02}) proved that $u_r$ is H\"older continuous on $\overline{B_p(r)}$.

For any $0<r_1<r_2<R$, by using maximum principle, we have
$$|\mu(r_1)-\mu(r_2)|\ls \max_{x\in \partial B_p(r_1)}|u_{r_1}(x)-u_{r_2}(x)|=\max_{x\in \partial B_p(r_1)}|\varphi(x)-u_{r_2}(x)|.$$
 By combining with the H\"older continuity of $\varphi$ and $u_{r_2}$, we have that $|\mu(r_1)-\mu(r_2)|\to0$ as $r_2-r_1\to0^+$.
  Hence $\mu(r)$ is continuous.
\end{proof}
\begin{rem}\label{rem4.9}
If $p$ is a regular point, then the constant $R$ given in Lemma \ref{lem4.8} can be chosen uniformly in a neighborhood of $p$.

 Indeed, there exists a neighborhood of $p$, $B_p(R_0)$ and a bi-Lipschitz homeomorphism $F$ mapping $B_p(R_0)$ to an open domain
  of $\mathbb R^n$ with bi-Lipschitz constant $\ls 1/100$. Then for  each ball $B_{q}(r)\subset B_p(R_0/4)$ with $r\ls R_0/4$ and
   $x\in\partial B_q(r)$,  let $x'\in \mathbb R^n$ such that
$$|x'F(x)|=|F(q)F(x)|\quad {\rm and}\quad|x'F(q)|=|F(q)F(x)|+|F(x)x'|=2|F(q)F(x)|,$$
we have
$$|qF^{-1}(x')|\gs \frac{99}{100}|F(q)x'|
\gs 2\Big(\frac{99}{100}\Big)^2|xq|\qquad {\rm and}\qquad |x'F(x)|\ls \Big(\frac{101}{100}\Big)^2|xq|.$$
Hence, it is easy check that $B_q(r)$ satisfies  an exterior ball condition as above (the similar way as above).
Therefore,   the constant $R$ in Lemma \ref{lem4.8} can be choose $R_0/4$ for all $q\in B_p(R_0/4).$
\end{rem}

 Now we can define the functional $\nu_{p,R}$ by
 $$\nu_{p,R}(\varphi)=\frac{\int_0^Rs_k^{n-1}\mu(r)dr}{\int_0^Rs_k^{n-1}dr}.$$
From Lemma \ref{lem4.7}, we have $\mu_r\gs0$, and   $\nu_{p,R}(\varphi)\gs0.$ Hence, it provides a Radon measure on $B_p(R).$
 Moreover, it is a  probability measure and by \eqref{eq4.18},
\begin{equation}\label{eq4.19}
\nu_{p,R}\gs \frac{\rv}{H^n(B_o(R)\subset T^k_p)}.
\end{equation}

Let $u$ be a harmonic function on $\Omega$. Then for any ball $B_p(R)\Subset \Omega$, we have
 \begin{equation}\label{eq4.20}
 u(p)=\int_{B_p(R)}u(x)d\nu_{p,R}.
 \end{equation}

The following strong maximum principle was proved in an abstract framework of Dirichlet form by Kuwae in \cite{k08} and
 Kuwae--Machiyashira--Shioya in \cite{kms01}. In metric  spaces supporting a doubling measure and a Poincar\'e inequality,
  it was proved by Kinnunen--Shanmugalingan in \cite{kshan01}. Here, by (\ref{eq4.20}), we give a short proof  in Alexandrov spaces.
\begin{cor}\label{cor4.10}{\rm (Strong Maximum Principle)}\indent Let $u$ be a subharmonic function on a bounded and connected
  open domain $\Omega$. Suppose there exists a point $p\in \Omega$ for which  $u(p)=\sup_{x\in \Omega}u.$ Then  $u$ is  constant.
\end{cor}
\begin{proof}Firstly, we consider $u$ is harmonic. By (\ref{eq4.19})--(\ref{eq4.20}) and  that $\nu_{p,R}$ is a probability measure,
 we have $u(x)=u(p)$ in some neighborhood $B_p(R)$. Hence the set $\{x\in\Omega:\ u(x)=u(p)\}$ is open. On the other hand,
  the continuity of $u$ implies that the set is close. Therefore, it is $\Omega$ and $u$ is a constant in $\Omega$.

If $u$ is a subharmonic function, the result follows from the definition of subharmonic and the above harmonic case.\end{proof}

The following lemma appeared in \cite{petu96} (Page 4). In this lemma, Petrunin constructed an auxiliary function, which is
 similar to Perelman's concave function.
\be{lem}
{\label{lem3.4}{\rm (Petrunin \cite{petu96})}\indent For any point $p\in M$,  there exists a neighborhood $B_p(r_2)$ and
 a function $h_0: B_p(r_2)\to\R$ satisfying:\\
\indent (i)\indent $h_0(p)=0$;\\
\indent (ii)\indent $\mathscr L_{h_0}\gs 1\cdot \rv$ on $B_p(r_2)$;\\
\indent (iii)\indent there are $0<c<C<\infty$ such that $$c\cdot |px|^2\ls h_0(x)\ls C\cdot |px|^2.$$
}
\begin{proof} A sketched proof was given in \cite{petu96}. For the completeness, we present a  detailed proof  as follows.

Without loss of generality, we may assume $M$ has curvature $\gs-1$ on a neighborhood of $p$.  Fix a small real number
 $r>0$ and  set
 \begin{equation*}\phi(t)=\begin{cases}a+bt^{2-n}+t^2&\quad t\ls r\\
 0&\quad t>r,\end{cases}
 \end{equation*}
where $a=-\frac{n}{n-2}r^2$ and $b=\frac{2}{n-2}r^n.$

Take a minimal set of points $\{q_\alpha\}_{\alpha=1}^{N}$ such that $|pq_\alpha|=r$ and $\min_{1\ls \alpha\ls N}\angle(\xi, \uparrow_p^{q_\alpha})\ls\pi/10$ for each direction $\xi\in \Sigma_p$. Consider
$$h_0(x)=\sum_{\alpha=1}^Nh_\alpha$$
 where $h_\alpha=\phi(|q_\alpha x|).$  Clearly, $h_0(p)=0$. Bishop--Gromov volume comparison of $\Sigma_p$ implies that $N\ls c(n)$,
  for some constant  depending only on  the dimension $n$.

 Fix any small $0<\delta\ll r$. For each $x\in B_p(\delta)\backslash\{p\}$, there is some $q_\alpha$ such that $\angle(\uparrow^x_p, \uparrow_p^{q_\alpha})\ls\pi/10$. When $\delta$ is small, the comparison angle $\wa x q_\alpha p$ is small. Then $\wa q_\alpha xp\gs \frac{3}{4}\pi$. This implies that $ |\nabla_x dist_{q_\alpha}|\gs 1/\sqrt2,$ when $\delta$ is sufficiently  small.

Fix any $\alpha$. Since the function $-h_\alpha$ is semi-concave near $p$, the singular part of $\mathscr L_{h_\alpha}$ is nonnegative.
   We only need to consider the absolutely continuous part $\Delta h_\alpha$. By Laplacian comparison theorem (see \cite{zz10} or \cite{ks07})
    and a  direct computation, we have $\Delta h_\alpha(x)\gs-C\delta$ a.e. in $B_p(\delta)$ and $\Delta h_\alpha(x)\gs n-C\delta$
     at almost all points $x$ with $|\nabla_x dist_{q_\alpha}|\gs 1/\sqrt2,$ where $C$ denoted the various positive constants depending
      only on $n$ and $r$.  Indeed, since $r-\delta \ls|q_\alpha x|\ls r+\delta$,
\begin{equation*}\begin{split}
\Delta h_\alpha(x)&=\phi'(|q_\alpha x|)\cdot\Delta dist_{q_\alpha}(x)+\phi''(|q_\alpha x|)|\nabla dist_{q_\alpha}|^2\\
&=2|q_\alpha x|\cdot\Big(1-\frac{r^n}{|q_\alpha x|^n} \Big)\cdot \Delta dist_{q_\alpha}(x)
+2\Big(1+(n-1)\frac{r^n}{|q_\alpha x|^n} \Big)\cdot |\nabla_x dist_{q_\alpha}|^2\\
&\gs 2|q_\alpha x|\cdot\Big(1-\frac{r^n}{|q_\alpha x|^n} \Big)\cdot \Big( \frac{n-1}{|q_\alpha x|}+C|q_\alpha x| \Big)\\
&\qquad+2\Big(1+(n-1)\frac{r^n}{|q_\alpha x|^n} \Big)\cdot |\nabla_x dist_{q_\alpha}|^2\\
&\gs -C_n \frac{\delta}{r}+2n\cdot|\nabla_x dist_{q_\alpha}|^2.
\end{split}\end{equation*}

 On the other hand,  at the points $x$ where $\angle(\uparrow_p^x, \uparrow_p^{q_\alpha})\ls \pi/10 $ and $|px|\ls|pq_\alpha|/10,$
   we have
 $$r-|px|\ls |q_\alpha x|\ls r-|px|/2.$$
 Hence,  by applying $\phi'(r)=0$ and $2n\ls \phi''(t)\ls 2n\cdot 2^n$ for all $r/2\ls t\ls r$,
 it is easy to check that  there exists two positive number $c_1,C_1$ depending only on  $n$ and $r$ such that
 $$c_1\cdot |px|^2\ls h_\alpha(x)=\phi(|q_\alpha x|)\ls C_1\cdot |px|^2$$
  if $r-|px|\ls |q_\alpha x|\ls r-|px|/2$.

Therefore, we have (since for each $x\in B_p(\delta)\backslash\{p\}$, there is some $q_\alpha$ such that $\Delta h_\alpha(x)\gs n-C\delta$.)
$$\Delta h_0\gs n-N\cdot C\delta \quad {\rm on}\quad B_p(\delta) $$
and
 $$c_1\cdot |px|^2\ls h_\alpha(x)=\phi(|q_\alpha x|)\ls N\cdot C_1\cdot |px|^2.$$
 By $N\ls c(n)$ for some constant $c(n)$ depending only on the dimension $n$, if $\delta< (C\cdot c_n)^{-1},$
  the function $h_0$ satisfies all of  conditions in the lemma.
\end{proof}
\begin{rem}\label{rem4.12}
If $p$ is a regular point, then the constant $r_2$ given in Lemma \ref{lem3.4} can be chosen uniformly in a neighborhood of $p$.
 Indeed, in this case, there exists a neighborhood of $p$ which is bi-Lipschitz homeomorphic to an open domain of $\mathbb R$
  with an bi-Lipschitz constant close to 1. The constant $r$ and $\delta$ in the above proof can be chosen to
   have a lower bound depending only on the bi-Lipschitz constant.
\end{rem}
\be{prop}{\label{prop4.13}{\rm(Petrunin \cite{petu96})} \indent Given any $p\in \Omega $ and $\lambda\gs0$,
 there exists constants $R_p$ and $c(p,\lambda)$ such that, for any
 $u\in W^{1,2}(\Omega)\cap C(\Omega)$ satisfing $\mathscr L_u\ls \lambda\cdot\rv$  on $\Omega$, we have
 \begin{equation}
\label{eq4.21}\int_{B_p(R)}ud\nu_{p,R}\ls u(p)+c(p,\lambda)\cdot R^2
\end{equation}
for any  ball $B_p(R)\Subset\Omega$ with $0<R<R_p$, where the  constant $c(p,\lambda)=0$ if $\lambda=0.$}
\begin{proof}This proposition was given by Petrunin  in \cite{petu96} (Page. 5). For completeness, we give a detailed proof as follows.

\emph{Case 1:} $\lambda=0$.

For each $r\in(0,R)$, let $u_r$ be the harmonic function on $B_p(r)$ with boundary value $u_r=u$ on $\partial B_p(r)$.
 Then $\mathscr L_{u-u_r}\ls0$ and  $(u-u_r)|_{\partial B_p(r)}=0.$ By applying maximum principle, we know that $u-u_r\gs0$ on $B_p(r)$.
That is,  by the definition of  $\mu(r)$, $ \mu(r)\ls u(p).$
Therefore, by the definition of  $\nu_{p,R}$, we have
$$\int_{B_p(R)}ud\nu_{p,R}\ls u(p).$$

\emph{Case 2:} $\lambda>0$.

 Let $h_0$ be the function given in Lemma \ref{lem3.4}, we have $\mathscr L_{u-\lambda h_0}\ls 0$ on $B_p(r_2)$,
  where $r_2$ is the constant given in Lemma \ref{lem3.4}.
Hence, we can use the case above for function $u-\lambda h_0.$ This gives us, by Lemma \ref{lem3.4},
$$u(p)=u(p)-\lambda h_0(p)\gs\int_{B_p(R)}(u-\lambda h_0)d\nu_{p,R}\gs \int_{B_p(R)}ud\nu_{p,R}-C\cdot\lambda\cdot R^2
$$
for all $0<R<r_2$, where $C$ is the constant given in Lemma \ref{lem3.4}.
 \end{proof}
\begin{rem}\label{rem4.14} If $p$ is regular, according to Remark \ref{rem4.9} and Remark \ref{rem4.12},
the constant $R_p$ can be chosen uniformly in a neighborhood of $p$.
\end{rem}

The following lemma is similar as one appeared in \cite{petu96} (Page. 10).
\begin{lem}\label{lem4.15}{\rm (Petrunin \cite{petu96})} \indent Let $h$ be the Perelman concave function given
 in Lemma \ref{lem3.3} on a neighborhood  $U\subset M$. Assume that $f$ is a semi-concave function defined on $U $.
  And suppose that   $u\in W^{1,2}(U )\cap C(U )$ satisfies $\mathscr L_u\ls \lambda\cdot\rv$  on $U$ for some constant $\lambda\in\R$.

We assume that point $x^*\in U$ is a minimal point of function $u+f+h$, then $x^*$ has to be regular.
 Moreover, $f$ is differentiable at $x^*$ $($in sense of Taylor expansion \eqref{eq2.16}$)$.
\end{lem}
\begin{proof}Without loss of generality, we may assume that $\lambda\gs0$. In the proof, we denote $B_{x^*}(R)\ (\subset U)$
 by $B_R$. From the minimum property of $x^*$, we have \begin{equation}\label{eq4.22}
\int_{B_R}(u+f+h)d\nu_{p,R}\gs u(x^*)+f(x^*)+h(x^*).\end{equation}
By Proposition \ref{prop4.13}, we get \begin{equation}\label{eq4.23}
\int_{B_R}ud\nu_{p,R}\ls u(x^*)+cR^2
\end{equation}for some constant $c=c(p,\lambda)$ and for all sufficiently small $R$.

On the other hand,   setting $\bar h=f+h$, we have
\begin{align}\label{eq4.24}
\int_{B_R}\bar hd\nu_{p,R}&= \bar h(x^*)+\int_{B_R}(\bar h-\bar h(x^*))d\Big(\nu_{p,R}-\frac{\rv}{H^n(B^k_o(R))}\Big)\\
\nonumber &\quad+\frac{1}{{H^n(B^k_o(R))}}\int_{B_R}(\bar h-\bar h(x^*))d{\rv}\\ \nonumber &:=\bar h(x^*)+J_1+J_2,\end{align}where $B^k_o(R)$ is the ball in $T^k_p$.

Because $\bar h=f+h$ is a Lipschitz function and $\frac{\rv(B_R)}{H^n(B^k_o(R))}=1+o(1),$ we have
\begin{equation}\label{eq4.25}|J_1|\ls O(R)\cdot \int_{B_R}1d\Big(\nu_{p,R}-\frac{\rv}{H^n(B^k_o(R))}\Big)=O(R)\cdot\Big(1-\frac{\rv(B_R)}{\rv(B^k_o(R))}\Big)=o(R).
\end{equation}
Since $\bar h=f+h$ is semi-concave, according to equation (\ref{eq2.7}), we have
\begin{equation}\label{eq4.26}
\begin{split}J_2&=\frac{\rv(B_p(R))}{{H^n(B^k_o(R))}}\fint_{B_R}(\bar h-\bar h(x^*))d{\rv}\\&
=\frac{\rv(B_R)}{H^n(B^k_o(R))}\cdot\Big(\frac{nR}{n+1}\fint_{\Sigma_{x^*}} d_{x^*}\bar h(\xi)d\xi+o(R)\Big)\\
&=\frac{nR}{n+1}\fint_{\Sigma_{x^*}} d_{x^*}\bar h(\xi)d\xi+o(R)\end{split}
\end{equation}

By combining (\ref{eq4.22})--(\ref{eq4.26}), we have
$$\frac{nR}{n+1}\fint_{\Sigma_{x^*}} d_{x^*}\bar h(\xi)d\xi+o(R)+cR^2\gs0.$$
By combining with  Proposition \ref{prop3.1},
$$\int_{\Sigma_{x^*}}d_{x^*}\bar h(\xi)d\xi=\int_{\Sigma_{x^*}}d_{x^*}f(\xi)d\xi+\int_{\Sigma_{x^*}}d_{x^*}  h(\xi)d\xi \ls0,$$
 we have
$$\int_{\Sigma_{x^*}}d_{x^*}f(\xi)d\xi=\int_{\Sigma_{x^*}}d_{x^*}h(\xi)d\xi=0.$$
Then by using  Lemma \ref{lem3.3} (iii), we conclude that $x^*$ is regular.

Next we want to show that $f$ is differentiable at $x^*$.

Since $x^*$ is regular, we have
$$\int_{\Sigma_{x^*}}\ip{\nabla_{x^*} f}{\xi}d\xi=\int_{\mathbb S^{n-1}}\ip{\nabla_{x^*} f}{\xi}d\xi=0.$$
Hence
$$\int_{\Sigma_{x^*}}\big(d_{x^*}f(\xi)-\ip{\nabla_{x^*} f}{\xi}\big)d\xi=\int_{\Sigma_{x^*}} d_{x^*}f(\xi)d\xi =0.$$
On the other hand, by the definition of $\nabla_{x^*}f$ (see Section 1.3 of \cite{petu07}), we have
$$d_{x^*}f(\xi)\ls \ip{\nabla_{x^*} f}{\xi}\qquad \forall\ \xi\in \Sigma_{x^*}.$$
The combination of above two equation, we have
$$d_{x^*}f(\xi)= \ip{\nabla_{x^*} f}{\xi}\qquad \forall\ \xi\in \Sigma_{x^*}.$$
Combining with the fact $x^*$ is regular, we get that $f$ is differentiable at $x^{*}$.
\end{proof}

 We now follow Petrunin in \cite{petu96} to introduce a perturbation argument. Let $u\in W^{1,2}(D)\cap C(\overline{D})$
  satisfy $\mathscr L_u\ls \lambda\cdot\rv$  on a bounded domain $D$. Suppose that $x_0$ is the unique minimum point of $u$ on $D$ and $u(x_0)<\min_{x\in \partial D}u.$ Suppose also that $x_0$ is regular and $g=(g_1,\ g_2,\ \cdots\ g_n): D\to \R^n$
   is a coordinate system around $x_0$ such that $g
$ satisfies the following:\\
 \indent (i) $g$ is an almost isometry from $D$ to $g(D)\subset\R^n$ (see \cite{bgp92}).
  Namely, there exists a sufficiently small number $\delta_0>0$ such that
   $$\Big|\frac{\|g(x)-g(y)\|}{|xy|}-1\Big|\le \delta_0,\qquad {\rm for\ all}\quad x,y \in D, \ x\not=y;$$
  \indent (ii) all of the coordinate functions $g_j,\ 1\ls j\ls n,$ are concave (\cite{p94}).\\
  Then there exists $\epsilon_0>0$ such that, for each vector $V=(v^1,v^2,\cdots,v^n)\in \R^{n}$ with  $|v^j|\ls\epsilon_0$
   for all $1\ls j\ls n$, the function $$G(V,x):=u(x)+V\cdot g(x)$$ has a minimum point in the interior of $D$, where $\cdot$
    is the Euclidean inner product of $\R^n$ and    $V\cdot g(x)=\sum_{j=1}^nv^{j}g_j(x)$.

  Let  $$\mathscr U=\{V\in \R^{n}:\ |v^j|<\epsilon_0,\ 1\ls j\ls n\}
 \subset\R^n.$$
  We define $ \rho: \mathscr U\to D$ by setting $\rho(V)$ to be one of minimum  point of $G(V,x)$.
 Note that the map $\rho$ might be not uniquely defined.

The following  was given by Petrunin  in \cite{petu96} (Page 8). For the completeness, a detailed proof is given here.
\begin{lem}\label{lem4.16}{\rm (Petrunin \cite{petu96})}\indent
Let $u,\ x_0,$ $\{g_j\}_{j=1}^n$ and $\rho$ be as above. There exists some $\epsilon\in(0,\epsilon_0)$ and $\delta>0$ such that   \begin{equation}\label{eq4.27}
|\rho(V)\rho(W)|\gs \delta \cdot\|V-W\|\qquad \forall\ V, W\in\mathscr U_\epsilon^+.
\end{equation}
 where
 $$ \mathscr U_\epsilon^+:=\{ V=(v_1,v_2,\cdots,v_n)\in \R^n: \  0<v^j<\epsilon \ \ {\rm for\ all}\ \ 1\ls j\ls n\}.$$
In particular, for arbitrary  $\epsilon'\in(0,\epsilon)$, the image  $\rho (\mathscr U^+_{\epsilon'})$ has nonzero Hausdorff measure.
\end{lem}
\begin{proof}Without loss of generality, we can assume that $\lambda\gs0$.

Since  $x_0$ is a regular point, according to Remark \ref{rem4.14}, the mean value inequality in Proposition \ref{prop4.13}
 holds uniformly on some neighborhood of $x_0$. Namely, there exists  neighborhood $U_{x_0}\ni x_0$ and two constants $R_0,c_0$
  such that for any $w\in W^{1,2}(D)\cap C(D)$ satisfying $\mathscr L_w\ls \lambda\cdot\rv$, we have
\begin{equation}\label{eq4+1}
\int_{B_q(R)}wd\nu_{q,R}\ls w(q)+c_0\cdot R^2
\end{equation}
for all $q\in U_{x_0} $ and all $R\in(0,R_0)$.

Noting that $G(V,x)=u(x)+V\cdot g$ converges to $u$ as $V\to0$, and that $x_0$ is the \emph{uniquely} minimal value point of
 $u(x)$, we can conclude that $\rho(V)$ converges to $x_0$ as $V\to0.$ Hence, there exists a positive number $\epsilon>0$
  such that $\rho(V)\in U_{x_0}$ provided $V=(v^1,\cdots, v^n) $ satisfies $|v^j|\ls \epsilon $ for all $1\ls j\ls n.$
From now on, we fix such $\epsilon$ and let
$$\mathscr U_\epsilon^+:=\{ V=(v_1,v_2,\cdots,v_n)\in \R^n: \  0<v^j<\epsilon \ \ {\rm for\ all}\ \ 1\ls j\ls n\}.$$

Let $V,W\in\mathscr U_\epsilon^+.$ Denote by $\rho:=\rho(V)$ and $\widehat{\rho}:=\rho(W)$. That means
$$G(V,\rho)\ls G(V,x)
\quad{\rm and}\quad G(W,\widehat{\rho})\ls G(W,x)$$
for any $x\in D$.
Hence, we have
\begin{align}\label{eq4.28}
(W-V)\cdot g(\widehat{\rho})-(W-V)\cdot g(x)&=G(W,\widehat{\rho})-G(V,\widehat{\rho})-G(W,x)+G(V,x)\\
 \nonumber&\ls G(V,x)-G(V,\widehat{\rho})\\
\nonumber&\ls G(V,x)-G(V,\rho).
\end{align}
Notice that $v^j>0$ and $g_j$ are concave for $1\ls j\ls n$. We know that $G(V,x)=u(x)+V\cdot g(x)$ also satisfies
$\mathscr L_{G(V,x)}\ls \lambda\cdot\rv $. By the mean value inequality \eqref{eq4+1}, we have
\begin{equation}\label{eq4.29}
\int_{B_\rho(R)}\big(G(V,x)-G(V,\rho)\big)d\nu_{\rho,R}\ls  c_0\cdot R^2
\end{equation}
for any $0<R<R_0.$
We denote $\phi_+:=\max\{\phi,0\}$ for a function $\phi$. It is clear that $(\phi+a)_+\ls \phi_++|a|$ for any $a\in \R.$
 By combining this and the assumption that $g$ is an almost isometry, we have
\begin{equation}\begin{split}\label{eq4.30}
\int_{B_\rho(R)}&\big((W-V)\cdot g(\rho)-(W-V)\cdot g(x)\big)_+d\nu_{\rho,R}\\
&  \ls
\int_{B_\rho(R)}\big((W-V)\cdot g(\widehat{\rho})-(W-V)\cdot g(x)\big)_+d\nu_{\rho,R}\\
&\qquad+|(W-V)\cdot g(\rho)-(W-V)\cdot g(\widehat{\rho})|\\
 &  \ls
\int_{B_\rho(R)}\big((W-V)\cdot g(\widehat{\rho})-(W-V)\cdot g(x)\big)_+d\nu_{\rho,R}\\&\qquad +  \|g(\rho)-g(\widehat{\rho})\|\cdot\|W-V\|\\
 &  \ls
\int_{B_\rho(R)}\big((W-V)\cdot g(\widehat{\rho})-(W-V)\cdot g(x)\big)_+d\nu_{\rho,R}+c_1\cdot|\rho\widehat{\rho}|\cdot\|W-V\|,
\end{split}\end{equation}
where constant $c_1$ depends only on $\delta_0$.

Consider the set $$K:=\big\{X\in\R^n\ \big|\ \frac{R}{4}\ls \|X-g(\rho)\|\ls \frac{R}{2},\ \ (X-g(\rho))\cdot(W-V)\ls-\frac{1}{2}\|X-g(\rho)\|\cdot\|V-W\|\big\}.$$
In fact, $K$ is a trunked cone in $\R^n$ with vertex $g(\rho)$, central direction $-W+V+g(\rho)$,
 cone angle $\frac{\pi}{3}$ and radius from $\frac{R}{4}$ to $\frac{R}{2}$.

Since $K\subset B_{g(\rho)}(R/2) $ and $g$ is an almost isometry with $\delta$ sufficiently small,
 it is obvious that $g^{-1}(K)\subset B_\rho(R)$. Hence, we have
\begin{equation}\label{eq4.31}\begin{split}
\int_{B_\rho(R)}\big(&(W-V)\cdot g(\rho)-(W-V)\cdot g(x)\big)_+d\nu_{\rho,R}\\
&\gs\int_{g^{-1}(K)}\big((W-V)\cdot (g(\rho) - g(x)\big)_+d\nu_{\rho,R}\\&\gs   \frac{1}{2}\|W-V\| \cdot \int_{g^{-1}(K)}  \|g(\rho)-g(x)\|d\nu_{\rho,R}\\
& \gs  \frac{R}{8}\|W-V\|\cdot \nu_{\rho,R}\big(g^{-1}(K)\big).
\end{split}
\end{equation}
By the estimate \eqref{eq4.18} and that $g$ is $\delta_0$-almost isometry, we have
\begin{equation}\label{eq4.32}\nu_{\rho,R}\big(g^{-1}(K)\big)\gs \frac{\rv(g^{-1}(K))}{H^n(B_o(R)\subset T^k_\rho)}\gs c_2
\end{equation}
for some constant $c_2$ depending only on $\delta_0$ and the dimension $n$, the lower bound $k$ of curvature.

By combining (\ref{eq4.28})--(\ref{eq4.32}), we obtain
$$ \frac{c_2 \cdot R}{8}\cdot \|W-V\|\ls  c_1|\rho\widehat{\rho}|\cdot\|W-V\|+c_0R^2$$ for any $0<R<R_0.$
We set \begin{equation}\label{eq4.33}N =\frac{n\epsilon\cdot c_2}{c_0R_0}+1.\end{equation}
Since $\|V-W\|\ls n\epsilon$, we get
$$ R':=\frac{c_2\cdot\|V-W\|}{10c_0\cdot N}\ls R_0/10.$$
 Then we have
 $$c_1\cdot|\rho\widehat{\rho}|\cdot\|V-W\|\gs  \frac{c_2R'}{8}\cdot\|V-W\|-c_0R'^2=\frac{c^2_2\cdot \|V-W\|^2}{10c_0N}\Big(\frac{1}{8}-\frac{1}{10N}\Big).$$
Now the desired estimate (\ref{eq4.27}) follows from the choice of
\begin{equation}\label{4.34}
\delta :=\frac{c^2_2}{400c_0\cdot c_1\cdot N}.
\end{equation}
Therefore, the proof of this lemma is completed.
\end{proof}

\section{Hamilton--Jacobi semigroup and   Bochner  type formula}
\subsection{Hamilton--Jacobi semigroup}
Let $M$ be an $n$-dimensional Alexandrov space  and $\Omega$ be a bounded domain of $M$.
Given a continuous  and bounded function $u$ on $\Omega$, the \emph{Hamilton--Jacobi} semigroup is defined by
$$Q_tu(x)=\inf_{y\in \Omega}\Big\{u(y)+ \frac{|xy|^2}{2t}\Big\},\qquad t>0$$and $Q_0u(x)
:=u(x).$ Clearly, $Q_tu$ is semi-concave for any $t>0$, since $u(y)+|\cdot y|^2/(2t)$ is semi-concave, for each $y\in\Omega$.
 In particular, $Q_tu$ is locally Lipschitz  for any $t>0.$

If
$|xy|>\sqrt{4t\|u\|_{L^\infty}},$ then
$$u(y)+\frac{|xy|^2}{2t}>u(y)+2\|u\|_{L^\infty}\gs \|u\|_{L^\infty}.$$
On the other hand, $Q_tu(x)\ls u(x)\ls \|u\|_{L^\infty}$. We conclude that
$$Q_tu(x)=\inf_{y\in \overline{B_x(C)}\cap\Omega}\Big\{u(y)+ \frac{|xy|^2}{2t}\Big\},$$
where $C=\sqrt{4t\|u\|_{L^\infty}}.$  Therefore, for any $\Omega'\Subset\Omega$, there exists $\bar t=\bar t(\Omega', \|u\|_{L^\infty})$
 such that
\begin{equation}
\label{eq5.1}Q_tu(x)=\min_{y\in \Omega}\Big\{u(y)+ \frac{|xy|^2}{2t}\Big\}
\end{equation}
for all $x\in \Omega'$ and $0<t<\bar t.$

For convenience, we always set $u_{t}:=Q_tu$ in this section.

The following was shown in  \cite{lv07-hj} in framework of length spaces.
\begin{lem}\label{lem5.1}{\rm (Lott--Villani \cite{lv07-hj})} \indent
$(i)$ For each $x\in \Omega'$, we have $\inf u\ls u_t(x)\ls u(x)$;\\
$(ii)$ $\lim_{t\to0^+}u_t=u$ in $C(\Omega')$;\\
$(iii)$ For any $t,s>0$ and   any $x\in \Omega'$, we have
\begin{equation}\label{eq5.2}
 0\ls u_{t}(x)-u_{t+s}(x) \ls   \frac{s}{2}\cdot{\bf Lip}^2u_t,
 \end{equation}
where ${\bf Lip}u_t$ is the Lipschitz constant of $u_t$ on $\Omega'$ $($see \cite{c99} for this notation.$)$;\\
$(iv)$ For any $t>0$ and almost  all $x\in \Omega'$, we have
\begin{equation}
\label{eq5.3}\lim_{s\to0^+}\frac{u_{t+s}(x)-u_t(x)}{t}=-\frac{|\nabla  u_t(x)|^2}{2}.
\end{equation}
\end{lem}
The following lemma is similar to Lemma 3.5 in \cite{b08}.
\begin{lem}
 \label{lem5.2} Let $t>0$. Assume $u_t$ is differentiable at $x\in \Omega'$.  Then there exists a unique point  $y\in \Omega$
   such that \begin{equation}\label{eq5.4}
u_t(x)=u(y)+\frac{|xy|^2}{2t}.
\end{equation}
 Furthermore, the direction $\uparrow_x^y$ is determined uniquely and
\begin{equation}\label{eq5.5}
|xy|\cdot\uparrow^y_x= -t\cdot \nabla u_t(x).
\end{equation}
\end{lem}
\begin{proof}

Now fix a regular point $x$. We choose arbitrarily  $y$  such that (\ref{eq5.4}) holds.
 Taking any geodesic $\gamma(s): [0,\epsilon)\to M$ with $\gamma(0)=x$, by the definition of $u_t$ and (\ref{eq5.4}), we have \begin{equation}\label{eq5.6}
 u_t(\gamma(s))-u_t(x)\ls \frac{|y\gamma(s)|^2}{2t}-\frac{|xy|^2}{2t}.
 \end{equation}
If $x=y$, we have $\nabla u_t(x)=0.$ Hence equation \eqref{eq5.5} holds.

If $x\not=y$, by using the differentiability  of $u_t$ at $x$ and the first variant formula, we have
\be{equation}
{\label{eq5.7}u_t(\gamma(s))=u_t(x)+d_xu_t(\gamma'(0))\cdot s+o(s)}
and
\begin{equation}
 \label{eq5.8}
  \frac{|y\gamma(s)|^2}{2t}- \frac{|xy|^2}{2t}\ls -\frac{|xy|}{t}\cdot \ip{\uparrow^y_x}{\gamma'(0)}\cdot s+o(s)
  \end{equation}
for any direction $\uparrow^y_x$ from $x$ to $y$.
By combining (\ref{eq5.6})--(\ref{eq5.8}), we have
 \begin{equation*}d_xu_t(\gamma'(0))\ls -\frac{|xy|}{t}\cdot \ip{\uparrow^y_x}{\gamma'(0)}\end{equation*}
 for all geodesic $\gamma$ with $\gamma(0)=x.$
For each $\xi\in \Sigma_x$, we take a sequence geodesics $\gamma(t)$ starting from $x$ such that $\gamma'(0)$ converges to $\xi$.
 Then we have
 \begin{equation}\label{eq5.9}
 d_xu_t(\xi)\ls -\frac{|xy|}{t}\cdot \ip{\uparrow^y_x}{\xi}
 \end{equation}
for all $\xi\in\Sigma_x$.

Since $u_t$ is differentiable at $x$, we know that the direction $-\xi$ exists and $d_xu(-\xi)=-d_xu(\xi)$. By replacing
  $\xi$ by $-\xi$ in the above inequality, we obtain
$$\nabla u_t(x)= -\frac{|xy|}{t}\cdot \uparrow^y_x.$$
The left-hand side does not depend on the choices of point $y$ and direction of $\uparrow_x^y$. This gives the desired assertion.
\end{proof}

For each $t>0$, we  define a  map $F_t:\ \Omega'\to \Omega$ by $F_t(x)$ to be one of point  such that
\begin{equation}\label{eq5.10}
u_t(x)=u\big(F_t(x)\big)+ \frac{|xF_t(x)|^2}{2t}.
\end{equation}
According to  the Lemma \ref{lem5.2} and Rademacher theorem (\cite{c99,b08}), we have, for almost all  $x\in\Omega'$,
\begin{equation}\label{eq5.11}
 |xF_t(x)|= t\cdot|\nabla u_t(x)|.
\end{equation}
 By Lemma \ref{lem5.2} again, $F_t$ is continuous at $x$, where $u_t$ is differentiable  (since the point $y$
  satisfying \eqref{eq5.4} is unique). Then $F_t$ is measurable.

In \cite{petu96}, Petrunin sketched a proof of his key Lemma, which states that, on an Alexandrov space with nonnegative curvature,
 $u_t$ is superharmonic on $\Omega'$ for each $t>0$ provided $u$ is supersolution of $\mathscr L_u=0$ on $\Omega$.
 The following proposition is an extension.

\be{prop}{\label{prop5.3}
  Let $M$ be an $n$-dimensional Alexandrov space with $Ric\gs-K$ and $\Omega$ be a bounded domain of $M$.
Assume that $u\in W^{1,2}(\Omega)\cap C(\Omega)$, $f\in L^\infty(\Omega)$ is upper semi-continuous for almost all $x\in\Omega$ and
$$\mathscr L_u\ls f\cdot\rv$$
in the sense of measure.
Then, for any $\Omega'\Subset\Omega$,
these exists some $t_0>0$ such that for all $0<t<t_0$, we have
\begin{equation}\label{eq5.12}
a^2\cdot \mathscr L_{ u_t}\ls \Big[f\circ F_t+\frac{n(a-1)^2}{t}+\frac{Kt}{3}(a^2+a+1)|\nabla u_t|^2\Big]\cdot\rv
\end{equation} on $\Omega'$ for all $a>0$.
}
\begin{proof}
 We divide the proof into the following four steps.

\noindent\textit{Step 1. Setting up a contradiction argument.}\indent

 Since, for almost all $x\in \Omega$, $f$ is upper semi-continuous  and $|xF_t(x)| =t|\nabla u_t(x)|,$
   it is sufficient to prove that
 there exists some $t_0>0$ such that for all $0<t<t_0$, we have
\begin{equation}
\label{eq5.13}a^2\cdot \mathscr L_{u_t}\ls \Big[\sup_{z\in B_{F_t(x)}(\theta)}f(z) +\frac{n(a-1)^2}{t}+\frac{K}{3t}\big(a^2+a+1\big)\cdot|xF_t(x)|^2+\theta\Big]\cdot\rv
\end{equation} on $\Omega'$ for all $a>0$ and all $\theta>0$.

 For each $t>0$,  $a>0$ and $\theta>0$, we set
 \begin{equation}
 \label{eq5.14}a^2\cdot  w_{t,a,\theta}(x)= \sup_{z\in B_{F_t(x)}(\theta)}f(z)+\frac{n(a-1)^2}{t}+\frac{K}{3t}\big(a^2+a+1\big)\cdot|xF_t(x)|^2+\theta.
 \end{equation}

  For each $t>0$, $a>0$ and $\theta>0$,   since $u_t$ is semi-concave,
 $|\nabla u_t|\in L^\infty(\Omega')$ and hence, we have $w_{t,a,\theta}\in L^\infty(\Omega')$. Noting that  $u_t$ is semi-concave again,
  it is sufficient to prove that $u_t$ satisfies the corresponding comparison property in Corollary \ref{cor4.6}
   for all sufficiently small $t>0$.

Let us argue by contradiction. Suppose that there exists a sequences of $t_j\to0^+$ as $j\to\infty$, a sequence $a_j>0$
 and a sequence $\theta_j>0$ satisfying the following: for each $t_j$ ,$a_j$ and $\theta_j$, we can find $p_j$ and $R_j>0$
  with $a_jR_j+R_j\to0^+$ and  $B_{p_j}(R_j)\Subset\Omega'$, such that the corresponding comparison property
   in Corollary \ref{cor4.6} is false. That is, if the function $v_j$ is the solution of equation
$$\mathscr L_{v_j}=-w_{t_j,a_j,\theta_j}\cdot\rv$$
 in $B_{p_j}(R_j)$ with boundary value $v_j=-u_{t_j}$ on $\partial B_{p_j}(R_j)$,
then the function $u_{t_j}+v_j$ has a minimum point in the interior of $B_{p_j}(R_j)$ and
$$\min_{x\in B_{p_j}(R_j)}(u_{t_j}+v_j)< \min_{x\in \partial B_{p_j}(R_j)}(u_{t_j}+v_j).$$
We call this case that $u_{t_j}+v_j$ has a strict minimum in the interior of $B_{p_j}(R_j)$.

 Since $\Omega'$ is bounded, we can assume that some subsequence of $\{p_j\}_{j=1}^\infty$ converges to a limit point $p_\infty$.
  Denote the subsequence by $\{p_j\}_{j=1}^\infty$ again.  So we can choose a convex neighborhood $U\Subset\Omega$ of $p_\infty$
  and a Perelman concave function $h$ on $U$ given in Lemma \ref{lem3.3}. Since $u$ is bounded, by
 $|xF_t(x)|^2\ls 4t\|u\|_{L^\infty(\Omega)}$, we have $|xF_{t_j}(x)|\to0$ as $j\to\infty$ uniformly on $\Omega'$. Now we fix some
  $j^*$ so large  that
 $$B_{p_{j^*}}\big(a_{j^*}R_{j^*}+R_{j^*}\big)\cup B_{F_{t_{j^*}}(p_{j^*})}\big(a_{j^*}R_{j^*}+R_{j^*}\big)\subset U$$
 and
$F_{t_{j^*}}(x)\in U$ for all $x\in B_{p_{j^*}}\big(a_{j^*}R_{j^*}+R_{j^*}\big).$\\

\noindent\textit{ Step 2. Perturbing the functions to achieve the minimums at smooth points}.

  From now on, we omit the index $j^*$ to simplify the notations.

 Let $x_1$ be a minimum of $u_t+v$ in the interior of $B_p(R)$. Because  $h$ is $2-$Lipschitz on $U$, for any sufficiently small
   positive number $\epsilon_0$, the function
 $$u_t+v+\epsilon_0h$$
  also achieves a strict minimum at some point $\bar x$ in the interior of $B_p(R).$ Noting that   $u_t$ is semi-concave and
    $w_{t,a,\theta}$ is bounded and
  $\mathscr L_{v }\ls  -w_{t,a,\theta}\cdot\rv$,
  according to Lemma \ref{lem4.15}, we know $\bar x$ is regular and that $u_t$ is differentiable at $\bar x$.
  Now we fix such a sufficiently small $\epsilon_0$.

  On the other hand, according to the condition $Ric\gs-K$ and Laplacian comparison (see \cite{zz10} or \cite{ks07}),
   we have $\mathscr L_{|x\bar x|^2}\ls c(n,K,{\rm diam}\Omega).$ Thus, by the fact $h$ is $(-1)$-concave,
    we can choose some sufficiently small positive number $\epsilon_0'$ such that
   $$\mathscr L_{\epsilon_0h+\epsilon_0'|x\bar x|^2}\ls0.$$
 Setting $v_0=v+\epsilon_0h+\epsilon_0'|x\bar x|^2$, we have that   the function
 $$u_t+v_0=u_t +v+\epsilon_0h+\epsilon_0'|x\bar x|^2$$
  achieves a \emph{unique} minimum at $\bar x$ and
  $$\mathscr L_{v_0}=\mathscr L_v+ \mathscr L_{\epsilon_0h+\epsilon_0'|x\bar x|^2}\ls \mathscr L_v=-w_{t,a,\theta}\cdot\rv.$$

 Consider function
$$H(x,y)=v_0(x)+u(y)+\frac{|xy|^2}{2t},\qquad (x,y)\in \Omega\times\Omega.$$
Then it achieves a \emph{unique} strict  minimum at $(\bar x,F_t(\bar x))\in B_p(R)\times U.$
Indeed,
 $$H(x,y)\gs u_t(x)+v_0(x)\gs u_t(\bar x)+v_0(\bar x)
 = u(F_t(\bar x))+\frac{|\bar x F_t(\bar x)|}{2t}+v_0(\bar x)=H(\bar x,F_t(\bar x)).$$
Since $\bar x$ is a regular point and $u_t$ is differentiable at $\bar x$, by Lemma \ref{lem5.2},
 the point pair $(\bar x,F_t(\bar x))$
 is the unique  minimum of $H$ in $ B_p(R)\times U.$\\

 Applying the fact that $h$ is $2$-Lipschitz on $U$, we know that, for any sufficiently small
  positive number $\epsilon_1$,
 $$H_1(x,y):=v_1(x)+u_1(y)+\frac{|xy|^2}{2t}$$
  also achieves its a strict minimum in the  interior of $B_p(R)\times U$, where
  $$v_1(x)=v_0(x)+\epsilon_1h(x)\quad{\rm and }\quad u_1(y)=u(y)+\epsilon_1h(y).$$
   Let  $(x^*,y^*)$ denote one of minimal point of $H_1$.

 By the condition $Ric\gs-K$ and Laplacian comparison (see \cite{zz10} or \cite{ks07}), we have
 $$\mathscr L_{|xx^*|^2}\ls c(n,K,{\rm diam}\Omega)\quad {\rm and }\quad \mathscr L_{|yy^*|^2}\ls c(n,K,{\rm diam}\Omega).$$
 Since
 $$H_1(x,y^*)=v_0(x)+u_1(y^*)+\frac{|xy^*|^2}{2t}+\epsilon_1h(x)$$
 is continuous and $w_{t,a,\theta}$ is bounded, we know that
 $$\mathscr L_{v_0+u_1(y^*)+\frac{|xy^*|^2}{2t}}\ls (-w_{t,a,\theta}+\frac{c(n,K,{\rm diam}\Omega)}{2t})\cdot\rv\ls \lambda \cdot\rv$$
  on $B_p(R)$ for some constant $\lambda\in \R$ and $H_1(x,y^*)$  has a minimum at $x^*$.
   By Lemma \ref{lem4.15}, we know that  $x^*$ is regular. The point $y^*$ is also regular,
    by the boundness of $f$ and the same argument.

Let $v_2(x)=v_1(x)+\epsilon_2|xx^*|^2$ and $u_2(y)=u_1(y)+\epsilon_2|yy^*|^2$ with  any positive number $\epsilon_2.$ Then $$H_2(x,y):=v_2(x)+u_2(y)+\frac{|xy|^2}{2t}$$
achieves a unique minimum point $(x^*,y^*)$.

Since $(x^*,y^*)$ is regular in $M\times M$, now we choose one almost orthogonal coordinate system near $x^*$
 by concave functions $g_1, g_2,\cdots, g_n$ and another almost orthogonal coordinate system near $y^*$ by concave functions $g_{n+1}$, $g_{n+2}$,$\cdots, g_{2n}$.
  Using Lemma \ref{lem4.16}, there exist arbitrarily small positive numbers $b_1,\ b_2,\cdots, b_{2n}$ such that
$$H_2(x,y)+\sum_{i=1}^n b_ig_i(x)+\sum_{i=n+1}^{2n} b_ig_i(y)$$
achieves a minimal point $(x^o,y^o)$ near point $(x^*,y^*)$, where $(x^o,y^o)$ satisfies the following properties:\\
\indent(1)\indent $x^o\not=y^o$;\\
\indent(2)\indent $x^o$ is a $dist_{y^o}$-regular point and $y^o$ is a $dist_{x^o}$-regular point (hence, they are smooth);\\
\indent(3)\indent  geodesic $x^oy^o$ can be extended beyond $x^o$ and $y^o$;\\
\indent(4)\indent  $y^o$ is a Lebesgue  point of  $f$;\\
\indent(5)\indent $x^o$ is a Lebesgue point of  $w_{t,a,\theta}$;\\
\indent(6)\indent $x^o$ is a Lebesgue point of  $\Delta (|xy^o|^2)$ and $y^o$ is a Lebesgue point of  $\Delta (|x^oy|^2)$,\\
where $\Delta (|xy^o|^2)$ (or $ \Delta (|x^oy|^2)$) is density of absolutely continuous part of $\mathscr L_{|xy^o|^2}$ (or
$ \mathscr L_{|x^oy|^2}$, resp.).

 Indeed, let $\mathcal A$ be the set of points satisfying all of conditions (1)--(6) above. It is easy to check that
 $H^{2n}\big((B_p(R)\times U)\backslash \mathcal A\big)=0.$ By applying Lemma  \ref{lem4.16}, we can find desired $(x^o,y^o)$.

Set $$v_3(x)=v_2(x)+\sum_{i=1}^n b_ig_i(x)\qquad {\rm and }\qquad u_3(y)=u_2(y)+\sum_{i=n+1}^{2n} b_ig_i(y).$$Then
$$H_3(x,y):=v_3(x)+u_3(y)+\frac{|xy|^2}{2t}$$
 has a minimal value at $(x^o,y^o).$\\

\noindent\textit{Step 3. Ricci curvature and second variation of arc-length.}

  Let $\gamma:[0,\bar s]\to U$ be a geodesic with $x^o,y^o\in \gamma\backslash\{\gamma(0),\gamma(\bar s)\}.$
 Put $x^o=\gamma(t_x)$ and $y^o=\gamma(t_y)$ with $0< t_x<t_y<\bar s.$ Assume that some neighborhood of $\gamma$ has curvature $\gs k_0$,
  for some $k_0\in\R.$
For each $t\in(0,\bar s)$, the tangent cone $T_{\gamma(t)}$ can be split isometrically into $T_{\gamma(t)}=\R\times L_{\gamma(t)}$.
 We denote
$$\Lambda_{\gamma(t)}=\Sigma_{\gamma}\cap L_{\gamma(t)}=\{\xi\in \Sigma_{\gamma(t)}\ \big | \ \ip{\xi}{\gamma'}=0.\}.$$

Fix an arbitrarily small positive number $\epsilon_3$. According the definition of $M$ having Ricci curvature $\gs-K$ along geodesic
 $\gamma$ (see Definition \ref{defn2.5}), for each $t_0\in [t_x,t_y],$ there exists an open neighborhood $I_{t_0}\ni t_0$ and
 a family functions $\{g_{\gamma(t)}\}_{t\in I_{t_0}}$ such that $\{g_{\gamma(t)}\}_{t\in I_{t_0}}$ satisfies $Condition\ (RC)$ and
\begin{equation}\label{eq5.15}
(n-1)\cdot\fint_{\Lambda_{\gamma(t)}}g_{\gamma(t)}(\xi)d\xi\gs- K-\epsilon_3, \qquad \forall t\in I_{t_0}.
\end{equation}
It is shown in \cite{zz10} that
\begin{equation}\label{eq5.16} |g_{\gamma(t)}|\ls C,\qquad \forall t\in I_{t_0}
\end{equation}
for some constant $C$ depends only on the distance $|x^o\gamma(0)|, |y^o\gamma(\bar s)|$, $|I_{t_0}|$ and
 the lower bound  $k_0$ of curvature on some neighborhood of $\gamma$.  For completeness, we recall its proof as follows.
   Since the family $\{\underline{g}_{\gamma(t)}=k_0\}$ satisfies $Condition\ (RC)$  (see Remark \ref{cuvtoRic}), we can assume that $g_{\gamma(t)}\gs k_0$. Otherwise, we replace $g_{\gamma(t)}$ by $g_{\gamma(t)}\vee k_0.$
    On the other hand, for any $q_1,q_2\in \gamma|_{I_{t_0}}$ with $|q_1q_2|\gs|I_{t_0}|/2$,
     letting isometry $T: \Sigma_{q_1}\to\Sigma_{q_2}$ and sequence $\delta_j$ be in the definition
      of $Condition\ (RC)$ (see Definition \ref{defn2.4}), by applying equation \eqref{eq2.17} with $l_1=l_2=1$ and $\ip{\xi}{\gamma'}=0$,
       we have
$$|\exp_{q_1}(\delta_j\xi)\ \exp_{q_2}(\delta_jT\xi)|\ls|q_1q_2| -g_{q_1}(\xi)\cdot|q_1q_2|\cdot \delta_j^2/2+o(\delta_j^2).$$
By the concavity of distance functions $dist_{\gamma(0)}$ and $dist_{\gamma(\bar s)}$, we get
$$|\gamma(0)\ \exp_{q_1}(\delta_j\xi)|\ls |\gamma(0)\ q_1|+C_{k_0,|\gamma(0)x^o|}\cdot \delta_j^2$$  and $$|\gamma(\bar s)\ \exp_{q_2}(\delta_jT\xi)|\ls |\gamma(\bar s)\ q_2|+C_{k_0,|\gamma(\bar s)y^o|}\cdot \delta_j^2.$$
Combining with triangle inequality
$$|\exp_{q_1}(\delta_j\xi)\ \exp_{q_2}(\delta_jT\xi)|\gs |\gamma(0)\ \gamma(\bar s)|-|\gamma(0)\ \exp_{q_1}(\delta_j\xi)|-|\gamma(\bar s)\ \exp_{q_2}(\delta_jT\xi)|,$$
we can obtain
$$g_{q_1}(\xi)\ls \frac{2}{|q_1q_2|}\cdot(C_{k_0,|\gamma(0)x^o|}+C_{k_0,|\gamma(\bar s)y^o|})\ls \frac{4}{|I_{t_0}|}\cdot(C_{k_0,|\gamma(0)x^o|}+C_{k_0,|\gamma(\bar s)y^o|}).$$

All of such neighborhood $I_{t_0}$ forms an open covering of $[t_x,t_y]$. Then there exists a sub-covering $I_1,I_2,\cdots, I_S.$
 Now we divide $[t_x,t_y]$ into $N$-equal part by
$$x_0=x^o, x_1,\cdots\, x_m,\cdots, x_N=x_N.$$
We can assume that any pair of adjacent $x_m, x_{m+1}$ lying into some same $I_{\alpha}, \ \alpha\in\{1,2,\cdots, S\}$.

By $Condition\ (RC)$, we can find a  sequence  $\{\delta_j\}$ and an
isometry $T_0: \Sigma_{x_0}\to \Sigma_{x_1}$ such  that equation \eqref{eq2.17} holds.  Next, we can find a further
subsequence $\{\delta_{1,j}\}\subset\{\delta_j\} $ and an isometry an
isometry $T_1: \Sigma_{x_1}\to \Sigma_{x_2}$ such  that equation \eqref{eq2.17} holds.
  After a finite steps of these procedures, we get a subsequence
$  \{\delta_{N-1,j}\}\subset\cdots\subset \{\delta_{1,j}\}\subset\{\delta_j\} $
  and a family isometries $T_m: \Sigma_{x_m}\to \Sigma_{x_{m+1}}$ such that, for each $m = 0, 1, . . . , N-1,$
\begin{equation*} \begin{split}
 |\exp_{x_m}&(\delta_{N-1,j} l_{1,m}\xi_m),\ \exp_{x_{m+1}}(\delta_{N-1,j} l_{2,m}T_m\xi_m)|\\ \ls  &|x_mx_{m+1}|+(l_{2,m}-l_{1,m})\ip{\xi_m}{\gamma'}\cdot \delta_{N-1,j}\\&+\Big(\frac{(l_{1,m}-l_{2,m})^2}{2|x_mx_{m+1}|}
-\frac{g_{x_m}(\xi_m^\bot)\cdot|x_mx_{m+1}|}{6}\cdot(l_{1,m}^2+l_{1,m} \cdot l_{2,m}+l^2_{2,m})\Big)\\ &\quad\cdot\Big(1-\ip{\xi_m}{\gamma'}^2\Big)\cdot\delta^2_{N-1,j}\\&+o(\delta_{N-1,j}^2)
\end{split}\end{equation*}
 for any $l_{1,m},l_{2,m}\gs0$ and any $\xi_m\in \Sigma_{x_m}$.

Denote the isometry $T: \Sigma_{x^o}\to \Sigma_{y^o}$ by
$$T=T_{N-1}\circ\cdots\circ T_1\circ T_0.$$
It is can be extend naturally to an isometry $T: T_{x^o}\to T_{y^o}$.\\

 We fix $a\gs0$ and
 $$a_m=\frac{m}{N}\cdot(1-a)+a,\qquad m=0,1,\cdots, N-1.$$
We have $a_m\gs0$, and $a_0=a, a_N=1$.

To simplify notations, we put $\{\delta_j\}=\{\delta_{N-1,j}\}$  and denote
$$\mathscr W=\{v\in T_{x_0}\ \big |\ av\in \mathscr W_{x^o}\quad{\rm and }\quad T v\in \mathscr W_{y^o}\}.$$

\noindent{\bf Claim 1:} We have
\begin{equation}\label{eq5.17}\begin{split}
&\int_{B_o(\delta_j)\cap \mathscr W}\Big(|\exp_{x^o}(a\eta)\ \exp_{y^o}(T\eta) |^2-|x^oy^o|^2\Big)dH^n(\eta)\\
&\qquad\ls\frac{\omega_{n-1}}{(n+2)}\cdot \delta^{2+n}_j\cdot \Big((1-a)^2
+\frac{(K+\epsilon_3)\cdot|x^oy^o|^2}{3n}\cdot \big(a^2+a +1\big) \Big)\\&\qquad +o(\delta_{j}^{n+2}).\end{split}
\end{equation}

By applying $Condition\ (RC)$, we have
\begin{equation*} \begin{split}
 |\exp_{x_m}&(\delta_{j} a_{m}\cdot b\xi_m),\ \exp_{x_{m+1}}(\delta_{j} a_{m+1}\cdot b\xi_{m+1})|\\
  \ls  & \frac{\ell}{N}+(a_{m+1}-a_{ m})\cdot b\ip{\xi}{\gamma'}\cdot \delta_{j}\\&+b^2\cdot\Big(\frac{N\cdot(a_{ m}-a_{m+1})^2}{2\ell}
-\frac{g_{x_m}(\xi_m^\bot)\cdot\ell}{6N}\cdot(a_{m}^2+a_{m} \cdot a_{m+1}+a^2_{m+1})\Big)\\&\quad\cdot\Big(1-\ip{\xi}{\gamma'}^2\Big)\cdot\delta^2_{j}\\&+o(\delta_{j}^2)\end{split}
\end{equation*}
 for any  $b\in[0,1]$ and  any $\xi \in \Sigma_{x_0}$, where $\ell=|x_0x_N|=|x^oy^o|$ and
 $$\xi_m:=T_m\circ T_{m-1}\circ\cdots\circ T_0\xi.$$
Hence, by combining the triangle inequality, we have
\begin{equation*} \begin{split} |\exp_{x_0}&(\delta_{j} a_{0}\cdot b\xi),\ \exp_{x_{N}}(\delta_{j} a_{N}\cdot b\xi_{N})|\\\ls& \sum_{m=0}^{N-1}
|\exp_{x_m}(\delta_{j} a_{m}\cdot b\xi_m),\ \exp_{x_{m+1}}(\delta_{j} a_{m+1}\cdot b\xi_{m+1})|\\
\ls&\ell +(a_{N}-a_{0})\ip{\xi}{\gamma'}b\cdot \delta_{j}\\&+b^2\cdot\sum_{m=0}^{N-1}\Big(\frac{N\cdot(a_{m}-a_{m+1})^2}{2\ell }
-\frac{g_{x_m}(\xi_m^\bot)\cdot\ell}{6N}\cdot(a_{m}^2+a_{m} \cdot a_{m+1}+a^2_{m+1})\Big)\\ &\qquad\qquad\cdot\Big(1-\ip{\xi}{\gamma'}^2\Big)\cdot\delta^2_{j}\\&+o(\delta_{j}^2)
\end{split}\end{equation*}
 for any $b\in[0,1]$. This is, by setting $v=b\xi$,
\begin{equation} \label{eq5.18}
\begin{split} |\exp_{x^o}&(\delta_{j} a v),\ \exp_{y^o}(\delta_{j}  Tv)|^2-|x^oy^o|^2
 \\ \ls  & 2\ell\cdot(1-a)\ip{v}{\gamma'}\cdot \delta_{j} +(1-a)^2\ip{v}{\gamma'}^2\cdot \delta^2_{j}\\&+\sum_{m=0}^{N-1}\Big(N\cdot(a_{m}-a_{m+1})^2
-\frac{g_{x_m}(\xi_m^\bot)\cdot\ell^2}{3N}\cdot(a_{m}^2+a_{m} \cdot a_{m+1}+a^2_{m+1})\Big)\\&\qquad\cdot\Big(|v|^2-\ip{v}{\gamma'}^2\Big)\cdot\delta^2_{j}\\&+o(\delta_{j}^2)
\end{split}\end{equation}
 for  any vector $v\in B_o(1)\subset T_{x_0}$.

Let $\mathscr F_j(v)$ be the function defined on $B_o(1)\subset T_{x_0}$ by
\begin{equation*}
\begin{split}\mathscr F_j(v):=\ &|\exp_{x^o}(\delta_{j} a v),\ \exp_{y^o}(\delta_{j}  Tv)|^2-|x^oy^o|^2
 \\ &- 2\ell\cdot(1-a)\ip{v}{\gamma'}\cdot \delta_{j} -(1-a)^2\ip{v}{\gamma'}^2\cdot \delta^2_{j}\\
  &-\sum_{m=0}^{N-1}\Big(N\cdot(a_{m}-a_{m+1})^2
-\frac{g_{x_m}(\xi_m^\bot)\cdot\ell^2}{3N}\cdot(a_{m}^2+a_{m} \cdot a_{m+1}
+a^2_{m+1})\Big)\\
&\qquad\cdot\Big(|v|^2-\ip{v}{\gamma'}^2\Big)\cdot\delta^2_{j}.
\end{split}\end{equation*}

 For any $v\in B_o(1)$, we rewrite \eqref{eq5.18} as
$$\limsup_{j\to\infty}\mathscr F_j(v)/\delta_j^2\ls0.$$

Next, we will prove that $\mathscr F_j(v)/\delta_j^2$ has a uniformly upper bound on $B_o(1).$
 Take the midpoint $z$ of $x^o$ and $y^o$. By the semi-concavity of distance function $dist_z$, we have
$$|z\ \exp_{x^o}(\delta_j\cdot av)| \ls |zx^o|-a\ip{v}{\gamma'}\delta_j+C_{k_0,|x^oy^o|}\cdot \delta_j^2$$
and
$$|z\ \exp_{y^o}(\delta_j\cdot Tv)| \ls |zy^o|+ \ip{Tv}{\gamma'}\delta_j+C_{k_0,|x^oy^o|}\cdot \delta_j^2.$$
By applying triangle inequality, we get
$$|\exp_{x^o}(\delta_j\cdot av)\ \exp_{y^o}(\delta_j\cdot Tv)|\ls
  |x^oy^o| +(1-a)\ip{v}{\gamma'}\delta_j+2C_{k_0,|x^oy^o|}\cdot \delta_j^2.$$
Hence
$$|\exp_{x^o}(\delta_j\cdot av)\ \exp_{y^o}(\delta_j\cdot Tv)|^2-|x^oy^o|^2\ls
  2\ell\cdot(1-a)\ip{v}{\gamma'}\delta_j+(4C^2+(1-a)^2) \cdot \delta_j^2.$$
By combining with the boundness of $g_{x_m}$ (i.e., equation \eqref{eq5.16}), we conclude that $\mathscr F_j(v)/\delta^2_j\ls C$.

Now, by applying Fatou's Lemma, we have
$$\limsup_{j\to\infty}\int_{B_o(1)}\frac{\mathscr F_j(v)}{\delta^2_j}dH^n(v)\ls\int_{B_o(1)}\limsup_{j\to\infty}\frac{\mathscr F_j(v)}{\delta^2_j}dH^n(v)\ls 0.$$
That is,
\begin{equation} \label{eq5.19}
\begin{split} \int_{B_o(1)}&\Big(|\exp_{x^o}(\delta_{j} a v),\ \exp_{y^o}(\delta_{j}  Tv)|^2-|x^oy^o|^2\Big)dH^n(v)
 \\ \ls  & 2\ell\cdot(1-a)\int_{B_o(1)}\ip{v}{\gamma'}dH^n(v)\cdot \delta_{j} +(1-a)^2\int_{B_o(1)}\ip{v}{\gamma'}^2dH^n(v)\cdot \delta^2_{j}\\&+\sum_{m=0}^{N-1}\Big(N\cdot(a_{m}-a_{m+1})^2\cdot\int_{B_o(1)}\Big(|v|^2
 -\ip{v}{\gamma'}^2\Big)dH^n(v)\cdot \delta^2_j\\
& -\frac{\ell^2}{3N}\cdot\sum_{m=0}^{N-1}(a_{m}^2+a_{m} \cdot a_{m+1}+a^2_{m+1})\\&\qquad \qquad \cdot\int_{B_o(1)}g_{x_m}(\xi_m^\bot)\cdot\Big(|v|^2-\ip{v}{\gamma'}^2\Big)dH^n(v)\cdot\delta^2_{j}\\& +o(\delta_{j}^2).
\end{split}\end{equation}
Since $x^o$ is regular, we have
\begin{equation*}\begin{split}
 &\int_{B_o(1)}\ip{v}{\gamma'}dH^n(v) =0,\\ &\int_{B_o(1)}\ip{v}{\gamma'}^2dH^n(v) =\frac{1}{n}\int_{B_o(1)}|v|^2dH^n(v)=\frac{\omega_{n-1}}{n(n+2)} \end{split}\end{equation*}
  and
  $$
 \int_{B_o(1)}\Big(|v|^2-\ip{v}{\gamma'}^2\Big)dH^n(v) =\frac{n-1}{n}\int_{B_o(1)}|v|^2dH^n(v)=\frac{(n-1)\omega_{n-1}}{n(n+2)},$$
where $\omega_{n-1}={\rm Vol}(\mathbb S^{n-1}).$\\
By equation \eqref{eq5.15}, and denoting $\xi_m=(\xi_m^\bot,\theta)\subset \Sigma_{x_m}$, the spherical suspension over $\Lambda_{x_m}$,
we have
\begin{equation*}\begin{split}
\int_{\Sigma_{x_m}}g_{x_m}(\xi_m^\bot)&\cdot\Big(|\xi_m|^2-\ip{\xi_m}{\gamma'}^2\Big)dH^{n-1}(\xi_m)\\ &= \int_{\Sigma_{x_m}} (1-\cos^2\theta)g_{x_m}(\xi_m^\bot)dH^{n-1}(\xi) \\
& = \int_0^\pi\int_{\Lambda_{x_{m}}}
 \sin^2\theta g_{x_m}(\xi_m^\bot)\cdot \sin^{n-2}\theta dH^{n-2}(\xi_m^\bot)d\theta\\
&= \int_0^\pi\sin^n\theta d\theta\int_{\Lambda_{x_{m}}} g_{x_m}(\xi_m^\bot)dH^{n-2}(\xi_m^\bot)\\
&\gs  \int_0^\pi\sin^n\theta
 d\theta\cdot \frac{-K-\epsilon_3}{n-1}\omega_{n-2}=-\frac{K+\epsilon_3}{n}\omega_{n-1}.
\end{split}\end{equation*}

Hence, we have
\begin{equation*}\begin{split}
 \int_{B_o(1)}& g_{x_m}(\xi_m^\bot)\cdot\Big(|v|^2-\ip{v}{\gamma'}^2\Big)dH^n(v)\\
 &= \int^1_0r^2\int_{\Sigma_{x_m}} g_{x_m}(\xi_m^\bot)\cdot\Big(|\xi_m|^2-\ip{\xi_m}{\gamma'}^2\Big)\cdot r^{n-1}dH^{n-1}(\xi_m)dr\\
 &=\frac{1}{n+2}\int_{\Sigma_{x_m}} g_{x_m}(\xi_m^\bot)\cdot\Big(|\xi_m|^2-\ip{\xi_m}{\gamma'}^2\Big) dH^{n-1}(\xi_m)\\
  &\gs  -\frac{K+\epsilon_3}{n(n+2)}\omega_{n-1}.
\end{split}\end{equation*}
Putting these   into \eqref{eq5.19}, and combining with $a_{m+1}-a_m=\frac{1-a}{N}$, we have
\begin{equation*}
\begin{split} \int_{B_o(1)}&\Big(|\exp_{x^o}(\delta_{j} a v),\ \exp_{y^o}(\delta_{j}  Tv)|^2-|x^oy^o|^2\Big)dH^n(v)
 \\ \ls  & (1-a)^2\frac{\omega_{n-1}}{n(n+2)}\cdot \delta^2_{j}\\&+\frac{(n-1)\omega_{n-1}}{n(n+2)}\cdot \delta^2_j\\ &\quad\cdot \sum_{m=0}^{N-1}\Big(N\cdot(a_{m}-a_{m+1})^2+\frac{\ell^2(K+\epsilon_3)}{3N(n-1)}\cdot \big(a_{m}^2+a_{m} \cdot a_{m+1}+a^2_{m+1}\big)
  \Big)\\& +o(\delta_{j}^2)
 \\ = & (1-a)^2\frac{\omega_{n-1}}{n(n+2)}\cdot \delta^2_{j}\\ &+\frac{(n-1)\omega_{n-1}}{n(n+2)}\cdot \delta^2_j\cdot \sum_{m=0}^{N-1}\Big( \frac{(a-1)^2}{N}+\frac{\ell^2(K+\epsilon_3)}{3N(n-1)}\cdot  \frac{a_{m+1}^3-a^3_{m}}{a_{m+1}-a_m}  \Big)\\& +o(\delta_{j}^2)\\
  = &\frac{\omega_{n-1}}{(n+2)}\cdot \delta^2_j\cdot \Big((1-a)^2+\frac{\ell^2(K+\epsilon_3)}{3n}\cdot \big(a^2+a +1\big) \Big)
   +o(\delta_{j}^2).
\end{split}
\end{equation*}

By set $\eta=v\delta_j$, we have
\begin{equation}\label{eq5.20}
\begin{split}
 \int_{B_o(\delta_j)}\Big(&|\exp_{x^o}(a \eta),\ \exp_{y^o}( T\eta)|^2-|x^oy^o|^2\Big)dH^n(\eta)\\
\ls &\frac{\omega_{n-1}}{(n+2)}\cdot \delta^{2+n}_j\cdot \Big((1-a)^2+\frac{\ell^2(K+\epsilon_3)}{3n}\cdot \big(a^2+a +1\big) \Big) +o(\delta_{j}^{n+2}).
\end{split}
\end{equation}
Since $x^o$ and $y^o$ are smooth, by \eqref{eq2.4} in Lemma \ref{lem2.1}, we have
$$H^n\big(B_o(\delta_j)\backslash\mathscr W\big) =o(\delta^{n+1}_j).$$
On the other hand, by triangle inequality, we have
\begin{equation*}\begin{split}
\big| |\exp_{x^o}(a \eta),&\ \exp_{y^o}( T\eta)|^2-|x^oy^o|^2\big|\\ &\ls  (|\exp_{x^o}(a \eta),\ \exp_{y^o}( T\eta)| +|x^oy^o| )\cdot (a|\eta|+|T\eta|)\\
&\ls C\delta_j
\end{split}
\end{equation*}
for all $\eta\in B_o(\delta_j)$.

Now the desired estimate \eqref{eq5.17} in   \textbf{Claim 1} follows from above two inequalities and equation \eqref{eq5.20}.\\

\noindent {\it Step 4. Integral version of maximum principle.}

 Let us recall that in Step 2, the point pair $(x^o,y^o)$ is a minimum of $H_3(x,y) $ on $B_p(R)\times U$.
Then we have \begin{equation}\label{eq5.21}\begin{split}
0&\ls\int_{ B_o(r)\cap \mathscr W}\Big(H_3\big(\exp_{x^o}(a\eta),\exp_{y^o}(T\eta)\big)-H_3(x^o,y^o)\Big)dH^{n}(\eta)\\
&=\int_{ B_o(r)\cap \mathscr W}\Big(v_3\big(\exp_{x^o}(a\eta)\big)-v_3(x^o)\Big)dH^{n}(\eta)\\&\qquad+\int_{ B_o(r)\cap \mathscr W}\Big(u_3\big(\exp_{y^o}(T\eta)\big)-u_3(y^o)\Big)dH^{n}(\eta)\\&\qquad+\int_{  B_o(r)\cap \mathscr W}\frac{|\exp_{x^o}(a\eta)\exp_{y^o}(T\eta)|^2-|x^oy^o|^2}{2t}dH^{n}(\eta)\\& :=I_1(r)+I_2(r)+I_3(r),\end{split}
\end{equation}
where $\mathscr W=\{v\in T_{x_0}\   |\ av\in \mathscr W_{x^o}\quad{\rm and }\quad T v\in \mathscr W_{y^o}\}.$

By the condition $Ric\gs-K$ and Laplacian comparison (see \cite{zz10} or \cite{ks07}), we have
$$\mathscr L_{|xx^o|^2}\ls c(n,K,{\rm diam}\Omega)\quad {\rm and }\quad \mathscr L_{|yy^o|^2}\ls c(n,K,{\rm diam}\Omega).$$

\noindent {\bf Claim 2:}\indent We have \begin{equation}\label{eq5.22} I_1(r)\ls\frac{-\epsilon_1+c\cdot\epsilon_2-w_{t,a,\theta}(x^o)}{2n(n+2)}\cdot a^2\cdot\omega_{n-1} r^{n+2}+o(r^{n+2})
\end{equation} and
\begin{equation} \label{eq5.23}
 I_2(r)\ls\frac{-\epsilon_1+c\cdot\epsilon_2+ f(y^o)}{2n(n+2)}\cdot\omega_{n-1} r^{n+2}+o(r^{n+2})\qquad
 \end{equation} for  all small $r>0$, where $c=c(n,K,{\rm diam}\Omega)$.

Let
$$\alpha(x)=v_3(x)+\frac{|xy^o|^2}{2t}\quad{\rm and}  \quad \beta=\frac{|xy^o|^2}{2t}.$$
Since  $x^o$ is a smooth point, by Lemma \ref{lem2.1}, we have
\begin{equation*}\begin{split}
\int_{B_o(r)\cap \mathscr W_{x^o}}&\Big(\alpha\big(\exp_{x^o}(a\eta)\big)-\alpha(x^o)\Big)dH^n(\eta)
\\ &= a^{-n}\cdot\int_{B_{x^o}(ar)}\big(\alpha(x)-\alpha(x^o)\big)(1+o(r))d\rv(x).
\end{split}\end{equation*}

 Note that  $\alpha(x)-\alpha(x^o)\gs0$ and
\begin{equation*}\begin{split}
\mathscr L_{v_3}&\ls \mathscr L_{v_2}\ls \big(-\epsilon_1+c(n,K,{\rm diam}\Omega)\cdot\epsilon_2\big)\cdot\rv+\mathscr L_{v_0}\\ &\ls (-w_{t,a,\theta}-\epsilon_1+c \cdot\epsilon_2)\cdot\rv,
\end{split}\end{equation*}
 $$\mathscr L_{\alpha-\alpha(x^o)} =\mathscr L_{v_3}+\mathscr L_\beta\ls \big(-w_{t,a,\theta}-\epsilon_1+c \cdot\epsilon_2+\Delta \beta\big)\cdot\rv.     $$
  Since $x^o$ is a Lebesgue point of $-w_{t,a,\theta}+\Delta \beta$, by Corollary \ref{cor4.5},
we get
\begin{equation*}
\begin{split}
\int_{\partial B_{x^o}(s)} &\big(\alpha (x)-\alpha(x^o)\big)d\rv(x)\\&  \ls \frac{-w_{t,a,\theta}(x^o)-\epsilon_1+c \cdot\epsilon_2+\Delta\beta(x^o)}{2n}\cdot s^2\cdot H^{n-1}\big(\partial B^k_o(s)\big) +o(r^{n+1})
\end{split}\end{equation*}
for all $0<s<ar$.
By combining with the fact that $x^o$ is regular, we have
\begin{equation*}\begin{split}
\int_{B_{x^o}(ar)}& \big(\alpha (x)-\alpha(x^o)\big)d\rv(x)\\&  \ls \frac{-w_{t,a,\theta}(x^o)-\epsilon_1
+c \cdot\epsilon_2+\Delta\beta(x^o)}{2n(n+2)}\cdot \omega_{n-1}\cdot(ar)^{n+2} +o(r^{n+2}).
\end{split}\end{equation*}
Therefore, we obtain (since $\alpha(x)-\alpha(x^o)\gs0$,)
\begin{equation}\label{eq5.24}
\begin{split}&\int_{B_o(r)\cap \mathscr W} \Big(\alpha\big(\exp_{x^o}(a\eta)\big)-\alpha(x^o)\Big)dH^{n}(\eta)\\& \qquad \ls
\int_{B_o(r)\cap \mathscr W_{x^o}}\Big(\alpha\big(\exp_{x^o}(a\eta)\big)-\alpha(x^o)\Big)dH^{n}(\eta)\\
& \qquad\ls \frac{-w_{t,a,\theta}(x^o)-\epsilon_1+c \cdot\epsilon_2+\Delta\beta(x^o)}{2n(n+2)}\cdot a^2\cdot\omega_{n-1} r^{2+n}+o(r^{2+n}).
\end{split}
\end{equation}

On the other hand,  since $\beta$ is Lipschitz (since it is semi-concave) and equation \eqref{eq2.4}
$$H^n\big(B_o(r)\backslash \mathscr W\big)=o(r^{n+1}),$$
 we have
 \begin{equation*}\begin{split}
 \int_{ B_o(r)\cap \mathscr W}& \Big(\frac{|\exp_{x^o}(a\eta)y^o|^2}{2t}-\frac{|x^oy^o|^2}{2t}\Big)dH^{n}(\eta)
 \\ &=\int_{B_o(r)\cap \mathscr W_{x^o} }\big(\beta(\exp_{x^o}(a\eta))-\beta(x^o)\big)dH^{n}(\eta)+o(r^{n+2}).
\end{split} \end{equation*}
Since $x^o\in Reg_\beta$, by applying equation \eqref{eq2.3} in Lemma \ref{lem2.1}, the Lipschitz continuity of $\beta$
  and Lemma \ref{lem2.3}, we get
\begin{equation*}\begin{split}\int_{B_o(r)\cap \mathscr W_{x^o} }&\big(\beta(\exp_{x^o}(a\eta))-\beta(x^o)\big)dH^{n}(\eta)\\ &=a^{-n}\int_{B_{x^o}(ar)}\big(\beta(x)-\beta(x^o)\big)d\rv+o(r^{n+2})\\
& =\frac{\Delta \beta(x^o)}{2n(n+2)}\cdot a^2\cdot\omega_{n-1} r^{n+2}+o(r^{n+2}).
\end{split}\end{equation*}
By combining above two equalities, we have
\begin{align}\label{eq5.25}\int_{B_o(r)\cap \mathscr W}\Big(&\frac{|\exp_{x^o}(a\eta)y^o|^2}{2t}-\frac{|x^oy^o|^2}{2t}\Big)dH^{n}(\eta)\\
\nonumber&=\frac{\Delta \beta(x^o)}{2n(n+2)}\cdot a^2\cdot\omega_{n-1}  r^{n+2}+o(r^{n+2}).\end{align}
Therefore, the desired estimate  (\ref{eq5.22}) follows from  equations \eqref{eq5.24},  \eqref{eq5.25}  and   $v_3=\alpha- \beta$.

The estimate for (\ref{eq5.23}) is similar.
Let
$$\widetilde{\alpha}(y)=u_3(y)+\frac{|x^oy|^2}{2t}\quad {\rm and} \quad\widetilde{\beta}=\frac{|x^oy|^2}{2t}.$$
  By a similar argument to (\ref{eq5.24}) and (\ref{eq5.25}), we have, for  all small $r>0$,
 \begin{align*}
 \int_{B_o(r)\cap   \mathscr W }\Big(\widetilde{\alpha}\big(\exp_{y^o}(T\eta)\big)&-\widetilde{\alpha}(y^o)\Big)dH^{n}(\eta)\\ &\ls \frac{f(y^o)-\epsilon_1+c\cdot\epsilon_2+\Delta\widetilde{\beta}(y^o)}{2n(n+2)}\cdot\omega_{n-1} r^{2+n}+o(r^{2+n})
 \end{align*}
and
\begin{align*}\int_{  B_o(r)\cap \mathscr W}&\Big(\frac{|\exp_{y^o}(T\eta)x^o|^2}{2t}-\frac{|x^oy^o|^2}{2t}\Big)dH^{n }(\eta)
=\frac{\Delta \widetilde{\beta}(y^o)}{2n(n+2)}\cdot\omega_{n-1}  r^{n+2}+o(r^{n+2}).\end{align*}
Thus the combination of these two estimates and $u_3(y)=\widetilde{\alpha}- \widetilde{\beta}$ implies (\ref{eq5.23}).
 The proof of  \textbf{Claim 2} is finished.

By combining  \eqref{eq5.21} and \textbf{Claim 1}  \eqref{eq5.17}, \textbf{Claim 2} \eqref{eq5.22}-- \eqref{eq5.23},   we have
\begin{equation*}\begin{split}&\Big[\frac{-\epsilon_1+c\cdot\epsilon_2}{2n}(a^2+1)-\frac{a^2\cdot w_{t,a,\theta}(x^o)}{2n}+\frac{f(y^o)}{2n}+\frac{(a-1)^2}{2t}\Big]\cdot \delta^{n+2}_j\\ &+\Big[\frac{(K+\epsilon_3)|x^oy^o|^2}{6nt}(1+a+a^2)\Big]\cdot \delta^{n+2}_j+o(\delta^{n+2}_j)\gs0
\end{split}\end{equation*}
for all $j\in \mathbb N.$ Thus,
\begin{equation*}\begin{split}
\frac{-\epsilon_1+c\cdot\epsilon_2}{2n}(a^2+1)&-\frac{a^2\cdot w_{t,a,\theta}(x^o)}{2n}+\frac{f(y^o)}{2n}+\frac{(a-1)^2}{2t}\\&+\frac{(K+\epsilon_3)|x^oy^o|^2}{6nt}(1+a+a^2)\gs0.
\end{split}\end{equation*}
Combining with the definition of function $w_{t,a,\theta}$, \eqref{eq5.14}, we have
\begin{align}
\label{eq5.26}0\ls & (a^2+1)\frac{-\epsilon_1+c\epsilon_2}{2n}+\frac{(a^2+a+1)}{6nt}\big((K+\epsilon_3)|x^oy^o|^2
-K|x^oF(x^o)|^2\big)
\\ \nonumber& -\frac{1}{2n}\big(\sup_{z\in B_{F_t(x^o)}(\theta)} f(z)-f(y_0)\big)-\frac{\theta}{2n}.
\end{align}

 In Step 2, we have known that $(\bar x,F_t(\bar x))$ is the \emph{unique} minimum point of $H(x,y)$.
 Because $H_3(x,y)$ converges to $H(x,y)$ as $\epsilon_1, \epsilon_2$ and $b_j, \ 1\ls j\ls 2n,$ tend to $0^+$,
  we know that $(x^o,y^o)$ converges to $(\bar x,F_t(\bar x))$, as $\epsilon_1, \epsilon_2$ and $b_j, \ 1\ls j\ls 2n,$ tend to $0^+$.

  On the other hand, because  $\bar x$ is regular and  $x^o$ converges to $\bar x$ as $\epsilon_1, \epsilon_2$ and $b_j, \ 1\ls j\ls 2n,$ tend to $0^+$,  functions
  $$u(y)+\frac{|x^oy|^2}{2t}$$
  converges to function
   $$u(y)+\frac{|\bar x y|^2}{2t}$$
   as $\epsilon_1, \epsilon_2$ and $b_j, \ 1\ls j\ls 2n,$ tend to $0^+$. $F_t(x^o)$ is a minimum of  $u(y)+ |x^oy|^2/(2t)$.
    $u_t$ is differentiable at $\bar x$ (see Step 2). So $F_t(\bar x)$ is the \emph{unique} minimum point of $u(y)+ |\bar xy|^2/(2t)$.
     Therefore,  $F_t(x^o)$ converges to $F_t(\bar x)$ as $\epsilon_1, \epsilon_2$ and $b_j, \ 1\ls j\ls 2n,$ tend to $0^+$.

    Hence, when we choose $\epsilon_1, \epsilon_2$ and $b_j, \ 1\ls j\ls 2n$ sufficiently small, we have that $|y^oF_t(x^o)|\ll\theta$.
     This implies
 $$y^o\in B_{F_t(x^o)}(\theta)\quad{\rm and }\quad \big||x^oy^o|-|x^oF_t(x^o)|\big|\ll\theta.$$
Now we can choose  $\epsilon_1$, $\epsilon_2$ and $\epsilon_3$ so small that
$$(a^2+1)\frac{-\epsilon_1+c\epsilon_2}{2n}+\frac{(a^2+a+1)}{6nt}\big((K+\epsilon_3)|x^oy^o|^2-K|x^oF(x^o)|^2\big)\ls\frac{\theta}{4n}$$
and $y\in B_{F_t(x^o)}(\theta)$.
This contradicts to \eqref{eq5.26}. Therefore  we  have completed the proof of the proposition.
\end{proof}

\begin{lem}\label{lem5.4}
Let $\Omega$ be a bounded open domain in an $n$-dimensional Alexandrov space. Assume that a  $W^{1,2}(\Omega)$-function
 $u$ satisfies
$\mathscr L_u\gs f\cdot\rv$ for some $f\in L^{\infty}(\Omega)$. Then, for any $\Omega'\Subset\Omega$, we have
$$\sup_{x\in \Omega'}u  \ls C \|u\|_{L^1(\Omega)}+C\|f\|_{L^\infty(\Omega)},$$
where the constant $C$ depending on lower bounds of curvature, $\Omega$, and $\Omega'$.
\end{lem}
\begin{proof}
If $f=0$ and $u\gs0$, this lemma has been shown in Theorem 8.2 of \cite{bm06} for any metric measure space supporting a doubling property and
 a weak Poincar\'e inequality.  According to volume comparison and Theorem 7.2 of \cite{kms01}, it holds for Alexandrov spaces.

  On the other hand, according to  Lemma 6.4 of \cite{bm06} (see also Lemma 3.10 of \cite{km03}), we know that $u_+$ is also a subsolution of $\mathscr L_u=0$, that is $\mathscr L_{u_+}\gs0$.

 Therefore, if $f=0$, we have
 $$\sup_{x\in \Omega'}u  \ls\sup_{x\in \Omega'}u_+\ls C \|u_+\|_{L^1(\Omega)} \ls C \|u\|_{L^1(\Omega)}.$$

   In fact, the proof in \cite{bm06} works for general $f\in L^\infty(\Omega)$. In the following, we give a simple argument
   for the general case on Alexandrov spaces.

For each $p\in \Omega$, we choose a  Perelman concave function $h$ defined on some neighborhood $B_p(r_p)$, which is given
 in Lemma \ref{lem3.3}, such that $ -1\ls  h \ls 0.$
Then we have
$$\mathscr L_{u-\|f\|_{L^\infty(\Omega)}h}\gs (f+\|f\|_{L^\infty(\Omega)})\cdot\rv\gs 0\quad {\rm on}\ B_p(r_p).$$
 Applying   the above estimate (in case $f=0$), we have
 \begin{equation*}\begin{split}
 \sup_{B_p(r_p/2)} u &\ls\sup_{B_p(r_p/2)}(u-\|f\|_{L^\infty(\Omega)}  h)\ls  C \|u-\|f\|_{L^\infty(\Omega)}h\|_{L^1(B_p(r_p))}\\
 &\ls  C \|u\|_{L^1(B_p(r_p))}+C\|f\|_{L^\infty(\Omega)}\cdot \rv(B_p(r_p)).
 \end{split}\end{equation*}
   Since $\overline{\Omega'}$ is compact, there is finite such balls $B_{p_i}(r_i)$ such that above estimate hold
    on each $B_{p_i}(r_i)$ and that $\Omega'\subset \cup_iB_{p_i}(r_i/2)$. Therefore, we have
   $$\sup_{\Omega'} u \ls  C \|u\|_{L^1(\Omega)}+C\|f\|_{L^\infty(\Omega)}\cdot \rv(\Omega).$$
   The proof of the lemma is finished.
\end{proof}

In \cite{petu96,petu03}, by using his key Lemma, Petrunin proved that any harmonic function on an Alexandrov space
 with nonnegative curvature  is locally Lipschitz continuous. Very recently,  this Lipschitz continuity result on compact Alexandrov spaces was also
  obtained by Gigli--Kuwada--Ohta  in \cite{gko10} via probability method. We can now establish the locally Lipschitz continuity
   for solutions of general Poisson equations.
\begin{cor}\label{cor5.5}
  Let $M$ be an $n$-dimensional Alexandrov space  and $\Omega$ be a bounded domain of $M$.
Assume that $u$ satisfies $\mathscr L_u=f\cdot\rv$  on $\Omega$ and $f\in Lip(\Omega)$. Then $u$ is locally Lipschitz continuous.
\end{cor}
\begin{proof}
Since $\Omega$ is bounded, we may assume that  $M$ has Ricci curvature $\gs-K$ on $\Omega$ with some $K\gs0$.

By applying Lemma \ref{lem5.4} to both $\mathscr L_u=f\cdot\rv$ and $\mathscr L_{-u}=-f\cdot\rv$, we can conclude that $u\in L^\infty(\Omega')$ for any $\Omega'\Subset \Omega$. Without loss of generality, we may assume
 $$-1\ls u\ls0 $$
 on $\Omega'$.
  Otherwise, we replace $u$ by $(u-\sup_{\Omega'} u)/(\sup_{\Omega'} u-\inf_{\Omega'} u)$.

 Fix any open subset $\Omega_1\Subset\Omega'$ and let $(u_t)_{0\ls t\ls \bar t}$ be its Hamilton--Jacobi semigroup defined on $\Omega_1$.
  By Lemma \ref{lem5.1}, we know
$$-1\ls u_t\ls 0$$ on $\Omega_1$, for all $0\ls t\ls \bar t.$

By Proposition \ref{prop5.3}, there is $t_0>0$ such that \eqref{eq5.12} holds for all $t\in (0,t_0)$ and all $a>0$.
  By putting $a=1$ in \eqref{eq5.12}, we have
$$\mathscr L_{u_t}\ls (f\circ F_t+ Kt|\nabla u_t|^2)\cdot\rv,\qquad \forall\ 0<t< t_0$$
 on $\Omega_1$.

Set
 $$\bar K=K+1\qquad {\rm and}\qquad \Phi_t(x)=\frac{\exp(-\bar Kt u_t)-1}{t}$$
for all $0<t< t_0 (\ls 1)$. Then we have
 $$0\ls \Phi_t \ls \bar Ke^{\bar K},\quad\ 1\ls \exp(-\bar Ktu_t)\ls e^{\bar K}$$
  and, for each $t\in(0,t_0)$,
  \begin{equation}\label{eq5.27}\begin{split}
  \mathscr L_{\Phi_t}&=-\bar K\exp(-\bar Ktu_t)\cdot(\mathscr L_{u_t}-\bar Kt|\nabla u_t|^2)\cdot\rv\\
  &\gs  -\bar K\exp(-\bar Ktu_t)\cdot(f\circ F_t+Kt|\nabla u_t|^2-\bar Kt|\nabla u_t|^2)\cdot\rv\\
  &\gs  -\bar K\exp(-\bar Ktu_t)\cdot \|f\|_{L^\infty(\Omega)} \cdot\rv\\
  &\gs  -C\cdot\rv
  \end{split}\end{equation}
 in sense of measure on $\Omega_1$. Here and in the following, $C$ will denote various positive constants that do not depend on $t$ (while they might depend on $K$, $t_0$,  $\Omega,\Omega_1, \Omega_2,  \Omega_3$, $\|f\|_{L^\infty(\Omega)}$ and  the Lipshitz constant  of $f$, ${\bf Lip}f$, on $\Omega$).

 By applying Caccioppoli inequality (see Proposition 7.1 of \cite{bm06},)  (or by choosing test function $\varphi \Phi_t$
 for some suitable cut-off $\varphi$ on $\Omega_1$), we have
  \begin{equation*} \|\nabla \Phi_t\|_{L^2(\Omega_2)}\ls C \|\Phi_t\|_{L^2(\Omega_1)}\ls C\end{equation*}
 for any open subset $\Omega_2\Subset\Omega_1.$

Noting that $-\bar Ku_t\gs0$ and
$$|\nabla \Phi_t|=\bar K\exp(-\bar Ktu_t)|\nabla u_t|\gs \bar K|\nabla u_t|,$$
we have
 \begin{equation}\label{eq5.28}\|\nabla  u_t\|_{L^2(\Omega_2)}\ls C.\end{equation}

 By  using inequalities $\exp(-\bar Ktu_t)\ls e^{\bar K}$ and  $|1-e^\gamma+\gamma\cdot e^\gamma|\ls C\cdot\gamma^2/2$
  for any $0\ls\gamma\ls \bar K  t_0.$  we have, for each $t\in (0,t_0)$ and $x\in \Omega_1$,
  \begin{equation}\begin{split}\label{eq5.29} \big| \Phi_{t+s}(x)-\Phi_t(x)\big|
  & \ls \Big|\frac{\exp\big(-\bar K(t+s)u_{t+s}\big)-1}{t+s} -\frac{\exp\big(-\bar Ktu_{t+s}\big)-1}{t} \Big|\\
  &\quad+ \Big|\frac{\exp\big(-\bar Ktu_{t+s}\big)-1}{t} -\frac{\exp\big(-\bar Ktu_{t}\big)-1}{t} \Big|  \\
  &\ls s\cdot \max_{t\ls t'\ls t+s}\Big|\frac{\exp(-t'\bar Ku_{t+s})(-\bar Ku_{t+s})t'-\exp(-t'\bar Ku_{t+s})+1}{(t')^2}\Big|\\
  &\quad+ \bar K|u_{t+s}-u_t|\cdot\max_{u_{t+s}\ls a\ls u_t}\exp(-\bar Kta)\\
  &\ls Cs+C |u_{t+s}-u_t|\end{split} \end{equation}
  for all $0<s< t_0-t.$

By applying Dominated convergence theorem,  \eqref{eq5.28}, \eqref{eq5.29} and Lemma \ref{lem5.1}(iii--iv),
  we have
\begin{equation*}\begin{split} \frac{\partial^+}{\partial t}\|\Phi_t\|_{L^1(\Omega_2)}&=\limsup_{s\to0^+}\int_{\Omega_2}\frac{\Phi_{t+s}(x)-\Phi_t(x)}{s}d\rv\\
&\ls
C\rv(\Omega_2)+C\limsup_{s\to0^+}\int_{\Omega_2}\frac{|u_{t+s}-u_t|}{s}d\rv\\
&= C\rv(\Omega_2)+\frac{C}{2}\int_{\Omega_2} |\nabla  u_t|^2  d\rv\ls C.
\end{split}\end{equation*}
This implies that
\begin{equation}\label{eq5.30}\|\Phi_t\|_{L^1(\Omega_2)}\ls \|\Phi_{t'}\|_{L^1(\Omega_2)}+C(t-t')\end{equation}
 for any $0<t'<t<t_0.$
Since  $0\ls\Phi_{t'}\ls \bar Ke^{\bar K}$ and $\lim_{t'\to0^+}\Phi_{t'}(x)=-\bar Ku(x)$, we have
$$\lim_{t'\to0^+}\|\Phi_{t'}\|_{L^1(\Omega_2)}=\int_{\Omega_2}(-\bar Ku)d\rv.$$
 By combining with  \eqref{eq5.30}, we have $$\int_{\Omega_2}\frac{\Phi_t+\bar Ku}{t}d\rv=\frac{1}{t}(\|\Phi_t\|_{L^1(\Omega'_1)}-\lim_{t'\to0^+}\|\Phi_{t'}\|_{L^1(\Omega'_1)})\ls C.$$

On the other hand,  for each $t\in (0,t_0)$,
since $f$ is Lipschitz and
$$|xF_t(x)|=t|\nabla u_t(x)|,$$
 for almost all $x\in \Omega_1$,  we have
 \begin{equation*}\begin{split}\mathscr L_{\Phi_t+\bar Ku}&= -\bar K\exp(-\bar Ktu_t)\big(\mathscr L_{u_t}-\bar Kt|\nabla u_t|^2\big)\cdot\rv
 +\bar Kf \cdot\rv\\
& = -\bar K\exp(-\bar Ktu_t)\cdot\big(\mathscr L_{u_t}-\bar Kt|\nabla u_t|^2-f\big)\cdot\rv\\
& \quad -\bar K f\cdot\big(\exp(-\bar Ktu_t) -1\big)\cdot\rv\\
&\gs -\bar K\exp(-\bar Ktu_t)\cdot\big(f\circ F_t+Kt|\nabla u_t|^2-\bar Kt|\nabla u_t|^2-f\big)\cdot\rv\\
& \quad -\bar K f\cdot\big(\exp(-\bar Ktu_t) -1\big)\cdot\rv\\
& \gs-\bar K\exp(-\bar Ktu_t)\cdot( {\bf Lip}f \cdot|xF_t(x)| -t|\nabla u_t|^2)\cdot\rv-Ct\cdot\|f\|_{L^\infty(\Omega)}\cdot\rv\\
& = -\bar K\exp(-\bar Ktu_t)\cdot t\cdot({\bf Lip}f\cdot|\nabla u_t| - |\nabla u_t|^2)\cdot\rv-Ct\cdot\|f\|_{L^\infty(\Omega)}\cdot\rv\\
& \gs -Ct\cdot\Big(\frac{{\bf Lip}^2f}{4}+\|f\|_{L^\infty(\Omega)}\Big)\cdot\rv\\
&\gs -Ct\cdot\rv
 \end{split}\end{equation*}
in sense of measure on $\Omega_2$. Note that $\Phi_t+\bar Ku\gs -\bar Ku_t+  \bar Ku\gs0 $ (since Lemma \ref{lem5.1}(i)).
According to Lemma \ref{lem5.4}, we get
$$\max_{\Omega_3}\Big|\frac{\Phi_t+\bar Ku}{t}\Big|\ls C \|\frac{\Phi_t+\bar Ku}{t}\|_{L^1(\Omega_2)}  +C   =  C \int_{\Omega_2} \frac{\Phi_t+\bar Ku}{t}d\rv  +C \ls C$$
for any open subset $\Omega_3\Subset\Omega_2$. Hence, we have (since $\Phi_t\gs -\bar Ku_t$)
 $$\frac{-u_t+u}{t}\ls \bar K^{-1}\frac{\Phi_t+\bar Ku}{t}\ls C $$
 on $\Omega_3$, for each $t\in(0,t_0)$.

Therefore, by the definition of $u_t$, we obtain
$$u(x)\ls u_t(x)+Ct\ls  u(y)+\frac{|xy|^2}{2t}+ Ct$$
 for all $x,y\in \Omega_3$ and $t\in(0, t_0)$.
Now fix $x$ and $y$ in $\Omega_3$ such that  $|xy|<  t_0$. By choosing $t=|xy|$, we get $$u(x)-u(y)\ls C|xy|.$$
 Hence, by replacing $x$ and $y$, we have $$|u(x)-u(y)|\ls C|xy|, \quad {\rm for\ all}\ |xy|<  t_0.$$
This implies that $u$ is Lipschitz continuous on $\Omega_3$.

 By the arbitrariness of $\Omega_3\Subset\Omega_2\Subset\Omega_1\Subset \Omega'\Subset \Omega$, we get that $u$ is locally Lipschitz continuous on $\Omega$, and complete the proof.
\end{proof}

\subsection{Bochner's type formula}
 Bochner  formula is one of  important tools in differential geometry.  In this subsection, we will extend it
  to Alexandrov space with Ricci curvature  bounded below.

\begin{lem}\label{lem5.6}
Let $u\in Lip(\Omega)$  with Lipschitz constant $ \mathbf{Lip}u$, and let $u_t$ is its \emph{Hamilton--Jacobi}
 semigroup defined on $\Omega'\Subset\Omega$, for $0\ls t<\bar t$. Then we have the following  properties:\\
 \indent (i)\indent For any $t>0$, we have
 \begin{equation}\label{eq5.31}
 |\nabla^- u|(F_t(x))\ls |\nabla u_t(x)|\ls {\rm Lip}u(F_t(x))
 \end{equation}
for almost all $x\in \Omega'$, where $F_t$ is defined in  \eqref{eq5.10}.

 In particular, the Lipschitz constant of $u_t$, $\mathbf{Lip}u_t\ls  \mathbf{Lip}u$;\\
\indent (ii)\indent For   almost all $x\in \Omega'$, we have
 \begin{equation}\label{eq5.32}
 \lim_{t\to0^+}\frac{u_t(x)-u(x)}{t}=-\frac{1}{2}|\nabla u(x)|^2.\end{equation}
 Furthermore, for each sequence $t_j$ converging to $0^+$, we have
 $$\lim_{t_j\to0^+}\nabla u_{t_j}(x)=\nabla u(x)$$
 for almost all $x\in \Omega'$.
\end{lem}
\begin{proof}
(i) Lipschitz function $u_t$ is differentiable at almost all point $x\in \Omega'$. For such a point $x$, we firstly prove
 $|\nabla^- u|(F_t(x))\ls |\nabla u_t(x)|.$

Assume $|\nabla^- u|(F_t(x))>0$. (If not, we are done.) This implies $y:=F_t(x)\not=x$. Indeed, if $F_t(x)=x$, we have
$$u(x)\ls u(z)+\frac{|xz|^2}{2t}$$
for all $z\in \Omega'$. Hence  $\big(u(x)-u(z)\big)_+\ls|xz|^2/(2t)$. This concludes  $|\nabla^- u|(F_t(x))=0.$

Take a sequence of points $y_j$ converging to $y$ such that
$$\lim_{y_j\to y} \frac{u (y)-u (y_j)}{|yy_j|} =|\nabla^-u|(y).$$
Let $x_j$ be points in geodesic $xy$ such that $|xx_j|=|yy_j|$.
 By \begin{equation*}u_t(x_j)\ls u(y_j)+\frac{|x_jy_j|^2}{2t}\quad {\rm and}\quad u_t(x)=u(y)+\frac{|xy|^2}{2t},
 \end{equation*}
we have
\begin{equation}\label{eq5.33}u_t(x_j)-u_t(x)\ls u(y_j)-u(y)+\frac{1}{2t}(|x_jy_j|^2-|xy|^2).\end{equation}
Since $u_t$ is differentiable at $x$,
$$u_t(x_j)-u_t(x)=|xx_j|\cdot\ip{\nabla u_t(x)}{\uparrow^{x_j}_x}+o(|xx_j|).$$
Triangle inequality implies
$$|x_jy_j|\ls |x_jy|+|yy_j|=|x_jy|+|xx_j|=|xy|.$$
Therefore, by combining with \eqref{eq5.33}, we have
\begin{equation*}
\begin{split}u(y)-u(y_j)&\ls -|xx_j|\cdot\ip{\nabla u_t(x)}{\uparrow^{x_j}_x}+o(|xx_j|)\ls |xx_j|\cdot|\nabla u_t(x)|+o(|xx_j|)\\
&=|yy_j|\cdot|\nabla u_t(x)|+o(|xx_j|).
\end{split}\end{equation*}
Letting $y_j\to y$, this implies $|\nabla^-u|(y)\ls|\nabla u_t(x)|.$

Now let us prove $ |\nabla u_t(x)|\ls {\rm Lip}u(F_t(x))$ at a point $x$, where $u$ is differentiable.
Assume $|\nabla u_t(x)|>0$. (If not, we are done.) This implies $y:=F_t(x)\not=x$. Indeed, If $y=x$, we have
$$u_t(z)\ls u(x)+\frac{|xz|^2}{2t}= u_t(x)+\frac{|xz|^2}{2t},\quad \forall\ z\in\Omega'.$$
On the other hand, $u_t$ is differentiable at $x$,
$$u_t(z)=u_t(x)+\ip{\nabla u_t(x)}{\uparrow^z_x}\cdot|xz|+o(|xz|).$$
Hence, we obtain
$$\ip{\nabla u_t(x)}{\uparrow^z_x}\ls |xz|/(2t)+o(1)$$
for all $z$ near $x$. Hence $|\nabla u_t(x)|=0$.

Let the sequence $x_j\in \Omega'$ converge to $x$ and
\begin{equation}\label{eq5.34}\lim_{x_j\to x}\ip{\nabla u_t(x)}{\uparrow^{x_j}_x}= |\nabla u_t(x)|.\end{equation}
 Take $y_j$ be   points  in geodesic $xy$ with $|yy_j|=|xx_j|$.
By triangle inequality, we have
 $$|x_jy_j|\ls|xx_j|+|xy_j|=|yy_j|+|xy_j|= |xy|.$$
 Combining with
 $$u_t(x_j)\ls u(y_j)+\frac{|x_jy_j|^2}{2t}\quad{\rm and}\quad u_t(x)=u(y)+\frac{|xy|^2}{2t},$$
 we have
  \begin{equation}
  \label{eq5.35}u_t(x_j)-u_t(x)\ls u(y_j)-u(y)\ls |u(y_j)-u(y)|.
   \end{equation}
Since $u_t$ is differentiable at $x$,
$$u_t(x_j)-u_t(x)=\ip{\nabla u_t(x)}{\uparrow^{x_j}_x}\cdot |xx_j|+o(|xx_j|).$$
Hence, by combining with \eqref{eq5.34}, \eqref{eq5.35} and $|x_jx|=|y_jy|$, we get
$$|\nabla u_t(x)|\ls \limsup_{y_j\to y} \frac{|u(y_j)-u(y)|}{|yy_j|}\ls {\rm Lip}u(y).$$
The assertion (i) is proved.\\

\noindent(ii) The equation (\ref{eq5.32}) was proved by Lott--Villani in \cite{lv07-hj} (see also Theorem 30.30 in \cite{v09}).

Now let us prove the second assertion. The functions $u$ and $u_{t_j}$ are  Lipschitz on $\Omega'$.
 Then they are differentiable at almost all point $x\in \Omega'$. For such a point $x$,
according to (\ref{eq5.5}) in Lemma \ref{lem5.2}, we have, for each $t_j$,
$$u_{t_j}(x)= u (y_{t_j}) + \frac{|xy_{t_j}|^2}{2t_j} =u (y_{t_j}) +t_j\cdot\frac{|\nabla u_{t_j}(x)|^2}{2},$$
 where $y_{t_j}$ is the (unique) point such that \eqref{eq5.4} holds,  and
 $$u (y_{t_j})  =u(x)+|xy_{t_j}|\ip{\nabla u(x)}{\uparrow^{y_{t_j}}_x} +o(t_j)= u(x)-t_j\ip{\nabla u(x)}{\nabla u_{t_j}(x)} +o(t_j).$$
The combination of above two equation and (\ref{eq5.32}) implies that
$$\lim_{t_j\to0^+}\Big(-\ip{\nabla u(x)}{\nabla u_{t_j}(x)}+\frac{|\nabla u_{t_j}(x)|^2}{2}\Big)=-\frac{|\nabla u(x)|^2}{2}.$$
This is
$$\lim_{t_j\to0^+}\Big(|\nabla u(x)|^2-2\ip{\nabla u(x)}{\nabla u_{t_j}(x)}+ |\nabla u_{t_j}(x)|^2 \Big)=0,$$
which implies
$$\lim_{t_j\to0^+}\nabla u_{t_j}(x)=\nabla u(x).$$
 Now the proof of this lemma is completed.
\end{proof}

Now we have the following Bochner type formula.
\begin{thm}[Bochner type formula]\label{bochner}Let $M$ be an $n$-dimensional Alexandrov space with Ricci curvature bounded from below by $-K$
and $\Omega$ be a bounded domain in $M$.  Let  $f(x,s):\Omega\times[0,+\infty)\to \R$ be a Lipschitz function  and satisfy the following:\\
\indent $(a)$\indent  there exists a zero measure set $\mathcal N\subset \Omega$ such that for all $s\gs0$, the functions $f(\cdot,s)$ are differentiable at any $x\in \Omega\backslash \mathcal N;$\\
\indent $(b)$\indent the function $f(x,\cdot)$ is of class $C^1$ for  all $x\in \Omega$  and the function $\frac{\partial f}{\partial s}(x,s)$ is continuous, non-positive on $\Omega\times [0,+\infty)$.

Suppose that $u\in Lip(\Omega)$  and
  $$\mathscr L_u=f(x,|\nabla u|^2)\cdot\rv.$$
Then we have $|\nabla u|^2\in W^{1,2}_{loc}(\Omega) $ and
\begin{equation}\label{eq5.36}
\mathscr L_{|\nabla  u |^2} \gs  2\Big(\frac{f^2(x,|\nabla u|^2)}{n}+\ip{\nabla u}{\nabla  f(x,|\nabla u|^2)}-K|\nabla u|^2\Big)\cdot\rv
\end{equation}
in sense of measure on $\Omega$, provided  $|\nabla u|$ is lower semi-continuous at almost all $x\in \Omega$, namely, there exists a representative of $|\nabla u|$ which is lower semi-continuous at  almost all $x\in \Omega$.
\end{thm}
\begin{proof}
Recalling the  pointwise  Lipschitz constant ${\rm Lip}u $ of $u$ in Section 2.2, we defined a function
$$g(x):=\max\{{\rm Lip}^2u, |\nabla u(x)|^2\},\quad \forall\ x\in \Omega.$$
Noting that the fact  $ {\rm Lip}u =|\nabla u|$ for almost all $x\in \Omega$, we have $g=|\nabla u|^2$ for almost all $x\in \Omega$, and hence
  $$\mathscr L_u=f\big(x,g(x)\big)\cdot\rv $$
  in sense of measure on $\Omega.$

The function $g$ is lower semi-continuous at almost all $x\in \Omega$. Indeed, by the definition of $g$, we have $g(x)\gs |\nabla u(x)|^2$ at any $x\in\Omega$. On the other hand, $g(x)=|\nabla u(x)|^2$ at almost all $x\in\Omega$. Combining with the fact that $|\nabla u|$ is lower semi-continuous at  almost all $x\in \Omega$, we can get the desired lower semi-continuity of $g$ at almost all $x\in \Omega$.

The combination of the assumption $\frac{\partial f}{\partial s}\ls0$ and the lower semi-continuity of $g$ at  almost everywhere in $\Omega$  implies   that  $f=f\big(x,g(x)\big)$ is upper semi-continuous at almost all  $x\in\Omega$.

Fix any open subset $\Omega'\Subset\Omega$. Let $u_t$ be Hamilton--Jacobi semigroup of $u$, defined on $\Omega'$and
let $F_t$ be the map defined in (\ref{eq5.10}). By applying Proposition \ref{prop5.3}, there exists some $t_0>0$
 such that for each $t\in (0,t_0)$,   we have
$$a^2\cdot \mathscr L_{ u_t}\ls\Big[ f\circ F_t +\frac{n(a-1)^2}{t}+\frac{Kt}{3}(a^2+a+1)|\nabla u_t|^2\Big]\cdot\rv$$
for all $a>0$. Hence, the absolutely continuous part $\Delta u_t$ satisfies
$$a^2\cdot \Delta u_t(x)\ls f\circ F_t(x)+ \frac{n(a-1)^2}{t}+\frac{Kt}{3}(a^2+a+1)|\nabla u_t(x)|^2 $$
for all $a>0$ and  almost all $x\in\Omega'$.
By setting  $$D=-\frac{K}{3}|\nabla u_t(x)|^2$$ and $$A_1=-\Delta u_t(x)+\frac{n}{t}-tD,\quad A_2=-\frac{2n}{t}-tD,\quad A_3=f\circ F_t(x) +\frac{n}{t}-tD,$$ we can  rewrite this equation as $$\qquad A_1\cdot a^2+A_2\cdot a+A_3\gs0$$
for all $a>0$ and  almost all $x\in\Omega'$.

By taking $a=1$, we have
 \begin{equation}\label{eq5.37}
 \Delta u_t(x)\ls f\circ F_t(x)-3tD.
 \end{equation}
  Because  $u$ is in Lipschitz, by Lemma \ref{lem5.6}(i), we have
$$|D|=|K|\cdot|\nabla u_t|/3\ls |K|\cdot{\bf Lip}u/3,\qquad g\ls{\bf Lip}^2u,$$
and then $f=f\big(x,g(x)\big)$ is bounded.

 The combination of equation \eqref{eq5.37} and  the boundedness of $D, \ f$ implies that $A_1>0$ and $A_2<0$,
   when $t$ is sufficiently small. By choosing $a=-\frac{A_2}{2A_1}$, we obtain
 \begin{equation}\label{eq5.38}
 \big(\Delta u_t(x)-f\circ F_t(x)\big)\cdot\big(\frac{n}{t}-tD\big)\ls -\Delta u_t(x)\cdot f\circ F_t(x) -3nD+\frac{3}{4}t^2D^2.
 \end{equation}
Therefore, \\
(by writing $f=f(x,g(x)) $ and $f\circ F_t=f\circ F_t(x)=f(F_t(x),g\circ F_t(x))$,)
 \begin{equation*}\begin{split}
 &\frac{\Delta u_t(x)-f\big(x,g(x)\big)}{t}\ls \frac{(n-t^2D)\big(f\circ F_t -f\big)/t-f\cdot f\circ  F_t -3nD+3t^2D^2/4}{n-t^2D+tf\circ F_t}\\
 &\qquad=\frac{f\circ F_t -f}{t}-\frac{f^2+3nD}{\mathcal A}+\frac{f^2 - f^2\circ F_t }{\mathcal A}+\frac{3t^2D^2}{4\mathcal A}\\
 &\qquad=\frac{f\circ F_t-f\big(F_t(x),|\nabla u_t(x)|^2\big)}{t}+\frac{f\big(F_t(x),|\nabla u_t(x)|^2\big)-f}{t}-\frac{f^2+3nD}{\mathcal A}\\
 &\qquad\quad +\frac{f^2- f^2\big(F_t(x),|\nabla u_t(x)|^2\big)}{\mathcal A}+\frac{f^2\big(F_t(x),|\nabla u_t(x)|^2\big)-f^2\circ F_t }{\mathcal A}+\frac{3t^2D^2}{4\mathcal A}\\
 &\qquad=\frac{f\big(F_t(x),|\nabla u_t(x)|^2\big)-f}{t}+\frac{f^2 -f^2\big(F_t(x),|\nabla u_t(x)|^2\big)}{\mathcal A}-\frac{f^2+3nD}{\mathcal A}\\
 &\qquad\quad +\Big(f\circ F_t-f\big(F_t(x),|\nabla u_t(x)|^2\big)\Big)\cdot\Big(\frac{1}{t}-\frac{f\circ F_t+f\big(F_t(x),|\nabla u_t(x)|^2\big)}{\mathcal A} \Big)\\ &\qquad\quad+\frac{3t^2D^2}{4\mathcal A}\\
  \end{split}\end{equation*}
for almost all $x\in \Omega'$, where $$\mathcal A=n-t^2D+tf\circ F_t.$$

From Lemma \ref{lem5.6}(i) and the definition of function $g$, we have
$$g\circ F_t(x)\gs {\rm Lip}^2u(F_t(x))\gs |\nabla u_t(x)|^2, \quad a.e., \  x\in \Omega'.$$
Combining with   the assumption $\frac{\partial f}{\partial s}\ls0$, we have, for almast all  $x\in \Omega',$
$$f\circ F_t-f\big(F_t(x),|\nabla u_t(x)|^2\big)=f\big(F_t(x),g\circ F_t(x) \big)-f\big(F_t(x),|\nabla u_t(x)|^2\big)\ls0.$$
On the other hand, by the boundedness of  $D$ and $f$, we have
$$  \mathcal A=n-t^2D+tf\circ F_t\gs  \frac{n}{2}$$
 when $t$ is sufficiently small.
By combining with the boundedness of $f$, we have
$$\frac{1}{t}-\frac{f\circ F_t+f\big(F_t(x),|\nabla u_t(x)|^2\big)}{\mathcal A}\gs0$$
when $t$ is sufficiently small.

 When $t$ is sufficiently small, by using
 $ \mathcal A \gs  n/2$  and the boundedness of $D$ again, we have
 \begin{equation*}\begin{split}
 \frac{ \Delta u_t(x)-f\big(x,g(x)\big)}{t}&\ls \frac{f\big(F_t(x),|\nabla u_t(x)|^2\big)-f}{t}+\frac{f^2-f^2\big(F_t(x),|\nabla u_t(x)|^2\big)}{\mathcal A}\\
 &\quad -\frac{f^2+3nD}{\mathcal A}+C\cdot t.
  \end{split}\end{equation*}
 Here and in the following in this proof, $C$ will denote various positive constants
 that do not depend on $t$.

Note that $\mathscr L_{u_t}\ls \Delta u_t\cdot\rv$ and $\mathscr L_{u}=f\cdot\rv$. The above inequality implies that
\begin{equation*}\begin{split}
&\frac{1}{t}\mathscr L_{u_t-u}\\&\ls \Big[\frac{f\big(F_t(x),|\nabla u_t(x)|^2\big)-f}{t}+\frac{f^2-f^2\big(F_t(x),|\nabla u_t(x)|^2\big)}{\mathcal A} -\frac{f^2+3nD}{\mathcal A}+C\cdot t\Big]\cdot\rv
\end{split}\end{equation*}
in sense of measure on $\Omega'$.

Fix arbitrary  $0\ls\phi\in Lip_0(\Omega').$ We have
\begin{align}\label{eq5.39}\frac{1}{t}\mathscr L_{u_t-u}(\phi)&\ls \int_{\Omega'}\phi\cdot\Big(\frac{f\big(F_t(x),|\nabla u_t(x)|^2\big)-f}{t}\Big)d\rv\\
\nonumber&\qquad+\int_{\Omega'}\phi\cdot\frac{f^2-f^2\big(F_t(x),|\nabla u_t(x)|^2\big)}{\mathcal A}d\rv\\
\nonumber&\qquad-\int_{\Omega'}\phi\cdot\frac{f^2+3nD}{\mathcal A}d\rv+Ct\sup|\phi|\\
\nonumber &:=I_1(t)+  I_2(t)-  I_3(t)+Ct\sup|\phi|.\end{align}

We want to take limit in above inequality. So we have to estimate  the limits of  $I_1(t)$, $I_2(t)$ and $I_3(t)$, as $t\to0^+$.

Since for almost all $x\in \Omega'$,
  $$g={\rm Lip}u(x)=|\nabla u(x)|,$$
  we have
\begin{equation}\label{eq5.40}\begin{split}
I_1(t)&=\int_{\Omega'}\phi\frac{f\big( F_t(x), |\nabla u_t(x)|^2 \big)-f\big(x ,g(x)\big)}{t}d\rv\\
&= \int_{\Omega'}\phi\frac{f\big(F_t(x), |\nabla u_t(x)|^2 \big)-f(x ,|\nabla u(x)|^2)}{t}d\rv\\
&=\int_{\Omega'}\phi\frac{f\big(F_t(x),|\nabla u_t(x)|^2  \big)-f\big(F_t(x),|\nabla u (x)|^2\big)}{t}d\rv\\
&\qquad + \int_{\Omega'}\phi\frac{f\big(F_t(x),|\nabla u(x)|^2\big)-f\big(x,|\nabla u(x)|^2\big)}{t}d\rv\\
&:=J_1(t)+J_2(t).
\end{split}
\end{equation}

$\ $\\
\indent In order to calculate $\lim_{t\to0^+}J_1(t)$, we need the following:\\
\noindent{\bf Claim:}\indent For any $\Omega_1\Subset\Omega'$, there exists constant $C>0$ such that
$$\int_{\Omega_1}\Big|\nabla\Big(\frac{u_t-u}{t}\Big)\Big|^2d\rv\ls C$$
for all $t\in(0, t_0)$.
\begin{proof}[Proof of the Claim]
For each  $t\in (0,t_0)$, by combining equation \eqref{eq5.37} and semi-concavity of  $u_t$, we have
\begin{equation}\label{eq5.41}\begin{split}
\mathscr L_{\frac{u_t-u}{t}}&\ls  \Big(\frac{f\circ F_t-f}{t}+K|\nabla u_t|^2\Big)\cdot\rv\\
&=\Big(\frac{f\big(F_t(x), g \circ F_t(x) \big)-f(x ,g)}{t} + K|\nabla u_t|^2\Big)\cdot\rv
\end{split}\end{equation}
in sense of measure on $\Omega'$.
Noting that $\frac{\partial f}{\partial s}\ls0$, and that, for almost all $x\in \Omega'$,
 $$g\circ F_t(x) \gs {\rm Lip}^2u(F_t(x))\gs|\nabla u_t(x)|^2,
\qquad g(x)=|\nabla u(x)|^2,$$
(see  Lemma \ref{lem5.6}(i))   we have, for each $t\in(0,t_0)$,
  \begin{equation*}\begin{split}
\mathscr L_{\frac{u_t-u}{t}}&\ls \Big(\frac{f\big(F_t(x), |\nabla u_t(x)|^2 \big)-f(x,|\nabla u|^2) }{t} + K|\nabla u_t|^2\Big)\cdot\rv\\
&\ls \Big(2{\bf Lip}f\cdot\frac{|xF_t(x)|+\big| |\nabla u_t|^2  -|\nabla u|^2\big| }{t} + K|\nabla u_t|^2\Big)\cdot\rv\\
&\ls \Big(2{\bf Lip}f\cdot\frac{\big| |\nabla u_t|^2  -|\nabla u|^2\big| }{t}+ 2{\bf Lip}f\cdot|\nabla u_t|+ K|\nabla u_t|^2\Big)\cdot\rv\\
&\qquad\qquad {\rm because}\ |xF_t(x)|=t\cdot|\nabla u_t(x)|\ {\rm for\ a.e.\ }  x\in \Omega'  ({\rm see}\ \eqref{eq5.11})\\
&\ls \Big(C \cdot\frac{ \big| |\nabla u_t|^2  -|\nabla u|^2\big| }{t}+C\Big)\cdot\rv\\
&\qquad\qquad {\rm  because}\  |\nabla u_t(x)|\ls {\bf Lip}u\ {\rm (see\ Lemma \ref{lem5.6}(i))}
\end{split}\end{equation*}
\begin{equation*}\begin{split}
&=\Big(C \cdot\ip{\nabla\Big(\frac{u_t - u }{t}\Big)}{\nabla (u_t+u)}+C\Big)\cdot\rv\qquad\qquad \\
&\ls \Big(C \cdot\Big|\nabla\Big(\frac{u_t - u }{t}\Big)\Big|+C\Big)\cdot\rv
\end{split}\end{equation*}
in sense of measure on $\Omega'$.

Since $u_t-u\ls0$, according to Caccioppoli inequality, Theorem 7.1 in \cite{bm06} (or by choosing test function $-\varphi(u_t-u)/t$
 for some suitable nonnegative cut-off $\varphi$ on $\Omega'$), for any $\Omega_1\Subset\Omega'$, there exists positive constant $C$,
  independent of $t$,
   such that
\begin{equation}\label{eq5.42}
\int_{\Omega_1}\Big|\nabla \Big(\frac{u_t-u}{t}\Big)\Big|^2d\rv\ls C\int_{\Omega'}\Big(\frac{u_t-u}{t}\Big)^2d\rv+C.
\end{equation}

On the other hand, for almost all $x\in \Omega'$, according Eq. (2.6) in \cite{lv07}, we have
$$\frac{|u(x)- u_t(x)|}{t}\ls \frac{{\bf Lip}^2u}{2}.$$
Consequently,
$$\int_{\Omega_1}\Big(\frac{u_t-u}{t}\Big)^2d\rv\ls C.$$
The desired estimate follows from the combination of this and \eqref{eq5.42}. Now the proof of the Claim is finished.
\end{proof}

Let us continue the proof of Theorem \ref{bochner}.

Let $\Omega_1={\rm supp}\phi\Subset \Omega'.$
By combining  \eqref{eq5.32}, above \textbf{Claim} and   reflexivity of  $W^{1,2}(\Omega)$  (see Theorem 4.48 of \cite{c99}),
 we can conclude the following \textbf{facts}:\\
\indent (i)\indent \ $u_t$  converges (strongly) to $u$ in $W^{1,2}(\Omega_1)$ as $t\to0^+;$\\
\indent (ii)\indent there exists some sequence $t_j$ converging to $0^+$, such that $(u_{t_j}-u)/{t_j}$ converges weakly to
 $-|\nabla u|^2/2$
 in $W^{1,2}(\Omega_1)$, as $t_j\to0^+$.\\

Let us estimate $J_1(t)$.
For each $t\in (0,t_0)$,
\begin{equation*}\begin{split}
J_1(t)&=\int_{\Omega'}\phi\frac{f\big( F_t(x),|\nabla u_t(x)|^2 \big)-f(F_t(x),|\nabla u(x)|^2) }{t}d\rv\\
&=\int_{\Omega'}\phi\frac{f\big( F_t(x),|\nabla u_t(x)|^2 \big)-f(F_t(x),|\nabla u(x)|^2) }{|\nabla u_t|^2-|\nabla u|^2}\cdot
\ip{\nabla (u_t+u)}{\nabla\Big(\frac{u_t -u}{t}\Big)}d\rv\\
&=\int_{\Omega'}\phi\cdot\frac{\partial f}{\partial s}\big(F_t(x),s_t\big)\cdot
\ip{\nabla (u_t+u)}{\nabla\Big(\frac{u_t -u}{t}\Big)}d\rv
\end{split}\end{equation*}
for some   $s_t$ between $|\nabla u_t(x)|^2$ and $|\nabla u(x)|^2$.

Let $t_j$ be the sequence coming from above fact (ii).  According to Lemma \ref{lem5.6}(ii),
$$\lim_{t_j\to0^+}|\nabla u_{t_j}(x)|=|\nabla u(x)|$$
 for almost all $x\in \Omega'$, combining with the continuity of $\frac{\partial f}{\partial s} $, we get
$$\lim_{t_j\to0^+}\frac{\partial f}{\partial s}\big(F_{t_j}(x),s_{t_j}\big)
=\frac{\partial f}{\partial s}\big(x,|\nabla u(x)|^2\big).$$
 On the other hand, by the above facts (i), (ii) and the boundedness of
$$\Big|\frac{\partial f}{\partial s}\big(F_t(x),s_t\big)\Big|\ls {\bf Lip}f,$$
we have
\begin{equation}\label{eq5.43}\begin{split}
\lim_{t_j\to0^+}J_1(t_j)&=\int_{\Omega'}\phi\cdot\frac{\partial f}{\partial s}(x,|\nabla u|^2)\cdot
\ip{2\nabla u}{\nabla \Big(\frac{-|\nabla u|^2}{2}\Big)}d\rv\\
&=-\int_{\Omega'}\phi\cdot\frac{\partial f}{\partial s}(x,|\nabla u|^2)\cdot\ip{\nabla  u }{\nabla |\nabla u|^2}d\rv.
\end{split}\end{equation}

Let us calculate the limit $J_2(t_j)$, where the sequence comes from above fact (ii).

For each $t\in (0,t_0)$, if  $x\in \Omega'\backslash \mathcal N$ and $u_t$ is differentiable at point $x$, by Lemma \ref{lem5.2}, we have
\begin{equation*}\begin{split}
f(F_t(x),|\nabla u(x)|^2)&-f(x, |\nabla u(x)|^2)\\&=|xF_t(x)|\ip{\nabla_1f(x, |\nabla u(x)|^2)}{\uparrow_x^{F_t(x)}}+o(|xF_t(x)|)\\
&=-t\cdot\ip{\nabla_1 f(x, |\nabla u(x)|^2)}{\nabla u_t(x)}+o(|xF_t(x)|)
\end{split}\end{equation*}
where $\nabla_1f(x,s)$ means the differential of function $f(\cdot, s)$ at point $x$ (see eqution \eqref{eq2.16}).
For the sequence $t_j$, the combination of this, equation \eqref{eq5.11} and Lemma \ref{lem5.6}(ii)
 $$\lim_{t_j\to0^+} \nabla u_{t_j}(x) =  \nabla u (x)$$
  implies that
$$\lim_{t_j\to 0^+}\frac{f(F_{t_j}(x),|\nabla u(x)|^2)-f(x, |\nabla u(x)|^2)}{t_j}=-\ip{\nabla_1 f(x, |\nabla u(x)|^2)}{\nabla u(x)} $$
for almost all $x\in \Omega'$.
Note that
$$\Big|\frac{f(F_{t_j}(x),|\nabla u(x)|^2)-f(x, |\nabla u(x)|^2)}{t_j}\Big|\ls {\bf Lip}f\cdot \frac{|xF_{t_j}(x)|}{t_j}\ls
  {\bf Lip}f\cdot{\bf Lip}u $$
for almost everywhere in $\Omega'$. Dominated Convergence Theorem concludes that
\begin{equation*}\begin{split}
\lim_{t_j\to0^+}J_2(t_j)&= \lim_{t_j\to0^+} \int_{\Omega'}\phi\frac{f(F_{t_j}(x),|\nabla u(x)|^2)-f(x, |\nabla u(x)|^2)}{t_j}d\rv\\
&=-\int_{\Omega'}\phi\ip{\nabla_1 f(x, |\nabla u(x)|^2)}{\nabla u(x)} d\rv.
\end{split}\end{equation*}

By combining with equation \eqref{eq5.40} and \eqref{eq5.43}, we have
\begin{equation}\label{eq5.44}\begin{split}
\lim_{t_j\to0^+}I_1(t_j)&\ls \lim_{t_j\to0^+}J_1(t_j)+\lim_{t_j\to0^+}J_2(t_j)\\
&= -\int_{\Omega'}\phi\cdot\ip{\nabla u}{\frac{\partial f}{\partial s}(x,|\nabla u|^2)\cdot\nabla |\nabla u|^2+\nabla_1 f(x, |\nabla u(x)|^2)} d\rv\\
&= -\int_{\Omega'}\phi\cdot\ip{\nabla u}{\nabla f(x,|\nabla u|^2)} d\rv.
\end{split}\end{equation}

$ \ $\\
\indent Let us calculate $\lim_{t_j\to0}I_2(t_j )$ for the sequence $t_j\to0^+ $ coming from the above fact (ii).

From Lemma \ref{lem5.6}(ii),
 $$\lim_{t_j\to0^+}|\nabla u_{t_j}(x)|^2= |\nabla u (x)|^2=g(x)$$
   at almost all $x\in\Omega'$. Combining with the Lipschitz continuity of $f(x,s)$ and $\mathcal A\gs n/2$ for sufficiently small $t$, we have
$$\lim_{t_j\to0^+}\frac{f^2\big(F_{t_j}(x),|\nabla u_{t_j}|^2\big)-f^2(x,g(x))}{\mathcal A} =0$$
at almost all $x\in\Omega'$. On the other hand, using that  $\mathcal A\gs n/2$  again (when $t$ is sufficiently small) and that $f$ is bounded, we have
$$\Big| \frac{f^2\big(F_{t_j}(x),|\nabla u_{t_j}|^2\big)-f^2(x,g(x))}{\mathcal A}\Big|\ls C,\quad {\rm for\ almost\ all}\ \ x\in\Omega',\quad j=1,2,\cdots,$$
for some constant $C$.
Dominated Convergence Theorem concludes that
\begin{equation}
\label{eq5.45}\lim_{t_j\to0^+}I_2(t_j)=\lim_{t_j\to0^+}\int_{\Omega'}\frac{-f^2\big(F_{t_j}(x),|\nabla u_{t_j}|^2\big)
+f^2\big(x,g(x)\big)}{\mathcal A}d\rv=0 .
\end{equation}

$ \ $\\
\indent Let us calculate $\lim_{t_j\to0}I_3(t_j)$ for the sequence $t_j$ coming from above fact (ii).

According to Lemma \ref{lem5.6} (i) and (ii), we get
 $$|\nabla u_{t_j}|\ls {\bf Lip}u\quad {\rm  and}\quad \lim_{t_j\to0^+}|\nabla u_{t_j}|=|\nabla u|.$$
By combining with the boundedness of $D$ and $f$,  and applying Dominated Convergence Theorem, we  conclude  that
$$\lim_{t_j\to0^+}I_3(t_j)=\int_{\Omega'}\phi \frac{f^2-nK|\nabla u|^2}{\mathcal A} d\rv=\int_{\Omega'}\phi\frac{f^2\big(x,g(x)\big)-nK|\nabla u|^2}{n} d\rv.$$
By the fact that $$g(x)={\rm Lip} u=|\nabla u|$$  for almost everywhere in $\Omega'$, we get
\begin{equation}
\label{eq5.46}\lim_{t_j\to0^+}I_3(t_j)= \int_{\Omega'}\phi\big(\frac{f^2(x,|\nabla u|^2)}{n}- K|\nabla u|^2\big)d\rv.
\end{equation}

By applying above \textbf{Claim} again,
$$ \frac{u_{t_j}-u}{t_j}\longrightarrow  -\frac{|\nabla u|^2}{2} \quad {\rm weakly\ in}\ \ W^{1,2}(\Omega_1), $$
  as $t_j\to0$.
 By combining the definition of $\mathscr L_{u_{t_j}-u}$, (see  the first paragraph of Section 4.1.) we have
\begin{equation}\label{eq5.47}
 \lim_{t_j\to0^+}\frac{1}{t_j}\mathscr L_{u_{t_j}-u}(\phi) =-\lim_{t_j\to0^+}\int_{\Omega'}\ip{\nabla \phi}{\nabla\Big(\frac{u_{t_j}-u}{t_j}\Big)}
= \frac{1}{2}\int_{\Omega'}\ip{\nabla \phi}{\nabla |\nabla u|^2}d\rv.
 \end{equation}

The combination of equations \eqref{eq5.39} and  \eqref{eq5.44}--\eqref{eq5.47} shows that, for any $\phi\in Lip_0(\Omega')$,
\begin{equation*}\begin{split}
\frac{1}{2}\int_{\Omega'}&\ip{\nabla \phi}{\nabla |\nabla u|^2}d\rv \\&\ls  -\int_{\Omega'}\phi\Big(\ip{\nabla u}{\nabla f(x,|\nabla u|^2)} +\frac{f^2(x,|\nabla u|^2)}{n}-K|\nabla u|^2\Big)d\rv.
\end{split}\end{equation*}
 The desired result follows from this and the definition of $\mathscr L_{|\nabla u|^2}$.
 Now the proof of Theorem \ref{bochner} is completed.
\end{proof}

If $f(x,s)=f(x)$, then we can remove the technical condition that $|\nabla u|$ is lower semi-continuous at almost everywhere in $\Omega$.
 That is,
\begin{cor}\label{corbochner} Let $M$ be an $n$-dimensional Alexandrov space with Ricci curvature bounded from below by $-K$
 and $\Omega$ be a domain in $M$.  Assume function $f\in Lip(\Omega)$ and $u\in W^{1,2}(\Omega)$ satisfying
  $$\mathscr L_u= f \cdot\rv.$$

  Then we have $|\nabla u|^2\in W^{1,2}_{loc}(\Omega) $ and $|\nabla u|$ is lower semi-continuous on $\Omega.$ Consequencely,  we get
\begin{equation*}
\mathscr L_{|\nabla  u |^2} \gs  2\Big( \frac{f^2}{n}+\ip{\nabla u}{\nabla  f}-K|\nabla u|^2\Big)d\rv
\end{equation*}
in sense of measure on $\Omega$.
\end{cor}
\begin{proof}
At first, by Corollary \ref{cor5.5}, we conclude that $u\in Lip_{loc}(\Omega).$ Fix any $\Omega^*\Subset\Omega$.
Then $u\in Lip(\Omega^*) $ and $f(x,s)=f(x)$ satisfies the condition (a), (b) on $\Omega^*$ in Theorem \ref{bochner}.

Let us recall that in the proof of Theorem \ref{bochner}, the technique condition that $|\nabla u|$ is lower semi-continuous
 (with $\frac{\partial f}{\partial s}\ls0$) is only used to ensure the upper semi-continuity of $f=f\big(x,g(x)\big)$ on $\Omega^*$ so that  Proposition \ref{prop5.3} is applicable.
 Now, since $f(x)$ is Lipschitz, Proposition \ref{prop5.3} still works for equation
 $$\mathscr L_u=f\cdot\rv.$$

 Using the same notations as in the above proof (with $f(x,s)=f(x)$) of Theorem \ref{bochner}, we get the corresponding equation
\begin{equation*}
\mathscr L_{\frac{u_t-u}{t}} \ls  \Big(\frac{f\circ F_t-f}{t}+K|\nabla u_t|^2\Big)\cdot\rv
 =\Big(\frac{f\big(F_t(x)\big)-f(x)}{t} + K|\nabla u_t|^2\Big)\cdot\rv
 \end{equation*}
in sense of measure  on any $\Omega'\Subset\Omega^*$, (see equation \eqref{eq5.41} in the proof of the above \textbf{Claim}).
Then, we get, by \eqref{eq5.11}, $|xF_t(x)|=t|\nabla u_t(x)| $ at almost all $x\in\Omega^*$,
\begin{equation}\label{eq5.48}\begin{split}
\mathscr L_{\frac{u_t-u}{t}} &\ls  \Big({\bf Lip}f\frac{|x F_t(x)|}{t}+K|\nabla u_t|^2\Big)\cdot\rv
 =\Big( {\bf Lip}f\cdot|\nabla u_t|+K|\nabla u_t|^2\Big)\cdot\rv\\
& \ls C\cdot \rv \qquad {\rm( because }\quad |\nabla u_t|\ls {\bf Lip}u.)
\end{split}\end{equation}
in sense of measure on   $\Omega'$. Here and in the following, $C$ denotes various positive constants independent of $t$.

The same argument as in the proof of above \textbf{Claim}, we obtain that the $W^{1,2}$-norm of $\frac{u_t-u}{t}$
is uniformly bounded on any $\Omega_1\Subset\Omega'$. Hence there exists a suquence $t_j\to0^+$ such that
$$ \frac{u_{t_j}-u}{t_j}\longrightarrow  -\frac{|\nabla u|^2}{2} \quad {\rm weakly\ in}\ \ W^{1,2}(\Omega_1), $$
  as $t_j\to0^+.$ Combining with \eqref{eq5.48}, we have $|\nabla u|^2\in W^{1,2}_{loc}(\Omega_1)$ and
  $$\mathscr L_{|\nabla u|^2}\gs -2C\cdot\rv$$
in sense of measure on $\Omega_1$.

By setting
$$w=|\nabla u|^2+2C,$$
we have $w\gs 2C$ and
$$\mathscr L_w\gs-2C\cdot\rv\gs -w\cdot\rv.$$
Consider the product space $M\times \R$ (with directly product metric) and  the function $v(x,t): \Omega'\times \R\to\R$ as
 $$v(x,t):=w(x)\cdot e^t.$$
 Then $v$ satisfies $\mathscr L_v\gs0 $ in $\Omega_1\times \R$. Hence it has a lower semi-continuous representative
   (see Theorem 5.1 in \cite{km02}). Therefore, $w$ is lower semi-continuous on  $\Omega_1$. So does $|\nabla u|$.

 Because of the arbitrariness of $\Omega_1\Subset\Omega'\Subset\Omega^*\Subset\Omega$, we obtain that $|\nabla u|^2\in W^{1,2}_{loc}(\Omega) $
 and $|\nabla u|$ is lower semi-continuous on $\Omega.$

 It is easy to check that $f(x,s)=f(x)$ satisfies the condition (a), (b) on $\Omega$ in Theorem \ref{bochner}
  (since $f$ is Lipschitz and $\partial f/\partial s=0$.). We can apply  Theorem \ref{bochner} to equation
 $$\mathscr L_u=f\cdot\rv $$
 and conclude the last assertion of the corollary.
\end{proof}

As a direct application of the Bochner type formula, we have the following Lichnerowicz estimate, which was earlier
 obtained by Lott--Villani in \cite{lv07}   by  a different method. Further applications have been given in \cite{qzz}.
\begin{cor}\label{cor5.9} Let $M$ be an $n$-dimensional Alexandrov space with Ricci curvature bounded below by a
 positive constant $n-1$. Then we have
  $$\int_M|\nabla u|^2d\rv\gs n \int_Mu^2d\rv$$ for all $u\in W^{1,2}(M)$ with $\int_Mud\rv=0.$
\end{cor}
\begin{proof}
Let $u$ be a first eigenfunction and $\lambda_1$ be the first eigenvalue. It is clear that $\lambda_1\gs0$
 and $u(x)e^{\sqrt{\lambda_1}t}$ is a harmonic function on $M\times\R.$
  According to  Corollary \ref{cor5.5}, we know that $u$ is locally Lipschitz continuous.

   (Generalized) Bonnet--Myers' theorem implies that $M$ is compact (see  Corollary 2.6 in \cite{s06-2}).
 By using the Bochner type formula Corollary \ref{corbochner} to equation
  $$\mathscr L_u=-\lambda_1u,$$
  and choosing test function $\phi=1$, we get the  desired estimate immediately.
\end{proof}
\section{Gradient estimates for harmonic functions}
Let $\Omega$ be a bounded domain of an $n$-dimensional Alexandrov space with Ricci curvature $\gs-K$ and $K\gs0$.

 In the section, we always assume that $u$ is a positive harmonic function on $\Omega$.
According to Corollary \ref{corbochner}, we know that
 $|\nabla u|$ is lower semi-continuous  in $\Omega $ and  $|\nabla u|^2\in W^{1,2}_{loc}(\Omega).$

\begin{rem} In the previous version of this paper, by using some complicated pointwise $C^1$-estimate of elliptic equation (see, for example, \cite{c89,lyv00}),
 we can  actually show that $|\nabla u|$ is continuous at almost all in $\Omega$. Nevertheless, in this new version, we avoid using this continuity result.
\end{rem}

Now, let us prove the following  estimate.
\begin{lem}\label{lem6.2}Let $M$ be an $n$-dimensional Alexandrov space with Ricci curvature $\gs -K$ and $K\gs0$.
 Suppose that  $u$ is a positive harmonic function in $B_p(2R)$.
Then we have
\begin{equation}\label{eq6.1}
\|Q\|_{L^s(B_p(R))}\ls \Big(2nK+
\frac{8ns}{R^2}\Big)\cdot \Big( \mathrm{vol}\big(B_p(2R)\big)\Big)^{1/s}
\end{equation}
for $s\gs 2n+4$, where $Q=|\nabla \log u|^2$.
\end{lem}
\begin{proof}Since $u>0$ in $B_p(2R)$,  setting $v=\log u$, then we have
$$\mathscr L_v=-|\nabla v|^2\cdot\rv=-Q\cdot\rv.$$

For simplicity, we  denote $B_p(2R)$ by $B_{2R}$.

Let $\psi(x)$ be a nonnegative Lipschitz function with support in $B_{2R}$.  By choosing test function $\psi^{2s}Q^{s-2}$
 and using the Bochner type formula (\ref{eq5.36}) to $v$ (with function $f(x,s)=-s$, which satisfies the condition (a) and (b)
  in Theorem \ref{bochner}), we get
\begin{align}\label{eq6.2}
-\int_{B_{2R}}&\ip{\nabla Q}{\nabla(\psi^{2s}Q^{s-2})}d\rv\\ \nonumber& \gs\frac{2}{n}\int_{B_{2R}}\psi^{2s}Q^sd\rv
-2\int_{B_{2R}}\psi^{2s}Q^{s-2}\ip{\nabla v}{\nabla Q}d\rv\\ \nonumber & \quad-2K\int_{B_{2R}}\psi^{2s}Q^{s-1}d\rv.
\end{align}
Hence we have
\begin{align}\label{eq6.3}
\frac{2}{n}&\int_{B_{2R}}\psi^{2s}Q^sd\rv  -2K\int_{B_{2R}}\psi^{2s}Q^{s-1}d\rv\\ \nonumber\ls&
-2s\int_{B_{2R}}\psi^{2s-1} Q^{s-2}\ip{\nabla Q}{\nabla\psi}d\rv\\ \nonumber & -(s-2)\int_{B_{2R}}\psi^{2s}Q^{s-3}|\nabla Q|^2d\rv
  +2 \int_{B_{2R}}\psi^{2s}Q^{s-2}\ip{\nabla v}{\nabla Q}d\rv\\ \nonumber=&s\cdot I_1-(s-2)\cdot I_2+I_3.
\end{align}
We now estimate $I_1,$ $I_2$ and $I_3.$ By Cauchy--Schwarz inequality, we have
$$I_1\ls \frac{1}{2}\int_{B_{2R}}\psi^{2s}Q^{s-3}|\nabla Q|^2d\rv+2\int_{B_{2R}}Q^{s-1}\psi^{2s-2}|\nabla \psi|^2d\rv.$$
and
$$I_3\ls n\int_{B_{2R}}\psi^{2s}Q^{s-3}|\nabla Q|^2d\rv+\frac{1}{n}\int_{B_{2R}}\psi^{2s}Q^{s}d\rv,$$
By combining  with (\ref{eq6.3}), we  obtain
\begin{align*}
\frac{1}{n}\int_{B_{2R}}&\psi^{2s}Q^sd\rv  -2K\int_{B_{2R}}\psi^{2s}Q^{s-1}d\rv\\&\ls \Big(\frac{s}{2}-(s-2)+n\Big) \cdot I_2 +2s\int_{B_{2R}}Q^{s-1}\psi^{2s-2}|\nabla \psi|^2d\rv.
\end{align*}
If we choose $s\gs 2n+4$, then we can drop the first term in RHS.

Set $$\tau=\Big(\int_{B_{2R}}\psi^{2s}Q^sd\rv\Big)^\frac{1}{s}.$$
Then by $K\gs0$ and  H\"older inequality, we have
$$\frac{\tau^s}{n}\ls 2K\Big(\int_{B_{2R}}\psi^{2s}d\rv\Big)^{1/s}\cdot\tau^{s-1}
+2s\Big(\int_{B_{2R}}|\nabla\psi|^{2s}d\rv\Big)^{1/s}\cdot\tau^{s-1} .$$
Therefore,  when we choose $\psi$ such that $\psi=1$ on $B_{R}$, $\psi=0$ outside $B_{2R}$ and $|\nabla\psi|\ls 2/R$,
 we get the desired estimate \eqref{eq6.1}.
\end{proof}

\begin{cor}
 Let $u$ be a positive harmonic function on an $n$--dimensional complete noncompact Alexandrov space with Ricci curvature
  $\gs -K$ and $K\gs0$. Then we have $$|\nabla\log u|\ls C_{n,K}.$$
\end{cor}
\begin{proof}
Without loss of generality, we may assume $K>0$. From Lemma \ref{lem6.2} above, setting $s= R^2$ for $R$ large enough,  we have
$$\big\||\nabla \log u|^2\big\|_{L^{R^2}(B_p(R))}\ls \Big(2nK+
8n\Big)\cdot \Big(\rv\big(B_p(2R)\big)\Big)^{\frac{1}{R^2}}.$$
According to Bishop--Gromov volume comparison theorem (see \cite{ks07} or \cite{s06-2}), we have
$$\rv\big(B_p(2R)\big)\ls H^n\big(B_o(2R)\subset \mathbb M^n_{K/(n-1)}\big)\ls C_1e^{C_2R},$$
where constants both $C_1$ and $C_2$ depend  only on $n$ and $K$.
Combining above two inequalities, we get
$$\big\||\nabla \log u|^2\big\|_{L^{R^2}(B_p(R))}\ls C_{n,K}\cdot C_1^{1/R^2}e^{C_2/R}.$$
Letting $R\to\infty$, we obtain the desired result.
\end{proof}

In order to  get a local $L^\infty$ estimate of $|\nabla \log u|$, let us recall the local version of Sobolev inequality.

Let $D=D(\Omega)$ be a doubling constant on $\Omega$, i.e., we have
 $$\rv(B_p(2R))\ls 2^{D}\cdot\rv(B_p(R))$$
 for all balls $B_p(2R)\Subset \Omega$. According to Bishop--Gromov volume comparison, it is known that
  if $M$ has nonnegative Ricci curvature, the constant $D$ can be chosen $D=n$.
  For the general case, if $M$ has $Ric\gs -K$ for some $K\gs0$, then the constant can be chosen to depend on $n$ and $\sqrt{K}\cdot {\rm diam}(\Omega)$, where ${\rm diam}(\Omega)$ is the diameter of $\Omega$. Here and in the following,
   without loss of generality, we always assume that the doubling constant $D\gs3.$

Let $C_P=C_P(\Omega)$ be a (weak) Poincar\'e constant on $\Omega$, i.e., we have
$$\int_{B_p(R)}|\varphi-\varphi_{p,R}|^2d\rv\ls C_P\cdot R^2\cdot\int_{B_p(2R)}|\nabla \varphi|^2d\rv $$
 for all balls $B_p(2R)\Subset \Omega$ and $\varphi\in W^{1,2}(\Omega)$, where $\varphi_{p,R}=\fint_{B_p(R)}\varphi d\rv$.
By Bishop--Gromov volume comparison and Cheeger--Colding's segment inequality,
 it is known that if $M$ has nonnegative Ricci curvature,
the constant $C_P$ can be chosen to depend only on $n$. For the general case,
if $M$ has $Ric\gs-K$ for some $K\gs0$, then the constant can be chosen to depend on $n$ and $\sqrt{K}\cdot {\rm diam}(\Omega)$.

It is well known that the doubling property and a Poincar\'e inequality imply a Sobolev inequality in length spaces
 (see, for example \cite{sa02, s96,hk00}). Explicitly, there exists a constant $C_S=C_S(\Omega)$,
  which depends only on $D$ and $C_P$, such that
\begin{equation} \label{eq6.4}
\Big(\int_{B_p(R)}|\varphi|^{\frac{2D}{D-2}}d\rv\Big)^{\frac{D-2}{D}}
\ls C_S\cdot \frac{R^2}{\rv(B_p(R))^{2/D}}\cdot\int_{B_p(2R)}\big(|\nabla \varphi|^2+R^{-2}\cdot \varphi^2\big)d\rv
\end{equation}
for all balls $B_p(2R)\Subset\Omega$ and $\varphi\in W_0^{1,2}(\Omega)$.

Now by combining Lemma \ref{lem6.2} and  the standard Nash--Moser iteration  method, we can get the following local estimate.
\begin{thm}\label{thm6.4} Let $M$ be an $n$-dimensional Alexandrov space with $Ric\gs-K$, for some $K\gs0$.
 Suppose that   $\Omega\subset M$ is a bounded open domain. Then there exists a constant $C=C\big(n,\sqrt K{\rm diam}(\Omega)\big)$
  such that
$$\max_{x\in B_p(R/2)}|\nabla \log u|\ls C\cdot(\sqrt K+\frac{1}{R})$$for all positive harmonic function $u$ on $\Omega$ and $B_p(2R)\Subset\Omega$.

If $K=0, $ the constant $C=C(n)$.
\end{thm}
\begin{proof}
Let $v$ and $Q$ be as in the  above Lemma \ref{lem6.2}. We choose test function $\psi^2 Q^{s-1},$
where $\psi$ has support in ball $B_R:=B_p(R)$. By using the Bochner type formula (\ref{eq5.36}) to function $v$ (with function $f(x,s)=-s$), we have
\begin{align}\label{eq6.5}
\frac{2}{n}\int_{B_R} \psi^2Q^{s+1}d\rv\ls &2\int_{B_R}\psi^2Q^{s-1}\ip{\nabla v}{\nabla Q}d\rv
-2\int_{B_R}\psi Q^{s-1}\ip{\nabla \psi}{\nabla Q}d\rv\\ \nonumber&-(s-1)\int_{B_R} \psi^2Q^{s-2}|\nabla Q|^2d\rv
+2K\int_{B_R}\psi^2Q^sd\rv.
\end{align}
Note that
$$2\int_{B_R}\psi^2Q^{s-1}\ip{\nabla v}{\nabla Q}d\rv\ls \frac{n}{2}\int_{B_R}\psi^2Q^{s-2}|\nabla Q|^2d\rv
+\frac{2}{n}\int_{B_R}\psi^2Q^{s}|\nabla v|^2d\rv$$
and
$$-2\int_{B_R}\psi Q^{s-1}\ip{\nabla \psi}{\nabla Q}d\rv\ls \int_{B_R}\psi^2Q^{s-2}|\nabla Q|^2d\rv+\int_{B_R} Q^{s}|\nabla \psi|^2d\rv.$$
By combining with \eqref{eq6.5}, we get
\begin{equation}\label{eq6.6}
(s-2-\frac{n}{2})\int_{B_R}\psi^2Q^{s-2}|\nabla Q|^2d\rv\ls \int_{B_R} Q^{s}|\nabla \psi|^2d\rv+2K\int_{B_R}\psi^2Q^sd\rv.
\end{equation}
Taking $s\gs 2n+4$, then $s-2-n/2\gs s/2$. Let $\frac{R}{2}\ls \rho<\rho'\ls R$.  Choose $\psi$ such that $\psi(x)=1$ if $x\in B_p(\rho)$, $\psi(x)=0$ if $x\not\in B_p(\rho')$ and $|\nabla \psi|\ls 2/(\rho'-\rho)$, Then by (\ref{eq6.4}) and \eqref{eq6.6}, we have $$\Big(\int_{B_p(\rho)}Q^{s\theta}d\rv\Big)^{1/\theta}\ls \Big(\mathscr A\cdot(2sK+\frac{1}{R^2}+\frac{8s}{(\rho'-\rho)^2})\Big)\cdot\int_{B_p(\rho')}Q^sd\rv,$$
where $\theta=D/(D-2)$ and
\begin{equation}\label{eq6.7}
\mathscr A=C_S\cdot\frac{R^2}{\rv(B_p(R))^{2/D}}.
\end{equation}
Let $l_0$ be an integer such that $\theta^{l_0}\gs 2n+4$.   Taking $s_l=\theta^l$, $\rho_l=R(1/2+1/2^l)$ with $l\gs l_0$,
 we have
 $$\log J_{l+1}-\log J_l\ls \frac{1}{\theta^l}\cdot \log\Big(\mathscr A\cdot(2\theta^lK+\frac{1}{R^2}+\frac{2\cdot\theta^l\cdot 4^{l+2}}{R^2})\Big),$$
where
 $$J_l=\Big(\int_{B_p(\rho_l)}Q^{s_l}d\rv\Big)^{1/s_l}=\|Q\|_{L^{\theta^l}(B_p(\rho_l))}.$$
Hence, we have \begin{align*}\log J_\infty-\log J_{l_0}&\ls \log \mathscr A\cdot\sum_{l=l_0}^\infty\theta^{-l}+\sum_{l=l_0}^\infty\theta^{-l}\cdot\log(2\theta^lK+\frac{33(4\theta)^l}{R^2})\\ \nonumber&\ls \theta^{-l_0}\cdot \log \mathscr A^{D/2}+\sum_{l=l_0}^\infty\theta^{-l}\cdot \big(l\cdot\log(4\theta)+\log(K+\frac{33}{R^2})\big).
\end{align*}
On the other hand, by Lemma \ref{lem6.2}, we have
$$\log J_{l_0}\ls \log(2nK+\frac{8n\theta^{l_0}}{R^2})+\theta^{-l_0}\log \rv(B_p(2R)).$$
Hence, we obtain
\begin{align}\label{eq6.8}
\log J_\infty&\ls\log(2nK+\frac{8n\theta^{l_0}}{R^2})+\theta^{-l_0}\Big(\log \rv(B_p(2R))+  \log \mathscr A^{D/2}\Big)\\
\nonumber&\quad +\log(4\theta)\cdot\sum_{l=l_0}^\infty l\cdot\theta^{-l}+\log(K+\frac{33}{R^2})\sum_{l=l_0}^\infty\theta^{-l}.
\end{align}
From (\ref{eq6.7}) and (\ref{eq6.8}), we have
\begin{align*}\log J_\infty&\ls\log(2nK+\frac{8n\theta^{l_0}}{R^2})+\frac{D}{2}\theta^{-l_0}\log\Big(4C_SR^2\Big)\\
\nonumber&\quad +\log(4\theta)\cdot\sum_{l=l_0}^\infty l\cdot\theta^{-l}+\log(K+\frac{33}{R^2})\sum_{l=l_0}^\infty\theta^{-l}\\
&\ls \log(2nK+\frac{8n\theta^{l_0}}{R^2})+\frac{D}{2}\theta^{-l_0}\log\Big(4C_S(KR^2+33)\Big)+C(\theta,l_0).
\end{align*}
Taking $l_0$ such that $\theta^{l_0}\ls 8n$, we get
$$\log J_\infty\ls \log(2nK+\frac{64n^2}{R^2})+C(n, \sqrt{K}{\rm diam}(\Omega)).$$
This gives the desired result.
\end{proof}

The gradient estimate shows that any sublinear
growth harmonic function on an Alexandrdov space with nonnegative Ricci curvture must be a constant. Explicitly, we have the following.
\begin{cor}Let $M$ be an $n$-dimensional complete non-compact Alexandrov space with nonnegative Ricci curvature.
 Assume that $u$ is harmonic function on $M$. If $$\lim_{r\to\infty}\frac{\sup_{x\in B_p(r)}|u(x)|}{r}=0$$ for some
  $p\in M$, then $u$ is a constant.
\end{cor}
\begin{proof} Clearly, for any $q\in M$, we still have  $$\lim_{r\to\infty}\frac{\sup_{x\in B_q(r)}|u(x)|}{r}=0.$$
Let $\overline{u_r}=\sup_{x\in B_q(r)}|u(x)|$. Then $2\overline{u_r}-u$ is a positive harmonic on $B_q(r)$,
 unless $u$ is identically zero. By Theorem \ref{thm6.4}, we have
 $$|\nabla u(q)|\ls C(n)\frac{\sup_{x\in B_q(r)}(2\overline{u_r}-u)}{r}\ls C(n)\frac{3\overline{u_r}}{r}.$$
 Letting $r\to\infty$, we get $|\nabla u(q)|=0$. This completes the proof.
\end{proof}

As another application of the gradient estimate, we have the following  mean value property, by using
 Cheeger--Colding--Minicozzi's argument in \cite{ccm95}. In smooth case, it was first proved by Peter Li in \cite{li86}
  via a parabolic method.
\begin{cor}
Let $M$ be an $n$-dimensional complete non-compact Alexandrov space with nonnegative Ricci curvature. Suppose that $u$ is
 a bounded superharmonic function on $M$. Then $$\lim_{r\to\infty}\fint_{\partial B_p(r)}ud\rv=\inf u.$$
\end{cor}
\begin{proof}Without loss of generality, we can assume that $\inf u=0$.

Fix any $\epsilon>0,$ Then there exists some $R(\epsilon)$ such that $\inf_{B_p(R)}u<\epsilon$  for all $R>R(\epsilon).$
For any $R>4R(\epsilon), $ we consider the harmonic function $h_R$ on $B_p(R)$ with boundary value $h_R=u$ on $\partial B_p(R)$.
 By maximum principle and the gradient estimate of $h_R$, we have $$\sup_{B_p(R/2)}h_R\ls C(n)\cdot \inf_{B_p(R/2)}h_R<C(n)\cdot\epsilon.$$

On the other hand, from the monotonicity of $r^{1-n}\cdot\int_{\partial B_p(r)}h_Rd\rv$ on $(0,R)$,
(see the proof of Proposition \ref{prop4.4}), we have $$\int_{\partial B_p(R)}h_Rd\rv\ls C(n) \int_{\partial B_p(R/2)}h_Rd\rv.$$
Then we get $$\int_{\partial B_p(R)}ud\rv=\int_{\partial B_p(R)}h_Rd\rv\ls C(n)\cdot\epsilon\cdot \rv(\partial B_p(R/2)).$$
Therefore, the desired result   follows from Bishop--Gromov volume comparison and the arbitrariness of $\epsilon.$
\end{proof}

\end{document}